\newif\ifdraftmode
\draftmodefalse
\documentclass[10pt]{article}

\ifdraftmode
\fi

\newcommand\thisShorttitle{Sample-based almost-sure quasi-optimal approximation in reproducing kernel Hilbert spaces}
\newcommand\thisTitle{\thisShorttitle}
\newcommand\thisSubject{math.NA, cs.LG}
\newcommand\thisAuthor{Philipp Trunschke, Anthony Nouy}
\newcommand\thisKeywords{approximation, reproducing kernel Hilbert spaces, optimal sampling, volume sampling, determinantal point processes, heuristics, greedy algorithm, submodular set functions}

\usepackage{arxiv}

\usepackage[utf8]{inputenc}      
\usepackage[T1]{fontenc}         
\usepackage[dvipsnames]{xcolor}  
\usepackage{float}
\usepackage{hyperref}            
\usepackage{url}                 
\usepackage{doi}                 
\usepackage{microtype}           
\usepackage{graphicx}
\usepackage{ulem}

\hypersetup{
    breaklinks,
    raiselinks=true,
    colorlinks,
    citecolor=gray,
    filecolor=gray,
    linkcolor=gray,
    urlcolor=gray,
    bookmarksopenlevel=1,    
    bookmarksopen=true,      
    bookmarksnumbered=true,  
    bookmarkstype=toc,       
    pdfdisplaydoctitle=true, 
    pdfstartview=FitH,       
    pdfpagemode=UseOutlines, 
    pdfpagelayout=OneColumn, 
    pdftitle={\thisTitle},
    pdfsubject={\thisSubject},
    pdfauthor={\thisAuthor},
    pdfkeywords={\thisKeywords},
}
\usepackage{lipsum}              

\usepackage{enumitem}            
\usepackage[ruled,linesnumbered]{algorithm2e}  
\usepackage{caption}             
\usepackage{subcaption}          

\usepackage{csquotes}  
\usepackage[
    style=numeric-comp,
    sorting=nyt, 
    sortcites=true,
    giveninits=true,
    maxbibnames=99,
    backend=biber,
    maxalphanames=5,
    backref=true,
    natbib=true,     
    url=false,
    date=year, 
    abbreviate=false,
]{biblatex}

\DeclareFieldFormat{eprint:hal}{%
    \mkbibacro{HAL}\addcolon\space
    \ifhyperref
        {\href{https://hal.archives-ouvertes.fr/#1}{%
        \nolinkurl{#1}
        \iffieldundef{eprintclass}
            {}
            {\addspace\UrlFont{\mkbibbrackets{\thefield{eprintclass}}}}}}
        {%
        \nolinkurl{#1}
        \iffieldundef{eprintclass}
            {}
            {\addspace\UrlFont{\mkbibbrackets{\thefield{eprintclass}}}}}}
\DeclareFieldAlias{eprint:HAL}{eprint:hal}

\DeclareSourcemap{
  \maps[datatype=bibtex]{
    \map{
      \step[fieldset=issn, null]
    }
  }
}
\renewbibmacro{in:}{}  

\def\keywordname{{\bfseries \emph{Key words.}}}%

\def\mscclassname{{\bfseries \emph{AMS subject classifications.}}}%
\def\mscclasses#1{\par\addvspace\medskipamount{\rightskip=0pt plus1cm
\def\and{\ifhmode\unskip\nobreak\fi\ $\cdot$
}\noindent\mscclassname\enspace\ignorespaces#1\par}}
\def\codename{{\bfseries \emph{Code.}}}%
\def\code#1{\par\addvspace\medskipamount{\rightskip=0pt plus1cm\noindent\codename\enspace\ignorespaces\url{#1}\par}}

\ifdraftmode
    \definecolor{amaranth}{rgb}{0.9, 0.17, 0.31}%
    \definecolor{americanrose}{rgb}{1.0, 0.01, 0.24}
    \colorlet{alertcolor}{amaranth}
    \colorlet{notecolor}{MidnightBlue}

    \makeatletter
    \newcommand{\todo}[1]{\marginpar{\tiny\color{alertcolor}#1}\@latex@warning{#1}\xspace}
    \makeatother
    \newcommand{\note}[1]{\marginpar{\tiny\color{notecolor}#1}}
    
    \usepackage[notref,notcite]{showkeys}
    
    \definecolor{bleudefrance}{rgb}{0.19, 0.55, 0.91}
    \newcommand{\numRevisions}{2}
    \newcommand{\revision}[2][0]{%
    \begingroup%
        \newcount\colorRatio%
        \colorRatio=\numexpr(100*(#1+1))/\numRevisions\relax%
        \colorlet{revisionColor}{bleudefrance!\the\colorRatio!black}\color{revisionColor}#2%
    \endgroup}

\else
    \newcommand{\todo}[1]{}
    \newcommand{\note}[1]{}
\fi

\usepackage{pgf,tikz}
\usetikzlibrary{positioning}
\usetikzlibrary{calc}
\usetikzlibrary{intersections}
\usetikzlibrary{decorations.pathreplacing}  

\tikzset{core/.style={inner sep=0pt}}
\tikzset{contraction/.style={line width=0.75}}
\tikzset{contractionDots/.style={contraction, dotted}}

\usepackage{amsmath}             
\usepackage{amsfonts}            
\usepackage{amssymb}             
\usepackage{amsthm}              
\usepackage{thmtools}      
\usepackage{thm-restate}
\usepackage[fixamsmath,disallowspaces]{mathtools} 
\mathtoolsset{showonlyrefs=true}
\usepackage{cancel}
\usepackage{newtxtext}
\usepackage[varvw]{newtxmath}  

\colorlet{dimgray}{black!35!white}
\colorlet{lightgray}{dimgray!35!white}
\declaretheoremstyle[bodyfont=\itshape, mdframed={backgroundcolor=lightgray, linecolor=dimgray, linewidth=0.75pt, innertopmargin=1.5ex}]{claim}
\declaretheorem[style=claim]{theorem}
\declaretheorem[style=claim, name=Theorem, numbered=no]{theorem*}
\declaretheorem[style=claim, numberlike=theorem]{lemma}
\declaretheorem[style=claim, numberlike=theorem]{proposition}
\declaretheorem[style=claim, numberlike=theorem]{corollary}

\declaretheoremstyle[mdframed={backgroundcolor=lightgray, linecolor=dimgray, linewidth=0.75pt, innertopmargin=1.5ex}]{definition}
\declaretheorem[style=definition, numberlike=theorem]{definition}
\declaretheoremstyle[bodyfont=\itshape, mdframed={backgroundcolor=white, linecolor=dimgray, linewidth=0.75pt, innertopmargin=1.5ex}]{remark}
\declaretheorem[style=remark, numberlike=theorem]{remark}
\declaretheoremstyle[mdframed={backgroundcolor=white, linecolor=dimgray, linewidth=0.75pt, innertopmargin=1.5ex}]{example}
\declaretheorem[style=example, numberlike=theorem]{example}

\newcommand{\indep}{\perp\kern-0.6em\perp}

\newcommand*{\mbb}[1]{\mathbb{#1}}
\newcommand*{\mcal}[1]{\mathcal{#1}}
\newcommand*{\mfrak}[1]{\mathfrak{#1}}
\newcommand*{\dd}{\ensuremath{\mathrm{d}}}
\newcommand*{\dx}[1][x]{\ensuremath{\,\dd{#1}}}

\let\inf\relax  
\DeclareMathOperator*{\inf}{inf\vphantom{\sup}}
\DeclareMathOperator*{\argmin}{arg\,min}
\DeclareMathOperator*{\argmax}{arg\,max}
\DeclareMathOperator*{\esssup}{ess\,sup}

\DeclarePairedDelimiter{\pars}{\ensuremath{(}}{\ensuremath{)}}
\DeclarePairedDelimiter{\bracs}{\ensuremath{[}}{\ensuremath{]}}
\DeclarePairedDelimiter{\braces}{\ensuremath{\{}}{\ensuremath{\}}}

\DeclarePairedDelimiter{\inner}{\langle}{\rangle}
\DeclarePairedDelimiter{\norm}{\|}{\|}
\DeclarePairedDelimiter{\abs}{\lvert}{\rvert}

\DeclarePairedDelimiter{\seminorm}{\vert}{\vert}

\makeatletter
\newcommand{\opnorm}{\@ifstar\@opnorms\@opnorm}
\newcommand{\@opnorms}[1]{%
  \left|\mkern-1.5mu\left|\mkern-1.5mu\left|
   #1
  \right|\mkern-1.5mu\right|\mkern-1.5mu\right|
}
\newcommand{\@opnorm}[2][]{%
  \mathopen{#1|\mkern-1.5mu#1|\mkern-1.5mu#1|}
  #2
  \mathclose{#1|\mkern-1.5mu#1|\mkern-1.5mu#1|}
}
\makeatother

\mathchardef\texthyphen="2D

\usepackage{dsfont}              

\let\oldbullet\bullet
\newlength{\raisebulletlen}
\setbox1=\hbox{$\bullet$}\setbox2=\hbox{\tiny$\bullet$}
\renewcommand\bullet{\raisebox{\raisebulletlen}{\,\tiny$\oldbullet$}\,}

\makeatletter
\newcommand*{\rom}[1]{\expandafter\@slowromancap\romannumeral #1@}
\makeatother

\DeclarePairedDelimiterX\Set[1]\{\}{%
  #1%
}

\def\multiset#1#2{\ensuremath{\left(\kern-.3em\left(\genfrac{}{}{0pt}{}{#1}{#2}\right)\kern-.3em\right)}}

\title{\thisTitle} 
\renewcommand{\undertitle}{}
\renewcommand{\shorttitle}{\thisShorttitle}

\date{}

\author{
\href{https://orcid.org/0000-0003-3953-9006}{\includegraphics[height=0.7em]{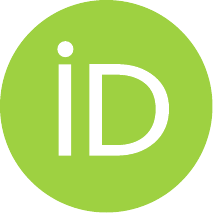}\hspace{1mm}\textcolor{black}{Nando Hegemann
}} \\
    Physikalisch-Technische Bundesanstalt \\
    Department 8.4 --- Mathematical Modelling and Data Analysis \\
    Abbestra\ss{}e 2--12\\
    10587 Berlin \\
    Germany \\
    \href{mailto:nando.hegemann@ptb.de}{\texttt{nando.hegemann@ptb.de}} \\
\And
\href{https://orcid.org/0000-0002-2149-2986}{\includegraphics[height=0.7em]{orcid.pdf}\hspace{1mm}\textcolor{black}{Anthony Nouy
}} \\
    Centrale Nantes \\
    Nantes Universit\'e \\
    Laboratoire de Math\'ematiques Jean Leray \\
    CNRS UMR 6629 \\
    France \\
    \href{mailto:anthony.nouy@ec-nantes.fr}{\texttt{anthony.nouy@ec-nantes.fr}} \\
\And
\href{https://orcid.org/0000-0002-2995-126X}{\includegraphics[height=0.7em]{orcid.pdf}\hspace{1mm}\textcolor{black}{Philipp Trunschke\thanks{Corresponding Author. 
}}} \\
    Centrale Nantes \\
    Nantes Universit\'e \\
    Laboratoire de Math\'ematiques Jean Leray \\
    CNRS UMR 6629 \\
    France \\
    \\
    Physikalisch-Technische Bundesanstalt \\
    Department 8.4 --- Mathematical Modelling and Data Analysis \\
    Abbestra\ss{}e 2--12\\
    10587 Berlin \\
    Germany \\
    \href{mailto:philipp.trunschke@ptb.de}{\texttt{philipp.trunschke@ptb.de}} \\
}

\begin{document}
\maketitle

\begin{abstract}
    This paper addresses the problem of approximating an unknown function from point evaluations.
    When obtaining these point evaluations is costly, minimising the required sample size becomes crucial, and it is unreasonable to reserve a sufficiently large test sample for estimating the approximation accuracy.
    Therefore, an approximation with a certified quasi-optimality factor is required.    
    This article shows that such an approximation can be obtained when the sought function lies in a \emph{reproducing kernel Hilbert space} (RKHS) and is to be approximated in a finite-dimensional linear subspace $\mcal{V}_d$.
    However, selecting the sample points to minimise the quasi-optimality factor requires optimising over an infinite set of points and computing exact inner products in RKHS, which is often infeasible in practice.
    Extending results from optimal sampling for $L^2$ approximation, the present paper proves that random points, drawn independently from the Christoffel sampling distribution associated with $\mcal{V}_d$, can yield a controllable quasi-optimality factor with high probability.
    Inspired by this result, a novel sampling scheme, coined subspace-informed volume sampling, is introduced and evaluated in numerical experiments, where it outperforms classical i.i.d.\ Christoffel sampling and continuous volume sampling.
    To reduce the size of such a random sample, an additional greedy subsampling scheme with provable suboptimality bounds is introduced.
    Our presentation is of independent interest to the inverse problems community, as it offers a simpler interpretation of the \emph{parametrised background data weak} (PBDW) method.
\end{abstract}

\vspace{-1em}
\mscclasses{41A25 \and 41A65 \and 68W25 \and 90C59}
\vspace{-1em}
\code{https://github.com/ptrunschke/almost_sure_least_squares}


\section{Introduction}
\label{sec:introduction}


This manuscript considers the problem of approximating a function in a finite-dimensional subspace of a \emph{reproducing kernel Hilbert space} (RKHS).
It presents a method for obtaining a set of points and computing a quasi-optimal approximation using the function's values at these points.

To make this concrete, let $\mathcal{V}$ be a RKHS with norm $\norm{\bullet}_{\mathcal{V}}$ and reproducing kernel $k : \mathcal{X}\times \mathcal{X} \to \mathbb{R}$ and consider the problem of approximating a function $u\in\mathcal{V}$ in a $d$-dimensional subspace $\mathcal{V}_d\subseteq \mathcal{V}$.
The best possible approximation is given by the $\mcal{V}$-orthogonal projection
$$
    P_{\mcal{V}_d}u := \argmin_{v\in\mcal{V}_d} \; \norm{u - v}_{\mcal{V}} ,
$$
which is, in general, impossible to compute.
This paper considers the computable, kernel-based approximation
$$
    P_{\mcal{V}_d}^{\boldsymbol{x}} u
    := \argmin_{v\in \mcal{V}_d} \; \norm{P_{\mcal{V}_{\boldsymbol{x}}} (u - v)}_{\mcal{V}} ,
$$
where $P_{\mcal{V}_{\boldsymbol{x}}}$ denotes the $\mcal{V}$-orthogonal projection onto the space $\mcal{V}_{\boldsymbol{x}} := \operatorname{span} \{k(x_i , \bullet) : 1\le i\le n\}$, and corresponds to the kernel interpolation operator at points $\boldsymbol{x} = \{x_1,\dots,x_n\} \in \mcal{X}^n$.
Notably, this approximation satisfies the subsequent quasi-optimality property.
\begin{theorem*}
    Let $u\in\mcal{V}$ and $\boldsymbol{x}\in\mcal{X}^n$.
    There exist (computable) constants
    $1\le\mu(\boldsymbol{x})$ and $0\le\tau(\boldsymbol{x})\le1$
    such that $P^{\boldsymbol{x}}_{\mathcal{V}_d} u$ is well-defined when $\mu(\boldsymbol{x}) < \infty$ and satisfies
    $$
        \norm{u - P_{\mcal{V}_d}^{\boldsymbol{x}} u}_{\mcal{V}}^2 
        \le (1 + \mu(\boldsymbol{x})^2 \tau(\boldsymbol{x})^2) \norm{u - P_{\mcal{V}_d} u}_{\mcal{V}}^2 .
    $$
    Moreover, it holds that $\tau(\boldsymbol{x}) \le \min\{(1 + \sqrt{d})(1 - \mu(\boldsymbol{x})^{-2}),1\}$.
\end{theorem*}

The preceding result is proven in Theorem~\ref{thm:error_bound_noiseless} and Lemmas~\ref{lem:mu_lmin} and~\ref{lem:mu_is_gramian_error}.
It guarantees that the error of $P_{\mcal{V}_d}^{\boldsymbol{x}} u$ is proportional to the best approximation error in the $\mcal{V}$-norm.
In this sense, the approximation $P_{\mcal{V}_d}^{\boldsymbol{x}} u$ is quasi-optimal in the $\mcal{V}$-norm.
Moreover, since the quasi-optimality factor depends mainly on the (computable) constant $\mu(\boldsymbol{x})$, it is natural to minimise this constant.
This is discussed in section~\ref{sec:error_control}.

To obtain a quasi-optimal complexity, we have to find a point set  $\boldsymbol{x}\in\mcal{X}^n$ satisfying $\mu(\boldsymbol{x})\le\mu_{\star}$ with $n\le Cd$.
Following~\cite{belhadji2020kernel}, we aim to do this via random sampling and provide sample size bounds in the subsequent version of Corollary~\ref{cor:ayoub}.

\begin{theorem*}
    Let $\Sigma$ be the integral operator associated to the kernel  $k$ of $\mathcal{V}$.
    Suppose that $\Sigma$ is compact and denote by $(\lambda_m)_{m\in\mathbb{N}}$ the sequence of eigenvalues of $\Sigma$.
    Then, if $\mathcal{V}_d$ is spanned by the first $d$ eigenfunctions of $\Sigma$, and $\boldsymbol{x}$ is drawn by continuous volume sampling, we can ensure $\mu(\boldsymbol{x}) \le \mu_\star$ if 
    \begin{align}
        n \gtrsim d^{1 + 1/(2s)}
        \text{\ \ if\ \ } \lambda_n \asymp n^{-2s}
        \qquad\text{and}\qquad
        n \gtrsim d 
        \text{\ \ if\ \ } \lambda_n \asymp \alpha^n
    \end{align}
    for $s\ge \tfrac{1}{2}$ (e.g., Sobolev spaces of finite smoothness) and $\alpha\in(0,1)$ (e.g., the Gaussian kernel).
\end{theorem*}

In addition to the sample size bounds not always being linear, this result only works for special approximation spaces $\mathcal{V}_d$.
To tackle these shortcomings, we propose another sampling method in section~\ref{sec:sampling}.
Although the analysis of this method turns out to be difficult, we provide several properties as well as sample size bounds in Theorems~\ref{thm:iid_bounds} and~\ref{thm:expectation}.
These probability bounds, and the fact that $\mu$ is computable, allow us to condition on the event $\mu(\boldsymbol{x}) \le \mu_\star$, as proposed in~\cite{haberstich2022boosted}.
This strategy is described in pseudocode in Algorithm~\ref{alg:det}.

Finally, to further reduce the point set size of a random selection, we investigate greedy subsampling schemes, culminating in Algorithm~\ref{alg:greedy}.
Theoretical bounds for this algorithm are provided in Proposition~\ref{prop:mu_properties}, Theorem~\ref{thm:submodular_bound} and Proposition~\ref{prop:approx_submodular} and summarised in the subsequent theorem.

\begin{minipage}{\textwidth}
\begin{theorem*}
    Let $\boldsymbol{x}\in\mathcal{X}^n$ be given.
    Then there exists an efficiently computable function $\eta : 2^{\boldsymbol{x}} \to [0, d\ \!]$ satisfying
    \begin{itemize}
        \item $(\eta(\boldsymbol{y})/d)^{-1/2} \le \mu(\boldsymbol{y})$ for all subsets $\boldsymbol{y}\subseteq\boldsymbol{x}$ and
        \item $\mu(\boldsymbol{y}) \le (\eta(\boldsymbol{y}) - (d-1))^{-1/2}$ for any subset $\boldsymbol{y}\subseteq\boldsymbol{x}$ satisfying $\eta(\boldsymbol{y}) \ge d-1$.
    \end{itemize}
    Define the greedy sequence of subsets
    $$
        \boldsymbol{y}_0 := \emptyset,
        \qquad
        \boldsymbol{y}_{k+1} = \boldsymbol{y}_k \oplus y_{k+1}, \quad y_{k+1} \in   \argmax_{y\in\boldsymbol{x}} \eta(\boldsymbol{y}_k \oplus y) .
    $$
    For any $C\in[0,\eta(\boldsymbol{x})]$, define
    the smallest set size $k(C)$ required to reach $\eta(\boldsymbol{y}_{k(C)}) \ge C$ with this greedy method as
    $
        k(C) := \min\{k\in\mathbb{N} : \eta(\boldsymbol{y}_k)\ge C\}
    $
    and define the optimal set size as
    $
        k_\star(C) := \min\{|\boldsymbol{y}| : \boldsymbol{y}\subseteq \boldsymbol{x} \text{\ and\ } \eta(\boldsymbol{y})\ge C\}
    $.
    Then for any $\varepsilon\in(0,1)$
    $$
        k(C) \lesssim \log(\varepsilon^{-1}) k_\star(C(1-\varepsilon)^{-1}) \,.
    $$
\end{theorem*}
\end{minipage}

Although the theoretical guarantees are not satisfactory, yet, initial numerical experiments indicate the potential of the proposed approach.


\subsection{Structure}

The remainder of this work is organised as follows.
After section~\ref{sec:notation} concisely introduces some required notations,
section~\ref{sec:noiseless} defines the projection operator $P_{\mcal{V}_d}^{\boldsymbol{x}}$ and proves the main Theorem~\ref{thm:error_bound_noiseless}.
Section~\ref{sec:error_control} starts by deriving basic properties of the quasi-optimality constants $\mu$ and $\tau$.
Subsequently, it discusses that optimising $\mu$ is NP-hard and proposes to use a probabilistic approach to generate the point set $\boldsymbol{x}\in\mcal{X}^n$.
This idea is justified by demonstrating that drawing $\boldsymbol{x}\in\mcal{X}^n$ by continuous volume sampling~\cite{belhadji2020kernel} already yields a (relatively) small value of $\mu(\boldsymbol{x})$ with non-zero probability.
Section~\ref{sec:sampling} introduces a novel sampling method that is better adapted to the problem of minimising $\mu$, and section~\ref{sec:subsampling} discusses a greedy sample selection strategy that can be used to reduce the sample size further.
Finally, section~\ref{sec:noisy} extends the error bounds from Theorem~\ref{thm:error_bound_noiseless} to perturbed observations and section~\ref{sec:experiments} provides experimental evidence for the proposed methods.
Section~\ref{sec:discussion} concludes the paper by discussing some limitations of the presented method.

\subsection{Related work}

\paragraph{Least squares approximation.}

The proposed framework agrees with least squares approximation in the fact that both use point-evaluations.
This is advantageous since evaluations of this form are often given.
However, the proposed methodology has three major advantages compared to the least squares framework
\begin{enumerate}
    \item quasi-optimality: $\|u - P_{\mathcal{V}_d}^{\boldsymbol{x}}u\| \lesssim \|u - P_{\mathcal{V}_d}u\|$ with the same norm on both sides,
    \item computable constants: $\mu$ and $\tau$ can be computed from the sample for every choice of basis of $\mcal{V}_d$,
    \item stronger error bounds: the error is measured in an RKHS norm, not the $L^2$ norm.
\end{enumerate}%
Controlling the quasi-optimality constant through $\mu(\boldsymbol{x})$ is not a new idea and has already been considered in the context of least squares approximation.
In this context, we require a   measure $\nu$ on $\mcal{X}$ and suppose that $\mcal{V}_d\subseteq L^2(\nu)$ 
(Note that the existence of a measure $\nu$ is not required for the kernel-based regression). 
Then the $L^2(\nu)$-orthogonal projection onto $\mcal{V}_d$ is well-defined and we can denote it by $Q_{\mcal{V}_d}$.
Moreover, for a point set $\boldsymbol{x}\in\mcal{X}^n$, we can define the empirical projection
\begin{equation}
\label{eq:wls_projection}
    Q_{\mcal{V}_d}^{\boldsymbol{x}} u
    := \argmin_{v\in\mcal{V}_d}\, \seminorm{u - v}_{\boldsymbol{x}}^2
    \qquad\text{with}\qquad
    \seminorm{u-v}_{\boldsymbol{x}}^2
    := \frac1n \sum_{i=1}^n w(x_i) \abs{u(x_i) - v(x_i)}^2 .
\end{equation}
Given a weight function $w : \mcal{X}\to(0,\infty)$ satisfying $\int w^{-1} \dx[\nu] = 1$, the measure $w^{-1}\nu$ is a probability measure.
Drawing $\boldsymbol{x}\sim (w^{-1}\nu)^{\otimes n}$, one can prove~\cite{cohen_2017_optimal} that
\begin{equation}
\label{eq:def:mu_L2}
    \norm{u-Q_{\mcal{V}_d}^{\boldsymbol{x}} u}_{L^2(\nu)}^2 \le \norm{u-Q_{\mcal{V}_d} u}_{L^2(\nu)}^2 + \mu_{L^2}(\boldsymbol{x})\, \seminorm{\mkern-.5mu u-Q_{\mcal{V}_d} u}_{\boldsymbol{x}\mkern.5mu}^2
    \qquad\text{with}\qquad
    \mu_{L^2}(\boldsymbol{x}) = 
    \sup_{v\in\mcal{V}_d} \frac{\norm{v}_{L^2(\nu)}^2}{\seminorm{v}_{\boldsymbol{x}}^2} .
\end{equation}
The aim of many sampling methods (cf.~\cite{cohen_2017_optimal,haberstich2022boosted,Derezinski2022Jan,nouy2024weighted}) is to bound the quasi-optimality factor $\mu_{L^2}(\boldsymbol{x}) \le \mu_\star$ with high probability $p_\star$. With an optimal choice of the density  $w^{-1}$ based on concentration of measure arguments~\cite{Tropp2011}, we obtain that $n \gtrsim d \log(d)$ i.i.d. sample points suffice to achieve this.
In~\cite{chkifa2024randomizedleastsquaresminimaloversampling}, the authors propose a dependent sampling strategy that reduces the sampling complexity of least squares projection also to $\mathcal{O}(d)$.
However, it is not clear if one can sample efficiently from this distribution, especially in high dimension. 

Conditioned on the event $\mu_{L^2}(\boldsymbol{x}) \le \mu_\star$, it holds that
$$
    \mbb{E}\bracs*{\norm{u-Q_{\mcal{V}_d}^{\boldsymbol{x}} u}_{L^2(\nu)}^{2}}
    \le (1+ \tfrac{\mu_\star}{p_\star})\norm{u-Q_{\mcal{V}_d} u}_{L^2(\nu)}^2 .
$$
This error bound is similar to the bound of Theorem~\ref{thm:error_bound_noiseless}.
However, compared to the least squares approximation $Q_{\mcal{V}_d}^{\boldsymbol{x}}u$, our kernel-based approximation $P_{\mcal{V}_d}^{\boldsymbol{x}} u$ makes additional use of the RKHS structure to obtain several theoretical advantages:
\begin{itemize}
    \item in our setting, quasi-optimality holds with respect to stronger norms,
    \item the error bound of Theorem~\ref{thm:error_bound_noiseless} holds for every $\boldsymbol{x}\in\mathcal{X}^n$ and not only in expectation.
    This gives confidence to practitioners who want their approximation to be correct in all cases and not just in expectation.
    Moreover, these ``almost sure'' error bounds make it easy to reuse sample points while error bounds that hold in expectation must navigate the difficulties of statistical dependence,
    \item a critical practical difference is that the constant $\mu(\boldsymbol{x})$ is always computable while computing $\mu_{L^2}(\boldsymbol{x})$ requires access to an $L^2$-orthonormal basis.
    From a theoretical point of view, $\mu(\boldsymbol{x})$ has the delightful property that it always decreases when adding additional data points, while this is not guaranteed for $\mu_{L^2}(\boldsymbol{x})$.
\end{itemize}
There also exist results that do not  work in expectation or with high probability, but almost surely, just like ours.
However, these results are not quasi-optimal projections since the error on the right-hand side is measured in a stronger norm than $L^2(\nu)$.

The i.i.d.\ sampling strategy from our theory yields a suboptimal sample complexity when $\mathcal{V}_d$ is not spanned by the eigenfunctions associated with the kernel.
The well-known works~\cite{cohen_2017_optimal,haberstich2022boosted} achieve almost optimal sample complexities of $\mathcal{O}(d \log(d))$ with i.i.d.\ samples while requiring no more assumption than $\mcal{V}_d\subseteq L^2(\nu)$.
This contrasts starkly with the bounds derived in our work, which 
use a dependent sample and requires strong assumptions on both $\mcal{V}$ and $\mcal{V}_d$.

Error bounds for general point-evaluation based approximation algorithm in RKHS are analysed in~\cite{Krieg_2020}.
However, these bounds only hold for the space spanned by the eigenfunctions, and they are not quasi-optimal.
Similarly, the results of~\cite{belhadji2020kernel,belhadji2024signalreconstructionusingdeterminantal} only work for the space spanned by the eigenfunctions, limiting applicability.

A balance between the regularities of $L^2(\nu)$ and $\mathcal{V}$ can be provided by the space of uniformly bounded functions $\mathcal{L}^\infty(\mathcal{X})$ with norm $\|u\|_{\mathcal{L}^\infty(\mathcal{X})} = \sup_{x\in\mathcal{X}} |u(x)|$.
Fekete points are a choice of $n=d$ points satisfying $\|Q_{\mathcal{V}_d}^{\boldsymbol{x}} u\|_{\mathcal{L}^\infty(\mathcal{X})} \le \|u\|_{\mathcal{L}^\infty(\mathcal{X})}$ and ensure that the least squares projection satisfies the stability bound~\cite[Proposition~1.2.5]{Novak1988}
$$
    \|u - Q_{\mathcal{V}_d}^{\boldsymbol{x}} u\|_{\mathcal{L}^\infty(\mathcal{X})}
    \le (n+1) \inf_{v\in\mathcal{V}_d} \|u - v\|_{\mathcal{L}^\infty(\mathcal{X})}
    .
$$
However, computing Fekete points for general domains $\mathcal{X}$ and spaces $\mathcal{V}_d$ can be intractable.
Leja sequences offer a greedy alternative (with larger yet still polynomial in $d$ quasi-optimality factors) for polynomial spaces $\mathcal{V}_d$ over compact domains $\mathcal{X} \subseteq \mathbb{R}$ or $\mathcal{X} \subseteq\mathbb{C}$~\cite{Andrievskii2022,Chkifa2013,Cohen2015}.

\paragraph{Kernel interpolation.}
When approximating the function $u$ living in a RKHS $\mcal{V}$ using point evaluations at $\boldsymbol{x}\in\mcal{X}^n$, it is natural to consider the kernel interpolation $P_{\mcal{V}_{\boldsymbol{x}}} u$.
The advantage of the proposed method over kernel interpolation lies in 
the fact that we can introduce further knowledge about the function $u$
to reduce the error.
Suppose that $\mcal{X}\subseteq\mbb{R}^D$ is a bounded Lipschitz domain and that $\mcal{V} = H^s(\mcal{X})$ for some $s>\tfrac{D}{2}$.
Then classical convergence rates for kernel interpolation (cf.~\cite{Wendland_2004} or~\cite[Corollary~2.3.11]{kempf_thesis}) are given by
$$
    \norm{u - P_{\mcal{V}_{\boldsymbol{x}}} u}_{H^t(\mcal{X})} \lesssim n^{-(s-t)/D},
    \quad 0\le t\le s,
$$
for all $u\in H^s(\mcal{X})$.
While this rate is optimal for approximation in a $n$-dimensional subspace of $\mathcal{V}$ (due to the Kolmogorov widths of Sobolev spaces), this rate is always algebraic and suffers from the curse of dimensionality.
A method to determine a set of points $\boldsymbol{x}$ that reach the above rate 
is addressed in~\cite{belhadji2020kernel} by proposing a random strategy to choose $\boldsymbol{x}$ (see Theorem~\ref{thm:ayoub} below).

Often, however, additional information is available.
To illustrate this, suppose that the best approximations $P_{\mcal{V}_d} u$ of $u$ in a sequence of space $(\mcal{V}_d)_{d\in \mathbb{N}}$ admit a convergence rate of
$$
    \norm{u-P_{\mcal{V}_d} u}_{\mcal{V}} \lesssim d^{-s}
    \qquad\text{or even}\qquad
    \norm{u-P_{\mcal{V}_d} u}_{\mcal{V}} \lesssim e^{-rd}
$$
for some $r>0$.
Since kernel interpolation is restricted to approximation in the spaces $\mathcal{V}_{\boldsymbol{x}}$, it can not directly benefit from this additional information.

Under suitable assumptions, and when $\mathcal{V}_d$ is spanned by eigenfunctions of the integral operator associated to $k$,~\cite{belhadji2024signalreconstructionusingdeterminantal} construct a sample-based approximation $\hat{u}\in\mathcal{V}_d$ satisfying
$$
    \mathbb{E}\big[\|u - \hat{u}\|_{\mathcal{V}}^2\big] \le (1 + d) \|u - P_{\mathcal{V}_d}u\|_{\mathcal{V}}^2 .
$$

%
In contrast,
Theorem~\ref{thm:error_bound_noiseless} below ensures that every linear space $\mcal{V}_d$ admits a suboptimality constant $\mu_d(\boldsymbol{x})$ such that
\begin{equation}
\label{eq:quasi-optimality}
    \norm{u - P_{\mcal{V}_d}^{\boldsymbol{x}} u}_{\mcal{V}}^2 
    \le (1 + \mu_d(\boldsymbol{x})^2) \norm{u - P_{\mcal{V}_d} u}_{\mcal{V}}^2 ,
\end{equation}
providing additional flexibility.
We argue that we can choose $\boldsymbol{x}$   to make the quasi-optimality constant as small as desired. 
This discussion presumes that the approximation space $\mathcal{V}_d$ is fixed (e.g.\ given by approximation theoretic results) and that the point set $\boldsymbol{x}$ can be chosen freely.
If $\boldsymbol{x}$ is fixed, as is often the case in applications, it is conceivable to use the bounds presented in this paper to select a suitable subspace $\mathcal{V}_d$ of a larger space of candidate functions $\mathcal{V}_D$ by means of an eigenvalue decomposition of an empirical Gramian matrix (see Section~\ref{sec:discussion}).



\paragraph{Parameterised background data weak (PBDW) method.}

\begin{figure}
    \centering
    \begin{subfigure}[b]{0.32\textwidth}
        \centering
        \hspace*{-2.6em}
        \includegraphics[width=1.2\textwidth]{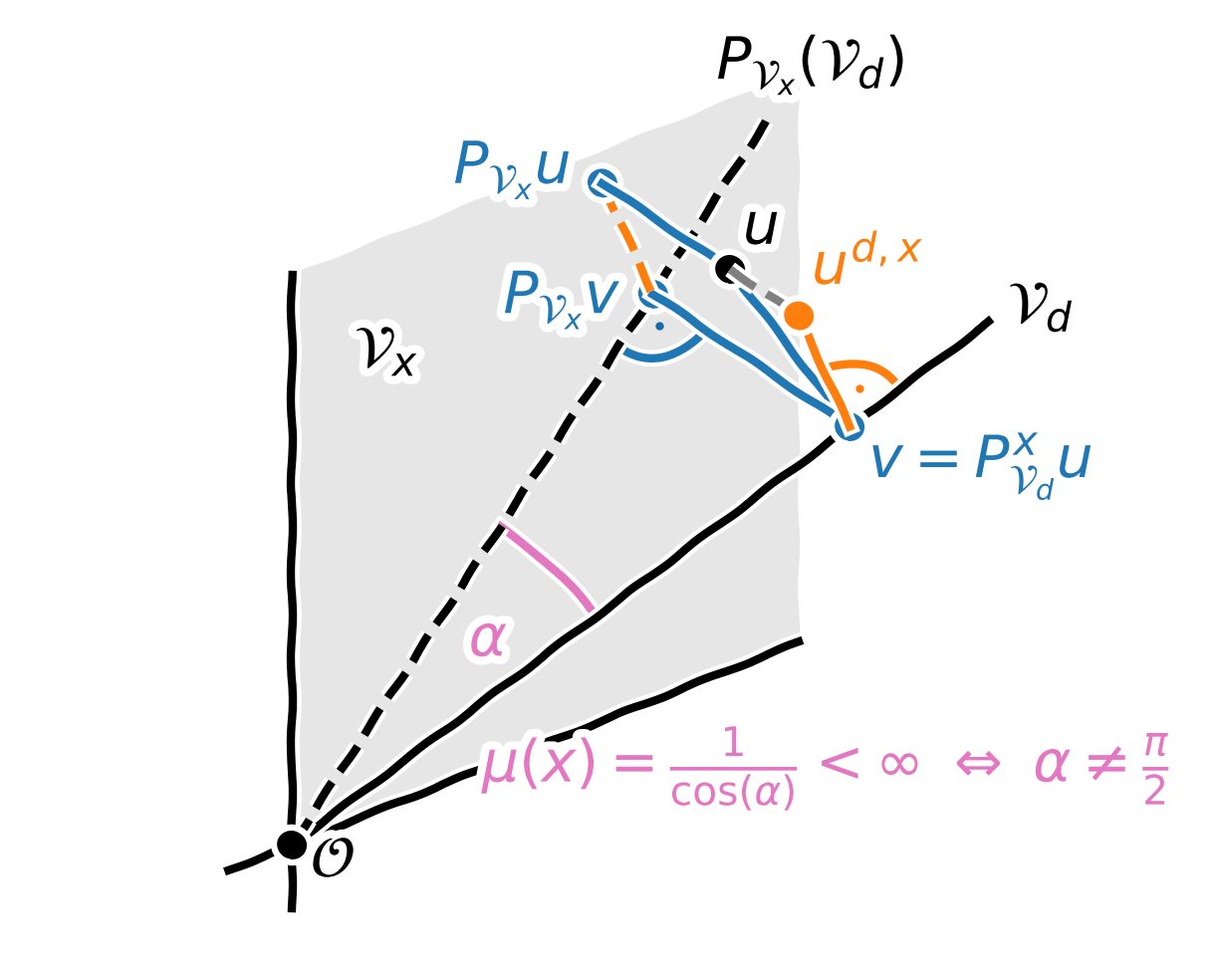}
        \caption{Linear $\mathcal{V}_d$ with bounded $\mu(x)$.}
        \label{fig:space_visualisation:linear}
    \end{subfigure}
    \hfill
    \begin{subfigure}[b]{0.32\textwidth}
        \centering
        \hspace*{-2em}
        \includegraphics[width=1.2\textwidth]{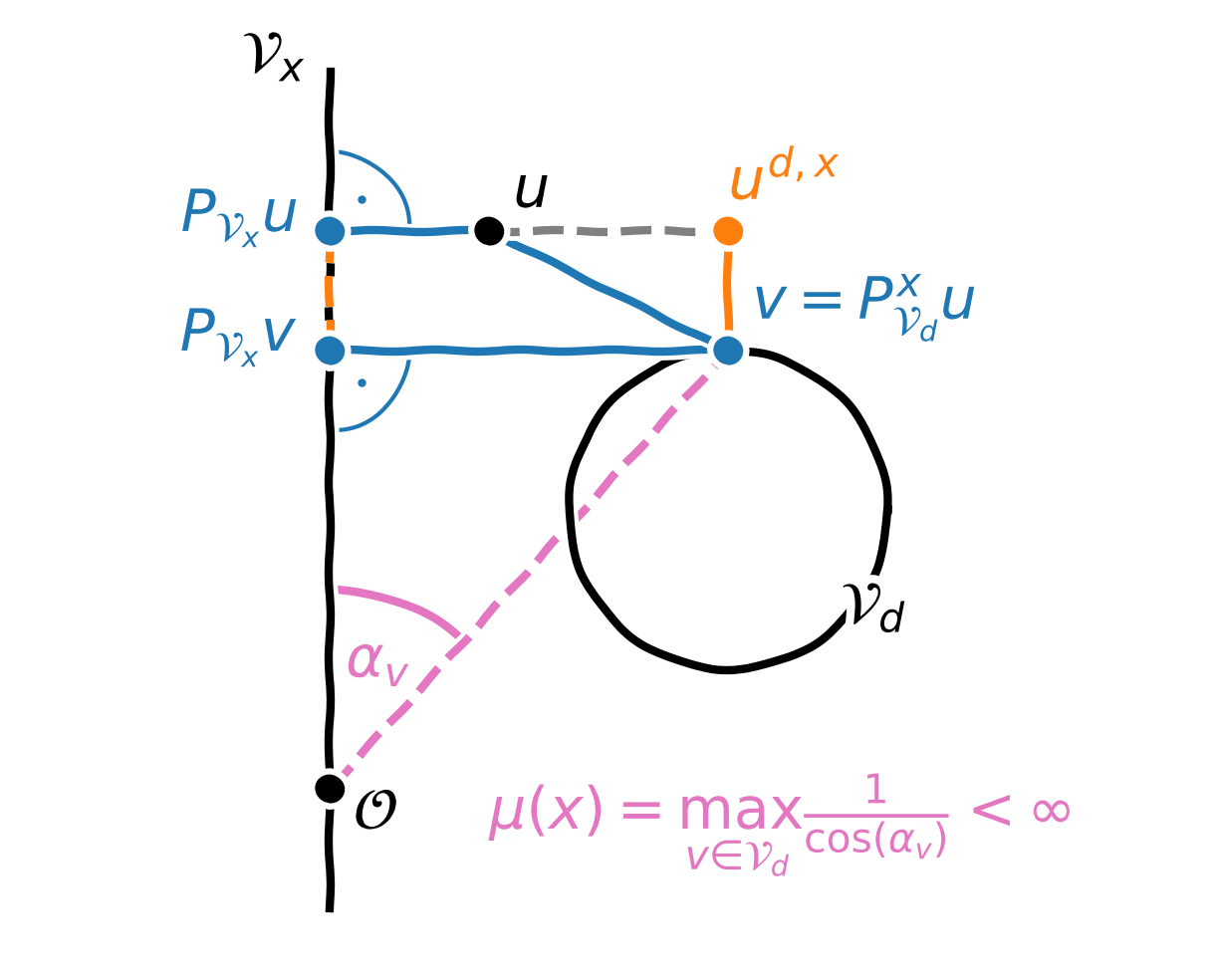}
        \caption{Nonlinear $\mathcal{V}_d$ with bounded $\mu(x)$.}
        \label{fig:space_visualisation:circle}
    \end{subfigure}
    \hfill
    \begin{subfigure}[b]{0.32\textwidth}
        \centering
        \hspace*{-2em}
        \includegraphics[width=1.2\textwidth]{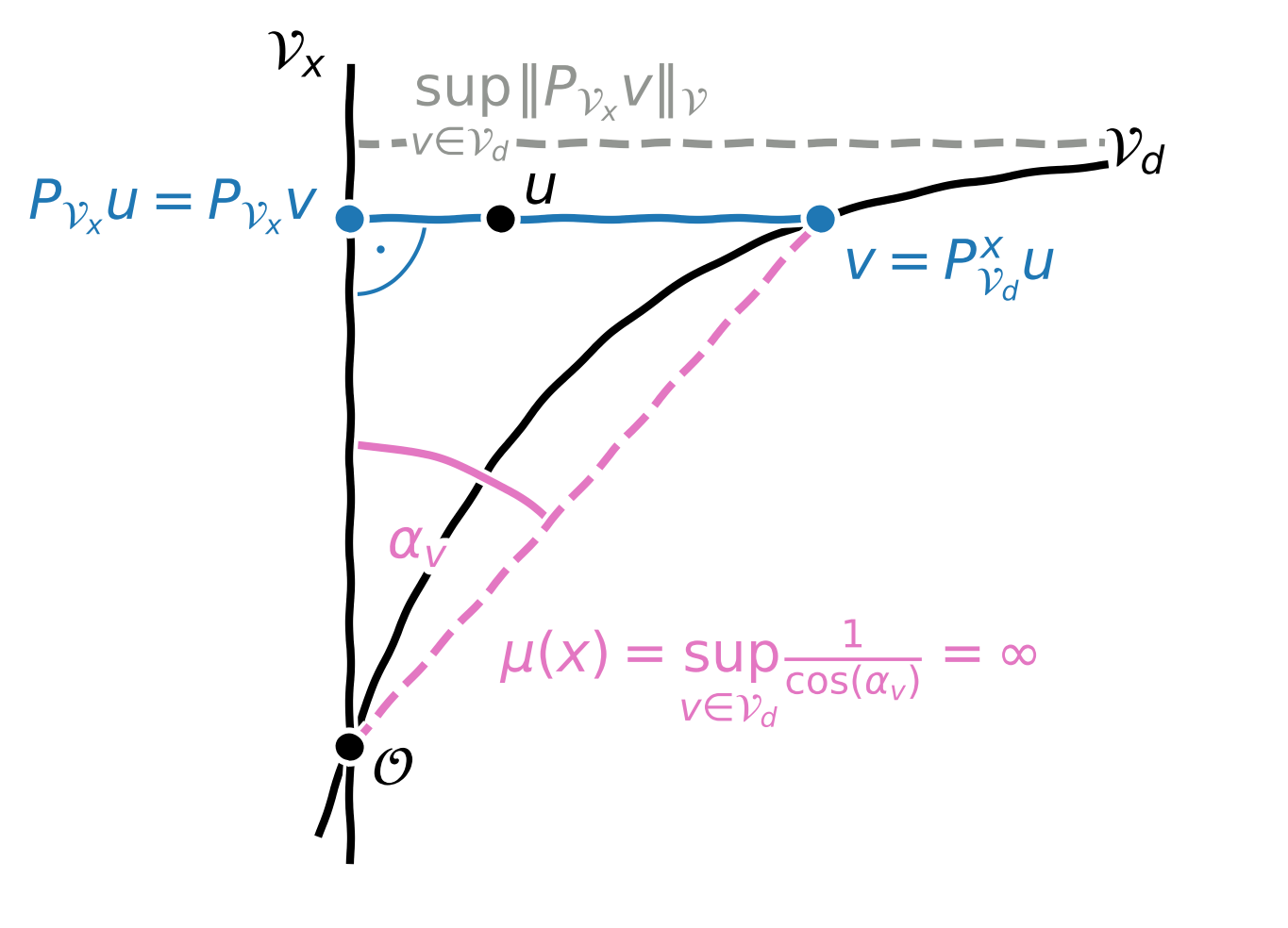}
        \caption{Nonlinar $\mathcal{V}_d$ with unbounded $\mu(x)$.}
        \label{fig:space_visualisation:diverging}
    \end{subfigure}
    \caption{
    Visualisation of projections involved in the PBDW method and relation of $\mu(x)$ to the angle between $\mathcal{V}_d$ and $\mathcal{V}_x$ for (\ref{fig:space_visualisation:linear}) linear $\mathcal{V}_d$, (\ref{fig:space_visualisation:circle}) nonlinear $\mathcal{V}_d$ with bounded $\mu(x)$ and (\ref{fig:space_visualisation:diverging}) nonlinear $\mathcal{V}_d$ with unbounded $\mu(x)$.
    Here $u^{d,\boldsymbol{x}}$ is the interpolant defined in~\eqref{eqn:def_u_dx}.
    }
    \label{fig:space_visualisation}
\end{figure}

The projector $P_{\mathcal{V}_d}^{\boldsymbol{x}}$ is a special case of the best-fit estimator introduced in~\cite{cohen2022nonlinear}.
In particular, Theorem~\ref{thm:error_bound_noiseless} below is a refined version of Theorem~2.3 in~\cite{cohen2022nonlinear}.
(Note that the general PBDW method works with nonlinear approximation spaces $\mathcal{V}_d$.)
The novelty in our approach is that we consider $P_{\mathcal{V}_d}^{\boldsymbol{x}}$ primarily as a tool for function approximation from point evaluations.
The norms and stability constants involved in our setting are explicitly computable, 
which is not possible in the general PBDW setting and might be very laborious and technical in special cases.

Since we have access to the projection $P_{\mathcal{V}_{\boldsymbol{x}}}$, it is always possible to obtain an interpolation $w = P_{\mathcal{V}_{\boldsymbol{x}}} u$, satisfying $\|P_{\mathcal{V}_{\boldsymbol{x}}}(u - w)\|_{\mathcal{V}}=0$.
Combining the approximation $P_{\mcal{V}_d}^{\boldsymbol{x}}u$ with a kernel interpolation of the residual $P_{\mcal{V}_{\boldsymbol{x}}} (I - P_{\mcal{V}_d}^{\boldsymbol{x}}) u$ gives rise to the PBDW method~\cite{Maday2015May,cohen2022nonlinear},
which defines the interpolant
\begin{align}
    \label{eqn:def_u_dx}
    u^{d,\boldsymbol{x}} := P_{\mcal{V}_d}^{\boldsymbol{x}} u + P_{\mcal{V}_{\boldsymbol{x}}} (I - P_{\mcal{V}_d}^{\boldsymbol{x}}) u .
\end{align}
Figure~\ref{fig:space_visualisation} illustrates the different projections $P_{\mathcal{V}_d}^{\boldsymbol{x}}u$, $P_{\mathcal{V}_{\boldsymbol{x}}}u$ and $u^{d,\boldsymbol{x}}$ as well as the connection between the stability constant $\mu(x)$ and the angle between approximation space $\mathcal{V}_d$ and kernel interpolation space $\mathcal{V}_x$.
Theorem~\ref{thm:error_bound_noiseless} and the $\mcal{V}$-orthogonality of the kernel interpolation operator $P_{\mcal{V}_{\boldsymbol{x}}}$ trivially yield the PBDW error bound
$$
    \norm{u - u^{d,\boldsymbol{x}}}_{\mcal{V}}
    \le \norm{u - P_{\mcal{V}_d}^{\boldsymbol{x}} u}_{\mcal{V}}
    + \norm{P_{\mcal{V}_{\boldsymbol{x}}} (I - P_{\mcal{V}_d}^{\boldsymbol{x}}) u}_{\mcal{V}}
    \le 2 \norm{u - P_{\mcal{V}_d}^{\boldsymbol{x}} u}_{\mcal{V}}
    \le 2 \sqrt{1 + \mu(\boldsymbol{x})^2\tau(\boldsymbol{x})^2} \norm{u - P_{\mcal{V}_d} u}_{\mcal{V}}
    .
$$
Note that this error bound is not optimal in general, since $u^{d,\boldsymbol{x}} \in (\mcal{V}_d + \mcal{V}_{\boldsymbol{x}})$ is compared to $P_{\mcal{V}_d}u \in \mcal{V}_d$.
Using intricate properties of the projectors, we can obtain the tighter bound~\cite{binev2015data}
\begin{align}
	\norm{u - u^{d,\boldsymbol{x}}}_{\mcal{V}}
	\le \mu(\boldsymbol{x}) \norm{u - P_{\mcal{V}_d\oplus(\mcal{V}_{\boldsymbol{x}}\cap \mathcal{V}_d^\perp)} u}_{\mcal{V}} .
\label{eq:PBDW_bound}
\end{align}
For the sake of completeness, we provide a simplified proof of this bound in the linear setting in Appendix~\ref{app:PBDW_bound}.

\paragraph{Optimal measurements selection.}%
The greedy Algorithm~\ref{alg:greedy} presented in section~\ref{sec:subsampling} is not the first attempt to optimise the constant $\mu(\boldsymbol{x})$.
In the more general context of approximating a function from bounded linear measurements,~\cite{Binev2018} introduced a greedy optimisation strategy for selecting optimal measurements.
This strategy, however, requires optimising over an infinite set of measurements and computing exact dual pairings, which is often infeasible in practice.

\section{Notations}
\label{sec:notation}

\begin{itemize}
    \item $\mcal{V}$ is a RKHS of functions on the set $\mcal{X}$.
    \item The inner product of $\mcal{V}$ is denoted by $(\bullet,\bullet)_{\mcal{V}}$ and the norm of $\mcal{V}$ is denoted by $\norm{\bullet}_{\mcal{V}}$.
    \item The reproducing kernel of $\mcal{V}$ is denoted by $k:\mcal{X}\times \mcal{X} \to \mbb{R}$.
    \item $\mcal{V}_d$ is a fixed, $d$-dimensional subspace of $\mcal{V}$.

    \item For any $\mcal{W}\subseteq\mcal{V}$ we let $P_{\mcal{W}}$ denote the $\mcal{V}$-orthogonal projection onto $\mcal{W}$.

    \item We identify a finite sequence of functions $b_1,\ldots,b_m : \mcal{X}\to\mcal{Y}$ with the vector-valued function $b : \mcal{X}\to\mcal{Y}^m$.

    \item It will be convenient to assign an arbitrary ordering to the point sets and view them as elements of $\mcal{X}^n$. The ordering has no influence on the theoretical arguments and is only used to simplify the notation.
    \item Given two point sets $\boldsymbol{x}\in\mcal{X}^n$ and $\boldsymbol{y}\in\mcal{X}^m$, we write $\boldsymbol{x}\oplus \boldsymbol{y}\in\mcal{X}^{n+m}$ as the concatenation of both sets.
    \item Given $\boldsymbol{x}\in\mcal{X}^n$ and $\boldsymbol{y}\in\mcal{X}^m$, we write $\boldsymbol{x} \subseteq \boldsymbol{y}$ if $n\le m$ and $\forall\,i\in\{1,\ldots n\}$, $\exists\,j\in\{1,\ldots,m\}$ such that $ \boldsymbol{y}_j = \boldsymbol{x}_i$.

    \item We define for any function $f : \mcal{X} \to \mcal{Y}$ and $\boldsymbol{x}\in\mcal{X}^n$ the vector of evaluations $f(\boldsymbol{x})\in\mcal{Y}^{n}$ by $f(\boldsymbol{x})_{k} := f(x_k)$.
    \item One important instance of this notation is the partial evaluation of the kernel $k$.
    Interpreting the kernel as a function $k : \mcal{X} \to \mbb{R}^{\mcal{X}}$, we can write $k(\boldsymbol{x}, \bullet) : \mcal{X} \to \mbb{R}^n$ for any $\boldsymbol{x} \in \mcal{X}^n$.

    \item Another important instance is the matrix of evaluations of a vector-valued function $f : \mcal{X} \to \mbb{R}^m$ on points $\boldsymbol{x}\in\mcal{X}^n$.
    Using the introduced notation we can interpret $f(\boldsymbol{x})\in(\mbb{R}^m)^n\simeq\mbb{R}^{m\times n}$ as $f(\boldsymbol{x})_{jk} := f(x_k)_j$. 
    The indices are ordered mnemonically. ``$f$'' comes before ``$\boldsymbol{x}$'' in ``$f(\boldsymbol{x})$'' and the index $j$ comes before $k$.

    \item   $\operatorname{span}(b) = \operatorname{span}\{b_1,\hdots , b_m\}$ denotes the linear space generated by $b : \mcal{X}\to\mbb{R}^m$. 
    \item $k(\boldsymbol{x}, \bullet)$ forms a generating system of the linear space
    $$
        \mcal{V}_{\boldsymbol{x}} := \operatorname{span}(k(\boldsymbol{x}, \bullet)) =  \operatorname{span}\{k(x_1,\bullet) , \hdots , k(x_n,\bullet)\}.
    $$
    \item Due to its significance, we define the special notation
    $$
        K(\boldsymbol{x}) := k(\boldsymbol{x}, \boldsymbol{x}) \in \mbb{R}^{n\times n},
    $$
    which is not only the evaluation of $k(\boldsymbol{x}, \bullet)$ at the point set $\boldsymbol{x}$ but also the Gramian matrix (or kernel matrix)
    $$
        K(\boldsymbol{x})_{jk}
        = (k(x_j, \bullet), k(x_k, \bullet))_{\mcal{V}} .
    $$
    \item Let $b : \mcal{X}\to\mbb{R}^d$ be a $\mcal{V}$-orthonormal basis of $\mcal{V}_d$.
    Then the reproducing kernel for $\mcal{V}_d$ is given by
    $$
        k_d(x, y)
        := ((P_{\mcal{V}_d}\otimes P_{\mcal{V}_d}) k)(x,y)
        = b(x)^\intercal b(y) .
    $$
    We denote the corresponding kernel matrix by $K_d(\boldsymbol{x})$.
    \item For all matrices $M\in\mathbb{R}^{m\times m}$ we denote by $M^+$ the Moore--Penrose pseudo-inverse of $M$.
    \item To simplify notation, we call any finite measure $\mu$ an unnormalised probability measure and write $x\sim\mu$ to denote a sample from the corresponding probability measure.
    \item For any two sequences $a,b$ we write $a \lesssim b$ if there is a constant $c > 0$ such that, element-wise, $a \le c b$.
    \item We write $a \asymp b$ if $b \lesssim a\lesssim b$.
\end{itemize}
\section{Exact evaluations}
\label{sec:noiseless}

Given a point set $\boldsymbol{x} := (x_1, \ldots, x_n)\in\mathcal{X}^n$, we let $P_{\mcal{V}_{\boldsymbol{x}}}$ denote the $\mcal{V}$-orthogonal projection onto $\mcal{V}_{\boldsymbol{x}} = \operatorname{span}(k(\boldsymbol{x},\bullet))$, and  
we  define the (semi-)inner product and its induced (semi-)norm by
$$
    \pars{\bullet, \bullet}_{\boldsymbol{x}} := \pars{P_{\mcal{V}_{\boldsymbol{x}}}\bullet, \bullet}_{\mcal{V}}
    \qquad\text{and}\qquad
    \|\bullet\|_{\boldsymbol{x}}
    := \norm{P_{\mcal{V}_{\boldsymbol{x}}} \bullet}_{\mcal{V}} .
$$
We let $P_{\mcal{V}_d}^{\boldsymbol{x}}$ denote the orthogonal projection onto $\mcal{V}_d\subseteq\mcal{V}$ with respect to the (semi-)inner-product $\pars{\bullet, \bullet}_{\boldsymbol{x}}$, defined for 
$u\in\mcal{V}$ by
\begin{equation}
\label{eq:def:PVdx}
    P_{\mcal{V}_d}^{\boldsymbol{x}} u := \argmin_{v\in\mcal{V}_d} \; \norm{u - v}_{\boldsymbol{x}} .
\end{equation}
The minimiser is uniquely defined when the constant
\begin{equation}
\label{eq:def:mu}
    \mu(\boldsymbol{x})
    := \max_{v\in\mathcal{V}_d} \frac{\|v\|_{\mathcal{V}}}{\|v\|_{\boldsymbol{x}}}
\end{equation}
is finite.
The subsequent three lemmas shows that the projection and the constant $\mu(\boldsymbol{x})$ are easily computable.

\begin{lemma}
\label{lem:x_inner_matrix}
    Let $v, w\in\mcal{V}$ and $\boldsymbol{x}\in\mcal{X}^n$.
    Then $\pars{v, w}_{\boldsymbol{x}} = v\pars{\boldsymbol{x}}^\intercal K(\boldsymbol{x})^{+} w\pars{\boldsymbol{x}}$.
\end{lemma}
\begin{proof}
    Since $P_{\mcal{V}_{\boldsymbol{x}}} v = v(\boldsymbol{x})^\intercal K(\boldsymbol{x})^{+} k(\boldsymbol{x}, \bullet) = \sum_{j=1}^n \pars*{K(\boldsymbol{x})^{+} v(\boldsymbol{x}) }_j k(x_j, \bullet)$, it follows that
    \begin{align}
        (v, w)_{\boldsymbol{x}}
        &= (P_{\mcal{V}_{\boldsymbol{x}}} v, P_{\mcal{V}_{\boldsymbol{x}}} w)_{\mcal{V}} \\
        &= \sum_{j,k=1}^n  (K(\boldsymbol{x})^{+} v(\boldsymbol{x}))_j  (k(x_j, \bullet), k(x_k, \bullet))_{\mcal{V}} (K(\boldsymbol{x})^{+} w(\boldsymbol{x}))_k \\
        &= v(\boldsymbol{x})^\intercal K(\boldsymbol{x})^{+} K(\boldsymbol{x}) K(\boldsymbol{x})^{+} w(\boldsymbol{x}) \\
        &= v(\boldsymbol{x})^\intercal K(\boldsymbol{x})^{+} w(\boldsymbol{x}). \qedhere
    \end{align}
\end{proof}

\begin{lemma}
\label{lem:mu_lmin}
    Let $\boldsymbol{x}\in\mcal{X}^n$ and define the Gramian matrix
    \begin{equation}
        \label{eq:def:Gramian_x}
        G^{\boldsymbol{x}}_{jk} 
        := (b_j, b_k)_{\boldsymbol{x}} 
        = (b(\boldsymbol{x})K(\boldsymbol{x})^+ b(\boldsymbol{x})^\intercal)_{jk}.
    \end{equation}
    It holds that $\mu(\boldsymbol{x}) = \lambda_{\mathrm{min}}(G^{\boldsymbol{x}})^{-1/2}$.
\end{lemma}
\begin{proof}
    Since every $v\in\mcal{V}_d$ can be written as $v(x) = c^\intercal b(x)$, Lemma~\ref{lem:x_inner_matrix} implies
    \begin{align}
        \mu(\boldsymbol{x})^2
        &= \max_{v\in\mathcal{V}_d} \frac{\|v\|_{\mathcal{V}}^2}{\|v\|_{\boldsymbol{x}}^2}
        = \pars*{\min_{c\in\mathbb{R}^d} \frac{\|c^\intercal b\|_{\boldsymbol{x}}^2}{\|c^\intercal b\|_{\mathcal{V}}^2}}^{-1}
        = \pars*{
        \min_{c\in\mathbb{R}^d} \frac{c^\intercal b(\boldsymbol{x}) K(\boldsymbol{x})^+ b(\boldsymbol{x})^\intercal c}{\|c\|_2^2}
        }^{-1}
        = \lambda_{\mathrm{min}}(G^{\boldsymbol{x}})^{-1} .
        \qedhere
    \end{align}
\end{proof}

\begin{lemma}
    Suppose that $\mu(\boldsymbol{x}) < \infty$.
    Then the projection $P_{\mathcal{V}_d}^{\boldsymbol{x}}$ is well-defined and satisfies
    $$
        P_{\mathcal{V}_d}^{\boldsymbol{x}}u = b^\intercal (G^{\boldsymbol{x}})^{-1} b(\boldsymbol{x}) K(\boldsymbol{x})^+ u(\boldsymbol{x}) \,.
    $$
\end{lemma}
\begin{proof}
    Any $v\in\mathcal{V}_d$ can be written as $v = b^\intercal \boldsymbol{v}$ for a unique $\boldsymbol{v}\in\mathbb{R}^d$.
    Lemma~\ref{lem:x_inner_matrix} thus implies for every $u\in\mathcal{V}$ that
    \begin{align}
        \|u - v\|_{\boldsymbol{x}}^2
        = \|u\|_{\boldsymbol{x}}^2
        -2u(\boldsymbol{x})^\intercal K(\boldsymbol{x})^+ b(\boldsymbol{x})^\intercal \boldsymbol{v}
        + \boldsymbol{v}^\intercal b(\boldsymbol{x}) K(\boldsymbol{x})^+ b(\boldsymbol{x})^\intercal \boldsymbol{v} \,.
    \end{align}
    Using the definition of the Gramian~\eqref{eq:def:Gramian_x}, the first order optimality condition for $\boldsymbol{v}$ can be written as
    $$
        G^{\boldsymbol{x}} \boldsymbol{v}
        = b(\boldsymbol{x}) K(\boldsymbol{x})^+ u(\boldsymbol{x})
        \,.
    $$
    Since $\mu(\boldsymbol{x}) < \infty$, Lemma~\ref{lem:mu_lmin} ensures that $G^{\boldsymbol{x}}$ is invertible.
    The claim follows by solving for $\boldsymbol{v}$ and using $v = b^\intercal \boldsymbol{v}$.
\end{proof}

The following theorem provides error bounds for the projection~\eqref{eq:def:PVdx} in terms of the constant $\mu(\boldsymbol{x})$ from~\eqref{eq:def:mu}.
It tightens the error bound for this type of approximation, which was first presented in Theorem~2.3 of~\cite{cohen2022nonlinear}.

In contrast to similar error bounds for least squares projections (cf.~\eqref{eq:def:mu_L2}), the key ingredient of the subsequent error bound is the property $\|\bullet\|_{\boldsymbol{x}}\le \|\bullet\|_{\mathcal{V}}$, which allows us to obtain true quasi-optimality.
The norm $\|\bullet\|_{\boldsymbol{x}}$ can also be seen as an instance of the measurement by random projectors, which is studied in~\cite{adcock2023unified,gruhlke2024optimal,eigel_convergence_2020}.

\begin{minipage}{\textwidth}
\begin{theorem}
\label{thm:error_bound_noiseless}
    Let $u\in\mcal{V}$, $\boldsymbol{x}\in\mcal{X}^n$ and suppose that $\mu(\boldsymbol{x}) < \infty$.
    Then
    $$
        \norm{(I - P_{\mcal{V}_d}^{\boldsymbol{x}}) u}_{\mcal{V}}^2 
        \le (1 + \mu(\boldsymbol{x})^2 \tau(\boldsymbol{x})^2) \norm{(I - P_{\mcal{V}_d}) u}_{\mcal{V}}^2 ,
    $$
    with the constant $\tau(\boldsymbol{x}) := \min\braces{\norm{I - G^{\boldsymbol{x}}}_{\mathrm{Fro}} + \norm{I - G^{\boldsymbol{x}}}_{2}, 1}$.
\end{theorem}
\end{minipage}

\begin{proof}
    Since $P^{\boldsymbol{x}}_{\mcal{V}_d}$ is a projection, it holds that
    \begin{align}
        \norm{(I - P_{\mcal{V}_d}^{\boldsymbol{x}}) u}_{\mcal{V}}^2
        &= \norm{(I - P_{\mcal{V}_d}) u}_{\mcal{V}}^2 + \norm{(P_{\mcal{V}_d} - P_{\mcal{V}_d}^{\boldsymbol{x}}) u}_{\mcal{V}}^2 \\
        &\le \norm{(I - P_{\mcal{V}_d}) u}_{\mcal{V}}^2 + \mu(\boldsymbol{x})^2\norm{(P_{\mcal{V}_d}^{\boldsymbol{x}} - P_{\mcal{V}_d}) u}_{\boldsymbol{x}}^2 \\
        &= \norm{(I - P_{\mcal{V}_d}) u}_{\mcal{V}}^2 + \mu(\boldsymbol{x})^2\norm{P_{\mcal{V}_{\boldsymbol{x}}} P_{\mcal{V}_d}^{\boldsymbol{x}} (I - P_{\mcal{V}_d}) u}_{\mcal{V}}^2 \\
        &\le \norm{(I - P_{\mcal{V}_d}) u}_{\mcal{V}}^2 + \mu(\boldsymbol{x})^2\norm{P_{\mcal{V}_{\boldsymbol{x}}} P_{\mcal{V}_d}^{\boldsymbol{x}} (I - P_{\mcal{V}_d})}_{\mcal{V}\to\mcal{V}}^2 \norm{(I - P_{\mcal{V}_d}) u}_{\mcal{V}}^2
        .
    \end{align}
    Defining the operator $A = P_{\mcal{V}_{\boldsymbol{x}}} P_{\mcal{V}_d}^{\boldsymbol{x}} (I - P_{\mcal{V}_d})$, it remains to bound $\norm{A}_{\mcal{V}\to\mcal{V}} \le \tau(\boldsymbol{x})$.
    Since $P_{\mcal{V}_{d}}^{\boldsymbol{x}}$ is an $(\bullet,\bullet)_{\boldsymbol{x}}$-orthogonal projection and $P_{\mcal{V}_{d}}$ is a $\mcal{V}$-orthogonal projection, it holds for every $v\in\mcal{V}$ that
    $$
        \norm{P_{\mcal{V}_{\boldsymbol{x}}} P_{\mcal{V}_d}^{\boldsymbol{x}} (I - P_{\mcal{V}_d}) v}_{\mcal{V}}
        = \norm{P_{\mcal{V}_d}^{\boldsymbol{x}} (I - P_{\mcal{V}_d}) v}_{\boldsymbol{x}}
        \le \norm{(I - P_{\mcal{V}_d}) v}_{\boldsymbol{x}}
        \le \norm{(I - P_{\mcal{V}_d}) v}_{\mcal{V}}
        \le \norm{v}_{\mcal{V}}
        .
    $$
    This implies $\norm{A}_{\mcal{V}\to\mcal{V}} \le 1$.
    Deriving the bound $\norm{A}_{\mcal{V}\to\mcal{V}} \le \norm{I - G^{\boldsymbol{x}}}_{\mathrm{Fro}} + \norm{I - G^{\boldsymbol{x}}}_{2}$ is more involved and therefore deferred to the Lemma~\ref{lem:tau_bound} in Appendix~\ref{app:tau_bound}.
\end{proof}

\begin{minipage}{\textwidth}
\begin{remark}
\label{rmk:cross-gramian}
    It was already noted in~\cite[Remark~2.13]{binev2015data} that $\mu(\boldsymbol{x})$ can be computed as the inverse of the $d$\textsuperscript{th} largest singular value of the cross-Gramian matrix $H^{\boldsymbol{x}}_{jk} := (b_j, \omega_k)_{\mcal{V}}$, where $b_1,\ldots,b_d$ is a $\mcal{V}$-orthonormal basis of $\mcal{V}_d$ and $\omega_1, \ldots, \omega_n$ is a $\mcal{V}$-orthonormal basis of $\mcal{V}_{\boldsymbol{x}}$.
    This is equivalent to Lemma~\ref{lem:mu_lmin}.
    When $K(\boldsymbol{x}) = U\Lambda U^\intercal$ is the rank-revealing spectral decomposition with $U\in\mathbb{R}^{n\times r}$ and $\Lambda\in\mathbb{R}^{r\times r}$, the basis $\omega_1,\ldots,\omega_n$ can be defined by $\omega = \Lambda^{-1/2} U^\intercal k(\boldsymbol{x}, \bullet)$.
    The smallest singular value of $H^{\boldsymbol{x}}$ is given by the square root of the smallest eigenvalue of $H^{\boldsymbol{x}} (H^{\boldsymbol{x}})^\intercal = G^{\boldsymbol{x}}$.
    Given that $n\ge d$, we can w.l.o.g.\ assume that 
    $$
        (b_j, \omega_k)_{\mathcal{V}} = \beta_j \delta_{jk}
    $$
    with $1 \ge \beta_1 \ge \ldots\ge \beta_d \ge 0$ and $\beta_{d+1} = \beta_{d+2} = \ldots = 0$.
    This shows that $\lambda_{\mathrm{min}}(G^{\boldsymbol{x}}) = \beta_d^2$ is maximised when $\mathcal{V}_d$ can be well approximated by $\mathcal{V}_{\boldsymbol{x}}$.
\end{remark}
\end{minipage}

\section{Controlling the error}
\label{sec:error_control}

Theorem~\ref{thm:error_bound_noiseless} implies that to control the error of the projection $P_{\mcal{V}_d}^{\boldsymbol{x}} u$, it is sufficient to control the constant $\mu(\boldsymbol{x})$.
It is, therefore, natural to ask whether we can minimise this constant.
Qualitative approximation guarantees can be derived with the help of the subsequent lemma.

\begin{lemma}
    \label{lem:gramian_ordering}
    Let $\boldsymbol{x}\subseteq \boldsymbol{y}$.
    Then $0\preceq G^{\boldsymbol{x}} \preceq G^{\boldsymbol{y}} \preceq I$ in Loewner ordering.
    In particular, $\lambda_{\mathrm{max}}(G^{\boldsymbol{x}}) \le 1$ for all $\boldsymbol{x}\in\mathcal{X}^n$.
\end{lemma}
\begin{proof}
    Let $\boldsymbol{v}\in\mathbb{R}^d$ be the coefficient vector of the function $v := \sum_{j=1}^d \boldsymbol{v}_j b_j$ in $\mathcal{V}_d$.
    Then,
    $$
        \boldsymbol{v}^\intercal G^{\boldsymbol{x}} \boldsymbol{v}
        = \sum_{j,k=1}^d \boldsymbol{v}_j \boldsymbol{v}_k \big(P_{\mathcal{V}_{\boldsymbol{x}}}b_j, P_{\mathcal{V}_{\boldsymbol{x}}}b_k\big)_{\mathcal{V}}
        = \Big(P_{\mathcal{V}_{\boldsymbol{x}}} \big({\textstyle \sum_{j=1}^d \boldsymbol{v}_j b_j}\big), P_{\mathcal{V}_{\boldsymbol{x}}} \big({\textstyle \sum_{k=1}^d \boldsymbol{v}_k b_k}\big)\Big)_{\mathcal{V}}
        = \|P_{\mathcal{V}_{\boldsymbol{x}}} v\|_{\mathcal{V}}^2 .
    $$
    Moreover, since $\{b_j\}_{j=1,\ldots,d}$ is a $\mathcal{V}$-orthonormal basis, it holds that $\boldsymbol{v}^\intercal\boldsymbol{v} = \|v\|_{\mathcal{V}}^2$. 
    Now, recall that
    \begin{align}
        0\preceq G^{\boldsymbol{x}} \preceq G^{\boldsymbol{y}} \preceq I
        \quad&\Leftrightarrow\quad
        \forall \boldsymbol{v}\in\mathbb{R}^d: 0 \le \boldsymbol{v}^\intercal G^{\boldsymbol{x}} \boldsymbol{v} \le \boldsymbol{v}^\intercal G^{\boldsymbol{y}} \boldsymbol{v} \le \boldsymbol{v}^\intercal \boldsymbol{v} \\
        \quad&\Leftrightarrow\quad
        \forall v\in\mathcal{V}_d: 0 \le \|P_{\mathcal{V}_{\boldsymbol{x}}} v\|_{\mathcal{V}}^2 \le \|P_{\mathcal{V}_{\boldsymbol{y}}} v\|_{\mathcal{V}}^2 \le \| v\|_{\mathcal{V}}^2
        \ ,
    \end{align}
    which is true, since $\mathcal{V}_{\boldsymbol{x}} \subseteq \mathcal{V}_{\boldsymbol{y}} \subseteq\mathcal{V}$.
\end{proof}

The above lemma will be useful in multiple places and can be used directly to derive the subsequent bounds on $\tau(\boldsymbol{x})$.

\begin{lemma}
    \label{lem:mu_is_gramian_error}
    It holds that $\mu(\boldsymbol{x})
    = \big(1 - \|I-G^{\boldsymbol{x}}\|_2\big)^{-1/2}
    $
    and
    $$
        \min\braces{2(1 - \mu(\boldsymbol{x})^{-2}), 1}
        \le \tau(\boldsymbol{x})
        \le \min\braces{(1 + \sqrt{d})(1 - \mu(\boldsymbol{x})^{-2}), 1} .
    $$
\end{lemma}
\begin{proof}
    Lemma~\ref{lem:gramian_ordering} implies
    $0 \preceq I-G^{\boldsymbol{x}} \preceq I$ and thus
    $
        \norm{I - G^{\boldsymbol{x}}}_2
        = \lambda_{\mathrm{max}}(I - G^{\boldsymbol{x}})
        = 1 - \lambda_{\mathrm{min}}(G^{\boldsymbol{x}})
        = 1 - \mu(\boldsymbol{x})^{-2}
    $.
    The second claim follows from the norm equivalence $\norm{I - G^{\boldsymbol{x}}}_2 \le \norm{I - G^{\boldsymbol{x}}}_{\mathrm{Fro}} \le \sqrt{d} \norm{I - G^{\boldsymbol{x}}}_2$.
\end{proof}

As a consequence, for a fixed space $\mcal{V}_d$, minimising the factor $\tau(\boldsymbol{x})$ in Theorem~\ref{thm:error_bound_noiseless} is equivalent to minimising $\mu(\boldsymbol{x})$.
Another consequence of Lemma~\ref{lem:gramian_ordering} is the subsequent qualitative approximation guarantees.

\begin{minipage}{\textwidth}
\begin{proposition}
\label{prop:mu_properties}
    For any $\boldsymbol{x}\in\mcal{X}^n$, the following properties are satisfied.
    \begin{enumerate}[label=(\roman*)]
        \item\label{prop:mu_properties:i}
            $\mu(\boldsymbol{x}) \ge 1$.
        \item\label{prop:mu_properties:ii}
            $\mu(\boldsymbol{x}) = \infty$ when $n < d$.
        \item\label{prop:mu_properties:iii}
            $\mu(\boldsymbol{x}) \le \big(1 - \sum_{j=1}^d \norm{(I-P_{\mcal{V}_{\boldsymbol{x}}}) b_j}_{\mcal{V}}^2\big)^{-1/2} = \big(1 - \operatorname{tr}(I - G^{\boldsymbol{x}})\big)^{-1/2}$ when $\operatorname{tr}(G^{\boldsymbol{x}}) \ge d-1$.
        \item\label{prop:mu_properties:iv}
            For every $\mu_\star > 1$ there exists an $n\in\mbb{N}$ and a sample $\boldsymbol{x}\in\mcal{X}^n$ such that $ \mu(\boldsymbol{x}) \le \mu_\star$.
    \end{enumerate}
\end{proposition}
\end{minipage}
\begin{proof}
    Property~\ref{prop:mu_properties:i} follows from Lemma~\ref{lem:mu_lmin} and Lemma~\ref{lem:gramian_ordering}, since $\mu(\boldsymbol{x}) = \lambda_{\mathrm{min}}(G^{\boldsymbol{x}})^{-1/2} \ge \lambda_{\mathrm{max}}(G^{\boldsymbol{x}})^{-1/2} \ge 1$.
    Claim~\ref{prop:mu_properties:ii} follows from Lemma~\ref{lem:mu_lmin}, since $b(\boldsymbol{x})\in\mbb{R}^{d\times n}$ and therefore $G^{\boldsymbol{x}}$ is at most of rank $n<d$.
    To show~\ref{prop:mu_properties:iii}, let $S(\mcal{V}_d) := \braces{v\in\mcal{V}_d : \norm{v}_{\mcal{V}} = 1}$ and recall that
    $$
        \mu(\boldsymbol{x})^{-2}
        = \min_{v\in S(\mathcal{V}_d)} \|P_{\mcal{V}_{\boldsymbol{x}}}v\|_{\mathcal{V}}^2 
        = 1 - \max_{v\in S(\mathcal{V}_d)} \|(I - P_{\mcal{V}_{\boldsymbol{x}}}) v\|_{\mathcal{V}}^2
        = 1 - \max_{v\in S(\mathcal{V})} \|(I - P_{\mcal{V}_{\boldsymbol{x}}}) P_{\mcal{V}_d} v\|_{\mathcal{V}}^2
        = 1 - \|(I - P_{\mcal{V}_{\boldsymbol{x}}})P_{\mcal{V}_d}\|_{\mathcal{V} \to \mathcal{V} }^2 .
    $$
    Bounding the operator norm with the Hilbert--Schmidt norm yields
    $$
        \mu(\boldsymbol{x})^{-2}
        \ge 1 - \norm{(I-P_{\mcal{V}_{\boldsymbol{x}}}) P_{\mcal{V}_d}}_{\mathrm{HS}}^2
        = 1 - \sum_{j=1}^d \norm{(I-P_{\mcal{V}_{\boldsymbol{x}}}) b_j}_{\mcal{V}}^2
        = \sum_{j=1}^d \norm{P_{\mcal{V}_{\boldsymbol{x}}} b_j}_{\mcal{V}}^2 - d + 1
        = \operatorname{tr}(G^{\boldsymbol{x}} - I) + 1 .
    $$
    To prove~\ref{prop:mu_properties:iv}, note that we can write~\eqref{eq:def:mu} as
    $$
        \mu(\boldsymbol{x})
        = \max_{v\in S(\mcal{V}_d)} \frac{1}{\sqrt{1 - \norm{v - P_{\mcal{V}_{\boldsymbol{x}}} v}_{\mcal{V}}^2}}.
    $$
    Consequently it suffices to show that for every $\varepsilon > 0$ there exists a sample $\boldsymbol{x}$ such that $\norm{v - P_{\mcal{V}_{\boldsymbol{x}}} v}_{\mcal{V}} \le \varepsilon$ for every $v\in S(\mcal{V}_d)$.
    Once this is proven, the claim follows by choosing $\varepsilon = (1 - \mu_\star^{-2})^{1/2}$.
    Since $\mcal{V}_d$ is finite-dimensional, the unit sphere $S(\mcal{V}_d)$ is compact, and there exists an $\frac\varepsilon2$-covering with centres $\braces{c_1, \ldots, c_n}$.
    This means that for every $v\in S(\mcal{V}_d)$ there exists a centre $c_{i(v)}$ such that
    $$
        \norm{v - c_{i(v)}} \le \tfrac{\varepsilon}{2} .
    $$
    By the Moore--Aronszajn theorem we know that $\mcal{V} = \overline{\operatorname{span}\braces{k(x, \bullet) : x\in\mcal{X}}}^{\norm{\bullet}_{\mcal{V}}}$.
    This implies that every $c_i$ can be approximated by
    $$
        \tilde{c}_i := \sum_{j=1}^{m_i} c_{i,j} k(x_{i,j}, \bullet)
    $$
    for some choice of $x_{i,j}\in\mcal{X}$ and $m_i$, 
    with an error of $\norm{c_i - \tilde{c}_i}_{\mcal{V}} \le \tfrac{\varepsilon}2$.
    We now define the sample $\boldsymbol{x} := (x_{i,j})_{{1\le i \le n , 1\le j\le m_i}}$ and observe that $\tilde{c}_i\in\mcal{V}_{\boldsymbol{x}}$ for every $i=1,\ldots,n$.
    Therefore
    $$
        \norm{v - P_{\mcal{V}_{\boldsymbol{x}}} v}
        \le \norm{v - \tilde{c}_{i(v)}}
        \le \norm{v - c_{i(v)}} + \norm{c_{i(v)} - \tilde{c}_{i(v)}}
        \le \varepsilon .
    $$
    This concludes the proof.
\end{proof}

The preceding proposition guarantees that we can always find a point set such that the suboptimality factor $\mu(\boldsymbol{x})$ deceeds a prespecified threshold $\mu_\star$.
However, it does not tell us how large this point set needs to be, nor does it provide a strategy to generate it.
This question is addressed by the subsequent theorem.

To prove this result, we assume that $\mcal{V}$ is compactly embedded in $L^2(\nu)$  for some probability measure $\nu$ on $\mcal{X}$.
Then, the integral operator
$$
    \Sigma : v \mapsto \int_{\mcal{X}} k(\bullet, y) v(y) \dx[\nu](y),
$$
is compact from $L^2(\nu)$ to $L^2(\nu)$ and admits a spectral decomposition
$$
    \Sigma = \sum_{m\in \mathbb{N}} \lambda_m \phi_m ( \phi_m , \bullet)_{L^2(\nu)},
$$
where $\{\phi_l\}_{l\in \mathbb{N}}$ is an $L^2(\nu)$-orthonormal and $\mcal{V}$-orthogonal system.
This corresponds to a decomposition of the kernel $k(x,y) = \sum_{m \in \mathbb{N}} \lambda_m \phi_m(x) \phi_m(y),$ which converges pointwise.
Moreover, we assume that $\Sigma$ is a trace-class (or nuclear) operator, i.e.\ that $\int K(x) \dx[\nu(x)] = \sum_{m\in \mathbb{N}} \lambda_m <\infty$.
Then it can be shown~\cite{belhadji2020kernel} that 
\begin{equation}
    \det(K(\boldsymbol{x})) \dx[\nu^{\otimes n}](\boldsymbol{x})
\end{equation}
is a finite measure, which can be used as an (unnormalised) sampling measure for $\boldsymbol{x}$.

\begin{minipage}{\textwidth}
\begin{theorem}[Proposition~5 and Theorem~6 in~\cite{belhadji2020kernel}]
\label{thm:ayoub}
    Assume that $\mcal{V}$ is compactly embedded in $L^2(\nu)$ for some probability measure $\nu$ on $\mcal{X}$ and that $\int K(x) \dx[\nu(x)] <\infty$.
    Moreover, suppose that $\lambda = (\lambda_m)_{m\in \mathbb{N}}$ satisfies either
    \begin{itemize}
        \item $\lambda_m = \lambda_m^{\mathrm{alg}} \asymp m^{-2s}$ for some $s > 1/2$ (e.g. Sobolev spaces of finite smoothness) or
        \item $\lambda_m = \lambda_m^{\mathrm{exp}} \asymp \alpha^m$ for some $\alpha\in(0,1)$ (e.g.\ the Gaussian kernel),
    \end{itemize}
    and define $B := B^{\mathrm{alg}}(s) \asymp (1+ \frac{1}{2s-1})^{2s}$ if $\lambda = \lambda^{\mathrm{alg}}$ or $B := B^{\mathrm{exp}}(\alpha) \asymp \frac{\alpha}{1-\alpha}$ if $\lambda = \lambda^{\mathrm{exp}}$. \\
    Let $\boldsymbol{x}\in \mcal{X}^n$ be drawn from the (unnormalised) \emph{continuous volume sampling} distribution
    \begin{equation}
        \det(K(\boldsymbol{x})) \dx[\nu^{\otimes n}](\boldsymbol{x}) . \label{eq:cvs}
    \end{equation}
    Then for any $v \in \Sigma^{1/2}\mcal{V}$, it holds that
    $$
        \mbb{E}\bracs*{\norm{v - P_{\mcal{V}_{\boldsymbol{x}}} v}_{\mcal{V}}^2}
        \le (2 + B) \lambda_n \Vert \Sigma^{-1/2} v \Vert_{\mcal{V}}^2 .
    $$
\end{theorem}
\end{minipage}

\begin{corollary}\label{cor:ayoub}
    Under the assumptions of Theorem~\ref{thm:ayoub}, for $\mcal{V}_d \subseteq \Sigma^{1/2}\mcal{V}$, and for any $\varepsilon\in(0,1)$, it holds that
    $$
        \mbb{P}\bracs*{\mu(\boldsymbol{x}) \ge \frac{1}{\sqrt{1-\varepsilon^2}}}
        \le \frac{(2 + B) dC_d^2 \lambda_n}{\varepsilon^2} ,
    $$ 
    where $C_d := \max_{v\in\mcal{V}_d} \frac{\norm{\Sigma^{-1/2}v}_{\mcal{V}}}{\norm{v}_{\mcal{V}}}$.
    Moreover, if $\mcal{V}_d = \operatorname{span}\braces{\phi_1,\ldots,\phi_d}$, then $C_d^2 = \lambda_d^{-1}$.
\end{corollary}
\begin{proof}
    From Proposition~\ref{prop:mu_properties} it follows that
    $\mu(\boldsymbol{x}) \ge \frac{1}{\sqrt{1-\varepsilon^2}}$ implies $1 - \sum_{j=1}^d  \norm{(I-P_{\mathcal{V}_{\boldsymbol{x}}}) b_j}_{\mathcal{V}}^2 \le 1 - \varepsilon^2$, which is equivalent to $\varepsilon^2\le \sum_{j=1}^d \norm{(I-P_{\mathcal{V}_{\boldsymbol{x}}}) b_j}_{\mathcal{V}}^2$.
    Markov's inequality thus yields
    $$
        \mathbb{P}\Big[
            \mu(\boldsymbol{x}) \ge \tfrac{1}{\sqrt{1-\varepsilon^2}}
        \Big]
        \le \mathbb{P}\Big[
            {\textstyle \sum_{j=1}^d} \norm{(I-P_{\mathcal{V}_{\boldsymbol{x}}}) b_j}_{\mathcal{V}}^2
            \ge \varepsilon^2
        \Big]
        \le \frac{\sum_{j=1}^d \mathbb{E}\big[\norm{(I-P_{\mathcal{V}_{\boldsymbol{x}}}) b_j}_{\mathcal{V}}^2\big]}{\varepsilon^2} .
    $$
    Using Theorem~\ref{thm:ayoub}, it follows that
    $$
        \mathbb{P}\Big[
            \mu(\boldsymbol{x}) \ge \tfrac{1}{\sqrt{1-\varepsilon^2}}
        \Big]
        \le \frac{(2 + B) d C_d^2 \lambda_n}{\varepsilon^2} .
    $$
    To prove the second claim, recall that $\Sigma^{-1/2} v = \sum_{l\in\mbb{N}} \lambda_l^{-1/2}(v, \phi_l)_{L^2}\phi_l$ and therefore, 
    $$
        \norm{v}_{\mcal{V}}^2 = \sum_{l\in\mbb{N}} \frac{(v, \phi_l)_{L^2}^2}{\lambda_l}
        \qquad\text{and}\qquad
        \norm{\Sigma^{-1/2} v}_{\mcal{V}}^2 = \sum_{l\in\mbb{N}} \frac{(v, \phi_l)_{L^2}^2}{\lambda_l^2} .
    $$
    Defining the matrix $A_d := \operatorname{diag}(\lambda_1, \ldots, \lambda_d)$ and expanding $v\in\mcal{V}_d$ in terms of $\phi_{l}$, it thus holds that
    \begin{align}
        C_d^2 &= \max_{c\in\mbb{R}^d} \frac{c^\intercal A_d^{-2} c}{c^\intercal A_d^{-1} c} = \lambda_{\mathrm{max}}(A_d^{-1}) = \lambda_{\mathrm{min}}(A_d)^{-1} = \lambda_d^{-1} .
        \qedhere
    \end{align}
\end{proof}

Corollary~\ref{cor:ayoub} guarantees that
$$
    \mbb{P}\bracs*{\mu(\boldsymbol{x}) \le \mu_\star}
    \ge 1 - (2 + B) \lambda_n 
    \tfrac{\mu_\star^2}{\mu_\star^2 - 1} d C_d^2 . 
$$
For any $n$ for which this lower bound is positive, there exists a point set $\boldsymbol{x}$ that satisfies $\mu(\boldsymbol{x}) \le \mu_\star$.
In particular, when $d > 1$ and $\mcal{V}_d = \operatorname{span}\{\phi_1 , \ldots, \phi_d\}$ is spanned by the first $d$ spectral basis functions of the RKHS $\mathcal{V}$, we know that
\begin{equation}
\label{eq:sample_size_bound}
    \mathbb{P}\bracs*{\mu(\boldsymbol{x}) \le \mu_\star}
    > 0
    \quad\Leftarrow\quad
    \begin{cases}
        n \ge d^{1 + 1/(2s)} \left((2 + B^{\mathrm{alg}}(s))\tfrac{\mu_\star^2}{\mu_\star^2 - 1}\right)^{1/(2s)}
            &\text{and }  \lambda_n = n^{-2s}, \text{ or} \\
        n \ge d + \tfrac{\ln(d)}{\ln(1 / \alpha)} + \tfrac{1}{\ln(1/\alpha)}\ln\left((2 + B^{\mathrm{exp}}(\alpha))\tfrac{\mu_\star^2}{\mu_\star^2 - 1}\right)
            &\text{and } \lambda_n = \alpha^n .
    \end{cases}
\end{equation}
This not only guarantees the existence of a sample $\boldsymbol{x}$ of size (almost) $n\propto d$ satisfying $\mu(\boldsymbol{x}) \le \mu_\star$ but also provides a strategy for generating such a sample by drawing $\boldsymbol{x}$ according to~\eqref{eq:cvs}.
However, this bound only holds under strong assumptions on $\mathcal{V}_d$ as the growth of $C_d$ must be controlled.
As an example suppose that $\mathcal{V} = H^k([0,1])$.
Then $\lambda_m \asymp m^{-2k}$ (cf.~\cite{sarazin_2023}) and
$$
    \norm{\Sigma^{-1/2} v}_{\mathcal{V}}^2
    = \norm{\Sigma^{-1/2} v}_{H^k([0,1])}^2
    = \sum_{m\in\mathbb{N}} \frac{(v, \phi_m)_{L^2}^2}{\lambda_m^2}
    \asymp \sum_{m\in\mathbb{N}} \frac{(v, \phi_m)_{L^2}^2}{m^{-4k}}
    \asymp \norm{v}_{H^{2k}([0,1])}^2 \;.
$$
This implies
$$
    C_d \asymp \max_{v\in\mathcal{V}_d} \frac{\|v\|_{H^{2k}([0,1])}}{\|v\|_{H^{k}([0,1])}} \;,
$$
which could grow extremely fast.
Numerical examples for illustration are provided in sections~\ref{sec:experiments:h1} and~\ref{sec:experiments:h10}.

\begin{remark}
    The problem of minimising $\mu$ has also been considered in the context of optimal sensor placement~\cite{Binev2018}.
    We can compare the sample size bound from equation~\eqref{eq:sample_size_bound} to the one provided in~\cite{Binev2018}.
    In the setting when $\mcal{V} = H^1_0(0,1)$ and when $\mcal{V}_d$ is the span of the first $d$ Fourier basis elements, \Cite[Theorem~2.3]{Binev2018}  guarantees the existence of a sample $\boldsymbol{x}\in\mcal{X}^n$ with $\mu(\boldsymbol{x})\le\mu_\star$ and $n \propto d$.
    Theorem~\ref{thm:ayoub} extends this bound to more general RKHS (such as $H^s(\mcal{X})$ for $\mcal{X}\subseteq \mbb{R}^d$ and $s>\frac{d}2$) under suitable regularity assumptions and for any choice of subspace $\mcal{V}_d$ satisfying assumptions of Corollary \ref{cor:ayoub}.
\end{remark}

\section{Sampling the points}
\label{sec:sampling}

Minimising $\mu(\boldsymbol{x}) = \lambda_{\mathrm{min}}(G^{\boldsymbol{x}})^{-1/2}$ over $\boldsymbol{x}\in\mcal{X}^n$ for a fixed sample size $n$ is a non-trivial, non-convex problem and potentially NP-hard.\footnote{In case $\mcal{X}$ has finite cardinality and $K(\boldsymbol{x}) = I$, minimising $\mu(\boldsymbol{x})$, i.e.\ maximising $\lambda_{\mathrm{min}}(G^{\boldsymbol{x}})$, is equivalent to an E-optimal design problem, which is known to be NP-hard~\cite{civril2009}.}
However, the Theorem~\ref{thm:ayoub} from the preceding section shows that exact optimisation is not necessary and that good point sets $\boldsymbol{x}\in\mcal{X}^n$ can already be obtained by sampling.

Building on this insight, we propose to replace the minimisation of $\mu$ with a related sampling problem.
We suppose that there exists some measure $\rho$ for which the integral $Z := \int\det(G^{\boldsymbol{x}}) \dx[\rho^{\otimes n}(\boldsymbol{x})]$ is finite.
Then
$$
    Z^{-1} \det(G^{\boldsymbol{x}}) \dx[\rho^{\otimes n}(\boldsymbol{x})]
$$
constitutes a probability measure 
and we can draw $\boldsymbol{x}$ according to this measure.

Compared to classical continuous volume sampling, where $\boldsymbol{x}$ is drawn from $\det(K(\boldsymbol{x})) \dx[\nu^{\otimes n}(\boldsymbol{x})]$, sampling from $\det(G^{\boldsymbol{x}}) \dx[\rho^{\otimes n}(\boldsymbol{x})]$ has two major advantages.
Firstly, the theory of continuous volume sampling (see Corollary~\ref{cor:ayoub}) requires a strong integrability condition $\int K(x) \dx[\nu(x)] < \infty$.
Secondly, the theory of continuous volume sampling depends on the embedding constant $C_d$, which is hard to compute and potentially very large.

To see why sampling from $\det(G^{\boldsymbol{x}}) \dx[\rho^{\otimes n}(\boldsymbol{x})]$ is a sensible idea, we first argue that $\boldsymbol{x}\mapsto\det(G^{\boldsymbol{x}})$ is a good surrogate for minimising $\mu(\boldsymbol{x}) = \lambda_{\mathrm{min}}(G^{\boldsymbol{x}})^{-1/2}$.
Recall that Lemma~\ref{lem:gramian_ordering} ensures that $0 \le \lambda_{\mathrm{min}}(G^{\boldsymbol{x}}) \le \lambda_{\mathrm{max}}(G^{\boldsymbol{x}})\le 1$ and therefore
$$
    \det(G^{\boldsymbol{x}})
    \le \lambda_{\mathrm{min}}(G^{\boldsymbol{x}})
    \le \det(G^{\boldsymbol{x}})^{1/d} .
$$
Hence, the value of the determinant provides a bound on the smallest eigenvalue, which in turn bounds $\mu$.
Replacing the maximisation of the surrogate $\boldsymbol{x}\mapsto\det(G^{\boldsymbol{x}})$ with a sample from $\det(G^{\boldsymbol{x}}) \dx[\rho^{\otimes n}(\boldsymbol{x})]$
is a classical idea in global optimisation and, in principle, justified by the Paley--Zygmund inequality, which provides lower bounds for the probability that a sample lies in a certain superlevel set of the objective function.
The reason for using the surrogate in the first place, even though it does not reduce the complexity of the optimisation problem,\footnote{In case $\mcal{X}$ has finite cardinality and $K(\boldsymbol{x}) = I$, maximising $\det(G^{\boldsymbol{x}})$ is equivalent to a D-optimal design problem, which is known to be NP-hard~\cite{Welch1982}.} is that it simplifies the sampling.

It remains to find a measure $\rho$ for which the normalisation constant $Z = \int \det(G^{\boldsymbol{x}}) \dx[\rho^{\otimes n}(\boldsymbol{x})]$ is finite.
This question is addressed in the subsequent lemma.


\begin{lemma}
\label{lem:det_is_pdf}
   Let $\rho$ be a finite measure on $\mcal{X}$ and $G^{\boldsymbol{x}}$ given by~\eqref{eq:def:Gramian_x}.
   The function $\mcal{X}^n\ni\boldsymbol{x} \mapsto \det(G^{\boldsymbol{x}})$ is non-negative and $\rho^{\otimes n}$-integrable.
\end{lemma}
\begin{proof}
    To show non-negativity, recall that $K(\boldsymbol{x})$ is a kernel matrix, and thus positive semidefinite by definition.
    This implies that $K(\boldsymbol{x})^+$ and therefore $ G^{\boldsymbol{x}} = b(\boldsymbol{x}) K(\boldsymbol{x})^+ b(\boldsymbol{x})^\intercal$ are positive semidefinite matrices. 
    This proves that $\det(G^{\boldsymbol{x}})$ is non-negative.
    To show integrability, recall that Lemma~\ref{lem:gramian_ordering} ensures $0 \le \lambda_{\mathrm{min}}(G^{\boldsymbol{x}}) \le \lambda_{\mathrm{max}}(G^{\boldsymbol{x}})\le 1$ and therefore
    $
        \det(G^{\boldsymbol{x}}) \le 1
    $.
    This implies
    \begin{align}
        \int\det(G^{\boldsymbol{x}})\,\mathrm{d}\rho^{\otimes n}(\boldsymbol{x})
        &\le \rho(\mcal{X})^n
        < \infty .
        \qedhere
    \end{align}
\end{proof}

\begin{minipage}{\textwidth}
\begin{remark}
    \label{rmk:choice}
    Note that the definition of the sampling density $\det(G^{\boldsymbol{x}})\dx[\rho^{\otimes n}(\boldsymbol{x})]$ depends on the arbitrary measure $\rho$, which does not appear in the definition of $\mu(\boldsymbol{x})$.
    This introduces an artificial preference of some $\boldsymbol{x}$ over others, which does not result from maximising $\det(G^{\boldsymbol{x}})$.
    There are indeed uncountably many choices for $\rho$, and the optimal choice would be as close as possible to the discrete measures $\rho = \delta_{\boldsymbol{x}^\star}$ where $\boldsymbol{x}^\star$ is a minimiser of $\mu(\boldsymbol{x}^\star)$.
    (Optimal in the sense that any sample from $\det(G^{\boldsymbol{x}})\dx[\rho]$ would minimise $\mu$.)
\end{remark}
\end{minipage}

In the following, we consider a generalisation of the above sampling distribution of the form
\begin{equation}
\label{eq:our-distribution}
    \det(G^{\boldsymbol{x}})\dx[(\rho_1 \otimes \dots \otimes \rho_n)](\boldsymbol{x})
\end{equation}
where $\{\rho_k\}_{k \in \mathbb{N}}$ is a sequence of possibly distinct measures.

\subsection{Drawing a sample of size \texorpdfstring{$\boldsymbol{d}$}{d}}
 

Assuming $\det(K(\boldsymbol{x})) \ne 0$, we can factorise the determinant as
$$
    \det(G^{\boldsymbol{x}}) = \det(b(\boldsymbol{x})K(\boldsymbol{x})^+ b(\boldsymbol{x})^\intercal)
    = \frac{\det(b(\boldsymbol{x})^\intercal b(\boldsymbol{x}))}{\det(K(\boldsymbol{x}))}
    = \frac{\det(K_d(\boldsymbol{x}))}{\det(K(\boldsymbol{x}))} .
$$
The density is a ratio of densities of determinantal point processes, and the resulting distribution appears as a competition between a repulsive point process associated with the kernel   $k_d$ and an attractive point process associated with the kernel $k$.
We now decompose $\boldsymbol{x}\in\mcal{X}^d$ as $\boldsymbol{x} = \boldsymbol{x}_{<d} \oplus x_d = \boldsymbol{x}_{<d-1} \oplus x_{d-1} \oplus x_d = \ldots$ and write
$$
    K(\boldsymbol{x}) = \begin{pmatrix}
        K(\boldsymbol{x}_{<d}) & k(\boldsymbol{x}_{<d}, x_d) \\
        k(x_d, \boldsymbol{x}_{<d}) & K(x_d)
    \end{pmatrix} .
$$
Since $K(x_d) > 0$, it follows that $\operatorname{ker}(K(\boldsymbol{x}_{<d})) = 0$, and we can use Schur's formula to compute
$$
    \det(K(\boldsymbol x))
    = \det(K(\boldsymbol{x}_{<d})) (K(x_d) - k(x_d, \boldsymbol{x}_{<d}) K(\boldsymbol{x}_{<d})^{-1} k(\boldsymbol{x}_{<d}, x_d)) .
$$
Recursive application of this relation to both kernels yields the factorisation
\begin{equation}
\label{eq:product_formula}
    \det(G^{\boldsymbol{x}})
    = \frac{\det(K_d(\boldsymbol{x}))}{\det(K(\boldsymbol{x}))}
    = \prod_{i=1}^{d} q(x_i \mid \boldsymbol{x}_{<i})
    \qquad\text{with}\qquad
    q(x \mid \boldsymbol{x})
    := \frac{
        K_d(x)
        - k_d(x, \boldsymbol{x}) K_d(\boldsymbol{x})^{-1} k_d(\boldsymbol{x}, x)
    }{
        K(x)
        - k(x, \boldsymbol{x}) K(\boldsymbol{x})^{-1} k(\boldsymbol{x}, x)
    } .
\end{equation}


Motivated by the chain rule for projection determinantal point processes~\cite{Hough2006}, it seems natural to expect the decomposition~\eqref{eq:product_formula} to yield the conditional marginals of $\det(G^{\boldsymbol{x}})$.
Note, however, that we can factorise any function $f(\boldsymbol{x}) = f_1(x_1) f_2(x_2 \mid x_1) \cdots f_d(x_d\mid x_1,\ldots,x_{d-1})$, such that all factors but the last one have arbitrary (positive) values.
In particular, the factors $f_i(x_i \mid \boldsymbol{x}_{<i})$ do not have to be (conditional) marginals.
The subsequent proposition quantifies how close the decomposition~\eqref{eq:product_formula} is to a product of (conditional) marginals.

\begin{minipage}{\textwidth}
\begin{proposition}
\label{prop:marginal}
    Suppose that $0 < c_i\le C_i <\infty$ for $i=1,\ldots,d$ and $\rho_1,\ldots,\rho_d$ are such that the functions
    $$
        Z_i : \mcal{X}^{i-1} \to \mbb{R}
        \qquad
        Z_i : \boldsymbol{x} \mapsto \int q(x\mid\boldsymbol{x})\dx[\rho_i(x)]
    $$
    are bounded almost surely by $Z_i \in [c_i, C_i]$ for all $i=1,\ldots,d$.
    Then, the probability density
    $$
        p(\boldsymbol{x})
        := \tfrac{1}{Z} \det(G^{\boldsymbol{x}})
        \qquad\text{with}\qquad
        Z := \int \det(G^{\boldsymbol{x}}) \dx[(\rho_1\otimes\cdots\otimes\rho_d)(\boldsymbol{x})]
    $$
    can be written in terms of the conditional probabilities $p(\boldsymbol{x}) = p(x_1) p(x_{2}\mid \boldsymbol{x}_{<2}) \cdots p(x_d \mid \boldsymbol{x}_{<d})$ with
    $$
        p(x_i \mid \boldsymbol{x}_{<i}) \in \tfrac1{Z_i(\boldsymbol{x}_{<i})} q(x_i \mid \boldsymbol{x}_{<i}) \cdot \Big[\big({\textstyle\prod_{j=i}^d \tfrac{c_j}{C_j}}\big), \big({\textstyle\prod_{j=i}^d \tfrac{C_j}{c_j}}\big)\Big] .
    $$
    Moreover, if $\rho_1=\rho_2=\cdots=\rho_d = \rho$, and $p_i$ is the marginal density of any component $x_i$ of $\boldsymbol{x} \sim \det(G^{\boldsymbol{x}}) \dx[(\rho_1 \otimes \dots \otimes \rho_n)](\boldsymbol{x})$, 
    then $p_i(x_i) \asymp \frac{K_d(x_i)}{K(x_i)}$.
\end{proposition}
\end{minipage}

\begin{proof}
    To compute the conditional probabilities, recall that
    $$
        p(x_i\mid \boldsymbol{x}_{<i})
        = \frac{p(\boldsymbol{x}_{\le i})}{\int p(\boldsymbol{x}_{\le i}) \dx[\rho_i(x_i)]}
        = \frac{
            \int p(\boldsymbol{x}) \dx[(\rho_{i+1}\otimes\cdots\otimes\rho_{d})(\boldsymbol{x}_{>i})]
        }{
            \int p(\boldsymbol{x}) \dx[(\rho_{i}\otimes\cdots\otimes\rho_{d})(\boldsymbol{x}_{\ge i})]
        }
        .
    $$
    Now recall that by~\eqref{eq:def:Gramian_x}, $\det(G^{\boldsymbol{x}}) = \det(b(\boldsymbol{x})K(\boldsymbol{x})^+ b(\boldsymbol{x})^\intercal)$.
    Since $\det(K(\boldsymbol{x})) = 0$ implies $\det\pars{K(\boldsymbol{x})^+} = 0$ and thus $\det\pars{G^{\boldsymbol{x}}} = 0$, sampling from $\det(G^{\boldsymbol{x}})$ guarantees that $\det(K(\boldsymbol{x})) \ne 0$ almost surely.
    We can thus use the factorisation of the determinantal density from equation~\eqref{eq:product_formula} and write
    $$
        \int p(\boldsymbol{x}) \dx[(\rho_{i}\otimes\cdots\otimes\rho_{d})(\boldsymbol{x}_{\ge i})]
        = \frac1Z \prod_{j=1}^{i-1} q(x_j \mid \boldsymbol{x}_{<j})
        \int \prod_{j=i}^{d} q(x_j \mid \boldsymbol{x}_{<j}) \dx[(\rho_{i}\otimes\cdots\otimes\rho_{d})(\boldsymbol{x}_{\ge i})]
        .
    $$
    Using Fubini's theorem and the assumption that the $Z_i$-functions are bounded, we can bound the right-hand side by
    \begin{align}
        \int \prod_{j=i}^{d} q(x_j \mid \boldsymbol{x}_{<j}) \dx[(\rho_{i}\otimes\cdots\otimes\rho_{d})(\boldsymbol{x}_{\ge i})]
        &\in \bigg[Z_{i}(\boldsymbol{x}_{<i})\prod_{j=i+1}^d c_j, Z_{i}(\boldsymbol{x}_{<i}) \prod_{j=i+1}^d C_j\bigg]
        \subseteq \bigg[\prod_{j=i}^d c_j, \prod_{j=i}^d C_j\bigg]
        .
    \end{align}
    Inserting these bounds into the definition of the conditional probability yields the lower bound
    $$
        p(x_i\mid \boldsymbol{x}_{<i})
        \ge \frac{
            \frac1Z \prod_{j=1}^i q(x_j\mid\boldsymbol{x}_{<j}) \prod_{j=i+1}^d c_{j}
        }{
            \frac1Z \prod_{j=1}^{i-1} q(x_j\mid\boldsymbol{x}_{<j}) Z_i(\boldsymbol{x}_{<i}) \prod_{j=i+1}^d C_{j}
        }
        = \pars*{\prod_{j=i+1}^d \frac{c_j}{C_j}} \frac{q(x_i\mid\boldsymbol{x}_{<i})}{Z_i(\boldsymbol{x}_{<i})}
    $$
    as well as the upper bound
    $$
        p(x_i\mid \boldsymbol{x}_{<i})
        \le \frac{
            \frac1Z \prod_{j=1}^i q(x_j\mid\boldsymbol{x}_{<j}) \prod_{j=i+1}^d C_{j}
        }{
            \frac1Z \prod_{j=1}^{i-1} q(x_j\mid\boldsymbol{x}_{<j}) Z_i(\boldsymbol{x}_{<i}) \prod_{j=i+1}^d c_{j}
        }
        = \pars*{\prod_{j=i+1}^d \frac{C_j}{c_j}} \frac{q(x_i\mid\boldsymbol{x}_{<i})}{Z_i(\boldsymbol{x}_{<i})}
        .
    $$
    Now suppose that $\rho := \rho_1 = \ldots = \rho_d$ and observe that the factorisation of $p$ into conditional probabilities yields $p_1(x_1) \asymp \frac{K_d(x_1)}{K(x_1)}\dx[\rho(x_1)]$.
    Since the probability $\det(G^{\boldsymbol x})\dx[\rho^{\otimes d}(\boldsymbol{x})]$ is invariant under permutation of the points, the same bound for the marginal probability density holds for all $i=1,\ldots,d$.
\end{proof}

Proposition~\ref{prop:marginal} provides a way to approximately draw $\boldsymbol{x}\in\mcal{X}^d$ from $\det(G^{\boldsymbol{x}})\dx[(\rho_1\otimes\cdots\otimes\rho_d)(\boldsymbol{x})]$ by sequential conditional sampling, first drawing $x_1$, then $x_2$ given $x_1$, then $x_3$ given $x_1$ and $x_2$, and so on.

\subsection{Drawing a sample of size \texorpdfstring{$\boldsymbol{n>d}$}{n>d}}

For any $\boldsymbol{x}\in\mcal{X}^{n-1}$ and $y\in\mcal{X}$, we can decompose
$$
    P_{\mcal{V}_{\boldsymbol{x} \oplus y}} v = P_{\mcal{V}_{\boldsymbol{x}}} v + (v, \omega_y)_{\mcal{V}} \omega_y ,
    \qquad v\in\mcal{V} ,
$$
with $\omega_y := \frac{\tilde\omega_y}{\norm{\tilde\omega_y}_{\mcal{V}}}$ and $\tilde\omega_y := (I - P_{\mcal{V}_{\boldsymbol{x}}}) k(y, \bullet)$.
The Gramian matrix can thus be written in terms of the rank-$1$ update
\begin{equation}
\label{eq:G_rk1_update}
    G^{\boldsymbol{x}\oplus y}_{jk}
    = (P_{\mcal{V}_{\boldsymbol{x} \oplus y}} b_j, P_{\mcal{V}_{\boldsymbol{x} \oplus y}} b_k)_{\mcal{V}}
    = G^{\boldsymbol{x}}_{jk} + (g_{y})_j (g_{y})_k ,
\end{equation}
with $(g_{y})_j := (\omega_y, b_j)_{\mcal{V}}$. 
Assuming $\det(G^{\boldsymbol{x}})\ne 0$, we can factorise the determinant as
\begin{equation}
\label{eq:r_conditional_pdf}
    \det(G^{\boldsymbol{x}\oplus y})=\det(G^{\boldsymbol{x}} + g_yg_y^\intercal) = (1 + r(y\mid\boldsymbol{x})) \det(G^{\boldsymbol{x}})
    \qquad\text{with}\qquad
    r(y\mid\boldsymbol{x}) := g_y^\intercal (G^{\boldsymbol{x}})^{-1} g_y
\end{equation}
via the matrix determinant lemma.
To compute this explicitly, note that $\tilde\omega_y = k(y, \bullet) - k(y, \boldsymbol{x}) K(\boldsymbol{x})^{-1} k(\boldsymbol{x}, \bullet)$, and therefore $(\tilde \omega_y, b_j)_{\mcal{V}} = b(y) - b(\boldsymbol{x}) K(\boldsymbol{x})^{-1} k(\boldsymbol{x}, y)$ and $\norm{\tilde\omega_{y}}_{\mcal{V}}^2 = K(y) - k(y, \boldsymbol{x}) K(\boldsymbol{x})^{-1} k(\boldsymbol{x}, y)$.

\begin{proposition}
\label{prop:conditional_marginal}
    Let $n>d$ and $\rho_1,\ldots,\rho_n$ be a sequence of probability measures.
    Suppose that $\boldsymbol{x}\in\mcal{X}^{n-1}$ is distributed according to $\boldsymbol{x} \sim \det(G^{\boldsymbol{x}}) \dx[(\rho_1\otimes\cdots\otimes\rho_{n-1})(\boldsymbol{x})]$ and $y\in\mcal{X}$ is distributed according to $y\sim (1+r(y\mid\boldsymbol{x})) \dx[\rho_n(y)]$.
    Then
    $$
        \boldsymbol{x}\oplus y \sim \det(G^{\boldsymbol{x}\oplus y}) \dx[(\rho_1\otimes\cdots\otimes\rho_n)(\boldsymbol{x}\oplus y)] .
    $$
\end{proposition}
\begin{proof}
    Since $\boldsymbol{x} \sim \det(G^{\boldsymbol{x}})$, it holds that $\det(G^{\boldsymbol{x}}) \ne 0$ almost surely.
    We can, therefore, factorise the determinant as described in equation~\eqref{eq:r_conditional_pdf} and compute
    $$
        \int \det(G^{\boldsymbol{x}\oplus y}) \dx[\rho_n(y)]
        = \underbrace{\int (1+r(y\mid \boldsymbol{x}) \dx[\rho_n(y)]}_{=Z(\boldsymbol{x})} \det(G^{\boldsymbol{x}}) .
    $$
    Consequently, the conditional probability density of $p(\boldsymbol{x}\oplus y) \propto \det(G^{\boldsymbol{x}\oplus y})$ given $\boldsymbol{x}$ is given by
    \begin{align}
        p(y\mid \boldsymbol{x})
        &= \frac{p(\boldsymbol{x}\oplus y)}{\int p(\boldsymbol{x}\oplus y) \dx[\rho_n(y)]}
        = \frac{1+r(y\mid\boldsymbol{x})}{Z(\boldsymbol{x})}
        . \qedhere
    \end{align}
\end{proof}

We emphasize here the usefulness of Proposition~\ref{prop:conditional_marginal}, even in the case where we aim to obtain a sample of minimal size $n = d$.
While Proposition~\ref{prop:marginal} only yields approximate marginals, Proposition~\ref{prop:conditional_marginal} provides exact ones.
Since our proposed method involves subsampling (see Section~\ref{sec:subsampling}), it is not essential to draw the initial $d$ points optimally.

\begin{remark}
\label{rmk:marginal}
    Proposition~\ref{prop:conditional_marginal} proves that the conditional density of $\det(G^{\boldsymbol{x}})$, given $\boldsymbol{x}_{<n}$, is a mixture of the uniform density and $r(x_n\mid \boldsymbol{x}_{<n})$.
    To examine the implications of this result on the sample $\boldsymbol{x} \sim \det(G^{\boldsymbol{x}})\dx[\rho^{\otimes n}(\boldsymbol{x})]$, recall that
    $
        r(x_n\mid \boldsymbol{x}_{<n})
        = g_{x_n}^\intercal (G^{\boldsymbol{x}_{<n}})^{-1} g_{x_n}
        \le \lambda_{\mathrm{min}}(G^{\boldsymbol{x}_{<n}})^{-1} \norm{g_{x_n}}_2^2
    $
    and
    $$
        \norm{g_{x_n}}_2^2
        = \frac{1}{\norm{\tilde{\omega}_{x_n}}^2_{\mcal{V}}} \sum_{j=1}^d (\tilde{\omega}_{x_n}, b_j)_{\mcal{V}}^2
        = \frac{1}{\norm{\tilde{\omega}_{x_n}}^2_{\mcal{V}}} \Big\|\sum_{j=1}^d (\tilde{\omega}_{x_n}, b_j)_{\mcal{V}} b_j\Big\|_{\mcal{V}}^2
        = \frac{\norm{P_{\mcal{V}_d} \tilde{\omega}_{x_n}}_{\mcal{V}}^2}{\norm{\tilde{\omega}_{x_n}}^2}
        \le 1
        .
    $$
    We can thus conclude from Lemma~\ref{lem:mu_lmin} that $0 \le r({x_n}\mid\boldsymbol{x}_{<n})\le \mu(\boldsymbol{x}_{<n})^2$.
    This indicates that when $\mu(\boldsymbol{x}_{<n})$ is small, the conditional density of $x_n$ contains a non-negligible uniform part.
    I.e., for large sample sizes $n$, many of the points $x_i$ will be i.i.d.\ with respect to $\rho$.
    This once again underscores the importance of the choice of $\rho$ discussed in Remark~\ref{rmk:choice}.
\end{remark}


Algorithm~\ref{alg:det} presents a heuristic sampling algorithm returning a sample $\boldsymbol{x} \in \mcal{X}^n$ which is approximately distributed according to $\det(G^{\boldsymbol{x}}) \dx[(\rho_1\otimes \dots \otimes \rho_n)](\boldsymbol{x})$ and with $n$ such that $\mu(\boldsymbol{x}) \le \mu_\star$.

\begin{algorithm}
    \caption{Heuristic determinantal sampling algorithm}\label{alg:det}
    \KwData{$\mu_\star > 1$ and sequence of probability measures $\braces{\rho_k}_{k\in\mbb{N}}$}
    \KwResult{$\boldsymbol{x}\in\mcal{X}^n$ satisfying $\mu(\boldsymbol{x})\le\mu_\star$}
    Set $\boldsymbol{x} := \emptyset$\\
    \For{$i := 1$ \KwTo $d$}{
        Draw $x_i \sim q(\bullet | \boldsymbol{x}) \rho_i$ with $q$ as in equation~\eqref{eq:product_formula}\\
        Update $\boldsymbol{x} := \boldsymbol{x} \oplus  \{x_i\}$\\
    }
    \While{$\mu(\boldsymbol{x}) > \mu_\star$}{
        Set $i := i+1$\\
        Draw $x_i \sim (1 + r(\bullet\mid\boldsymbol{x}))\rho_i$ with $r$ as in equation~\eqref{eq:r_conditional_pdf}\\
        Update $\boldsymbol{x} := \boldsymbol{x} \oplus \{x_i\}$\\
    }
\end{algorithm}

\subsection{Theoretical guarantees}
\label{sec:th-guarantees}

The preceding sections establish that sampling from the (unnormalised) probability measure $\det(G^{\boldsymbol{x}})\dx[\rho^{\otimes n}(\boldsymbol{x})]$ (or its generalisation $\det(G^{\boldsymbol{x}}) \dx[(\rho_1\otimes \dots \otimes \rho_n)](\boldsymbol{x})$) is well-defined and provides a concrete strategy for sampling from this distribution. 
However, as discussed in Remark~\ref{rmk:choice}, the introduction of $\rho$ is artificial, and a proper choice of this measure is important to ensure $\mu(\boldsymbol{x}) \le \mu^\star$ with high probability.
The subsequent section provides probability bounds for the event $\mu(\boldsymbol{x}) \le \mu^\star$ under different choices of sampling measures.

\begin{theorem}
\label{thm:iid_bounds}
    Let $\nu$ be a measure on $\mcal{X}$ and suppose that the embedding constant
    $$
        \tilde{C}_d := \max_{v\in\mathcal{V}_d} \frac{\|v\|_{\mathcal{V}}}{\|v\|_{L^2(K^{-1}\nu)}}
    $$
    is finite.
    Moreover, define the classical \emph{inverse Christoffel function}
    \begin{equation}
    \label{eq:christoffel}
        \mathfrak{K}(x)
        := \sup_{v\in\mathcal{V}_d} \frac{v(x)^2}{\|v\|_{L^2}^2} \;.
    \end{equation}
    Let $w$ be a weight function satisfying $w\ge 0$ and $\int w^{-1}\dx[\nu] = 1$.
    Then $w^{-1}\nu$ is a probability measure and we can draw $x_i\sim w^{-1}\nu$ i.i.d.\ for $i=1,\ldots, n$.
    Then, for all $\beta> 1$,
    \begin{equation}
        \mathbb{P}\big[
            \mu(\boldsymbol{x})
            \le \|wK\|_{L^\infty(\nu)}^{1/2} \tilde{C}_d \beta
        \big]
        \ge 1 - 2d \exp\bigg(
            -\frac{n}{2} \|w\mathfrak{K}\|_{L^\infty(\nu)}^{-1} (1-\beta^{-2})^2
        \bigg) \;.
    \end{equation}
    Moreover, if $\int K\dx[\nu] < \infty$, the choice $\beta=\sqrt{2}$ and $w = 2\big(\tfrac{\mathfrak{K}}{\|\mathfrak{K}\|_{L^1(\nu)}} + \tfrac{K}{\|K\|_{L^1(\nu)}}\big)^{-1}$ simplifies the bound to
    \begin{align}
        \mathbb{P}\big[
            \mu(\boldsymbol{x})
            \le 2 \|K\|_{L^1(\nu)}^{1/2} \tilde{C}_d
        \big]
        &\ge 1 - 2d \exp\bigg(
            -\frac{n}{16d}
        \bigg) \;.
    \end{align}
\end{theorem}
\begin{proof}
    Define the empirical $L^2(\nu)$ norm
    $$
        \|v\|_{L^2, \boldsymbol{x}}^2 := \frac1n \sum_{i=1}^n w(x_i) v(x_i)^2 .
    $$
    Recall that $\mu(\boldsymbol{x}) = \sup_{v\in\mathcal{V}_d}\frac{\|v\|_{\mathcal{V}}}{\|v\|_{\boldsymbol{x}}}$ and, therefore, for every $\alpha > 0$
    \begin{equation}
    \label{eq:P_mu_le_c}
        \mathbb{P}\big[
            \mu(\boldsymbol{x})
            \le \alpha
        \big]
        = \mathbb{P}\big[
            \forall v\in\mathcal{V}_d:
            \|v\|_{\boldsymbol{x}}^2
            \ge \tfrac{1}{\alpha^2} \|v\|_{\mathcal{V}}^2
        \big] \;.
    \end{equation}
    Now, observe that $\|v\|_{\boldsymbol{x}}^2 = \|P_{\mathcal{V}_{\boldsymbol{x}}} v\|_{\mathcal{V}}^2 \ge \|P_{\mathcal{V}_{x_i}} v\|_{\mathcal{V}}^2 = K(x_i)^{-1} v(x_i)^2$ for all $i=1,\ldots,n$ and, therefore,
    $$
        \|v\|_{\boldsymbol{x}}^2
        \ge \frac1n\sum_{i=1}^n K(x_i)^{-1} v(x_i)^2
        \ge \frac1n\sum_{i=1}^n \underbrace{\frac{w(x_i)K(x_i)}{\|wK\|_{L^\infty(\nu)}}}_{\le 1} K(x_i)^{-1} v(x_i)^2
        = \|wK\|_{L^\infty(\nu)}^{-1} \|v\|_{L^2,\boldsymbol{x}}^2 \;.
    $$
    Plugging this bound into~\eqref{eq:P_mu_le_c}, we obtain
    \begin{align}
        \mathbb{P}\big[
            \mu(\boldsymbol{x})
            \le \alpha
        \big]
        &\ge \mathbb{P}\big[
            \forall v\in\mathcal{V}_d:
            \|wK\|_{L^\infty(\nu)}^{-1} \|v\|_{L^2,\boldsymbol{x}}^2
            \ge \tfrac{1}{\alpha^2} \|v\|_{\mathcal{V}}^2
        \big] \\
        &\ge \mathbb{P}\big[
            \forall v\in\mathcal{V}_d:
            \|v\|_{L^2,\boldsymbol{x}}^2
            \ge \|wK\|_{L^\infty(\nu)} \tfrac{1}{\alpha^2} \tilde{C}_d^2 \|v\|_{L^2(\nu)}^2
        \big] \;.
    \end{align}
    Now, we define $\beta = \|wK\|_{L^\infty(\nu)}^{-1/2} \tilde{C}_d^{-1} \alpha$ and use the classical optimal sampling bound~\cite{cohen_2017_optimal}
    \begin{align}
        \mathbb{P}\big[
            \mu(\boldsymbol{x})
            \le \|wK\|_{L^\infty(\nu)}^{1/2}  \tilde{C}_d \beta
        \big]
        &\ge \mathbb{P}\big[
            \forall v\in\mathcal{V}_d:
            \|v\|_{L^2,\boldsymbol{x}}^2
            \ge \tfrac{1}{\beta^2} \|v\|_{L^2(\nu)}^2
        \big] \\
        &\ge 1 - 2d \exp\bigg(
            -\frac{n}{2} \|w\mathfrak{K}\|_{L^\infty(\nu)}^{-1} (1-\beta^{-2})^2
        \bigg) \;.
    \end{align}
    Let $\{\tilde{b}_{j}\}_{j=1,\ldots,d}$ be an $L^2(\nu)$-orthonormal basis of $\mathcal{V}_d$.
    Then $\mathfrak{K} = \sum_{j=1}^d \tilde{b}_j^2$ implies $\|\mathfrak{K}\|_{L^1(\nu)} = d < \infty$ and, thus, that $w = 2\big(\tfrac{\mathfrak{K}}{\|\mathfrak{K}\|_{L^1(\nu)}} + \tfrac{K}{\|K\|_{L^1(\nu)}}\big)^{-1}$ is a well-defined weight function.
    Using this trick, originally due to~\cite{haberstich2022boosted}, it holds that $\|wK\|_{L^\infty(\nu)} \le 2 \|K\|_{L^1(\nu)}$ and $\|w\mathfrak{K}\|_{L^\infty(\nu)} \le 2 \|\mathfrak{K}\|_{L^1(\nu)} = 2d$.
    Plugging these bounds into the first result yields
    \begin{align}
        \mathbb{P}\big[
            \mu(\boldsymbol{x})
            \le \sqrt{2} \|K\|_{L^1(\nu)}^{1/2} \tilde{C}_d \beta
        \big]
        &\ge 1 - 2d \exp\bigg(
            -\frac{n}{4d} (1-\beta^{-2})^2
        \bigg) \;. \qedhere
    \end{align}
\end{proof}

\begin{minipage}{\textwidth}
\begin{remark}
    In contrast to Corollary~\ref{cor:ayoub}, using continuous volume sampling, the above result is achieved with an i.i.d.\ sample.
\end{remark}
\begin{remark}   
    Although the density $\nu$ is a parameter that can reduce the embedding constant $\tilde{C}_d$, we have that 
    $$
        \tilde{C}_d
        = \max_{v\in\mathcal{V}_d} \frac{\|v\|_{\mathcal{V}}}{\|v\|_{L^2(K^{-1}\nu)}}
        \ge \max_{v\in\mathcal{V}_d} \frac{\|v\|_{\mathcal{V}}}{\|\tfrac{v}{K}\|_{C^0(\mathcal{X})}}
    $$
    for all probability measures $\nu$.
    For many RKHS, like $H^1([-1, 1], \tfrac12\dx)$ and $H^1(\mathbb{R}, \mathcal{N}(0,1))$, we have that $K\gtrsim 1$ (see section~\ref{sec:experiments}).
    Using uniformly bounded bases $\{b_k\}_{k\in\mathbb{N}}$ with $\|b_k\|_{C^0(\mathcal{X})} \asymp 1$ (e.g.\ the Fourier basis), it is easy to see that, similar to the constant $C_d$ from Corollary~\ref{cor:ayoub}, the constant $\tilde{C}_d$ can grow very rapidly.
\end{remark}
\end{minipage}

The experiments in section~\ref{sec:experiments} demonstrate that
\begin{itemize}
    \item the error bounds from the preceding theorem are suboptimal (we see $\mathcal{O}(d\log(d))$ rates), and
    \item the improvement of non-i.i.d.\ sampling is quite substantial (reducing the rate to $\mathcal{O}(d)$).
\end{itemize}
The following Theorem~\ref{thm:expectation} proves that the proposed Algorithm \ref{alg:det} necessarily improves over i.i.d.\ sampling, although it does not quantify by how much.

Under the assumptions of Proposition~\ref{prop:marginal},
the marginals of $\det(G^{\boldsymbol{x}})\dx[\rho^{\otimes n}(\boldsymbol{x})]$ with $\rho \propto f \frac{K}{K_d}\nu$ are approximately $f\nu$ for any probability density function $f$.
Theorem~\ref{thm:iid_bounds} hence suggests to use one of the densities $f = \frac{\mathfrak{K}}{d}$ or $f = \tfrac{1}{2}(\tfrac{\mathfrak{K}}{d} + \tfrac{K}{\|K\|_{L^1(\nu)}})$.
These densities ensure that a sample size of $n\gtrsim d\log(d)$ is sufficient to reach a certain $\mu(\boldsymbol{x})\le\mu^\star$ with high probability.
The subsequent theorem implies that sampling from the non-product density $\det(G^{\boldsymbol{x}})\dx[\rho^{\otimes n}(\boldsymbol{x})]$ improves over the i.i.d.\ sampling with respect to the approximate marginals (in the sense that $\log\det G^{\boldsymbol{x}}$ is larger in expectation).

\begin{theorem}\label{thm:expectation}
    Let $\nu$ be a probability measure on $\mcal{X}$ and $\rho = \tfrac1Z \det(G^{\boldsymbol{x}}) \dx[\nu^{\otimes n}(\boldsymbol{x})]$ with $Z = \int \det(G^{\boldsymbol{x}})\dx[\nu^{\otimes n}(\boldsymbol{x})]$.
    Moreover, let $\tilde\rho$ denote the marginal distribution of $\rho$ with respect to any one variable $x$.
    Suppose that $\boldsymbol{x}\in\mcal{X}^n$ is distributed according to $\rho$ and $\tilde{\boldsymbol{x}}\in\mcal{X}^n$ is distributed according to the product measure $\tilde{\rho}^{\otimes n}$.
    Then,
    $$
        \mbb{E}\bracs*{\log\det G^{\boldsymbol{x}}} 
        \ge \mbb{E}\bracs{\log\det G^{\tilde{\boldsymbol{x}}}} 
        .
    $$
\end{theorem}
\begin{proof}
    Recall that the probability density of $\rho$ with respect to $\nu^{\otimes n}$ is given by $P(x) := \tfrac1Z \det(G^{\boldsymbol{x}})$. Let $p$ denote the density of $\tilde\rho$ with respect to $\nu$ and define $Q = p^{\otimes n}$.
    The claim is then equivalent to
    \begin{align}
        &\mathbb{E}[\log\det G^{\boldsymbol{x}}]
        \ge \mathbb{E}[\log\det G^{\tilde{\boldsymbol{x}}}] \\
        \Leftrightarrow\quad&\mathbb{E}_{\boldsymbol{x}\sim P}\left[\log(ZP(\boldsymbol{x}))\right]
        \ge \mathbb{E}_{\boldsymbol{x}\sim Q}\left[\log(ZP(\boldsymbol{x}))\right] \\
        \Leftrightarrow\quad&\mathbb{E}_{\boldsymbol{x}\sim P}\left[\log(P(\boldsymbol{x}))\right]
        - \mathbb{E}_{\boldsymbol{x}\sim Q}\left[\log(P(\boldsymbol{x}))\right] \ge 0 \label{eq:E-difference}
        .
    \end{align}
    Recall the definitions of the (cross-)entropy and Kullback--Leibler divergence
    $$
        H(Q, P) := - \mathbb{E}_{\boldsymbol{x}\sim Q} \left[\log(P(\boldsymbol{x}))\right],
        \quad
        H(P) := H(P, P)
        \quad\text{and}\quad
        \operatorname{KL}(Q \,\|\, P) := H(Q, P) - H(Q)
    $$
    for two densities $P$ and $Q$.
    Using these definitions, we can reformulate equation~\eqref{eq:E-difference} as
    \begin{align}
        \mathbb{E}_{\boldsymbol{x}\sim P}\left[\log(P(\boldsymbol{x}))\right]
        - \mathbb{E}_{\boldsymbol{x}\sim Q}\left[\log(P(\boldsymbol{x}))\right] \ge 0
        \quad&\Leftrightarrow\quad
        -H(P) + H(Q,P) \ge 0 \\
        \quad&\Leftrightarrow\quad
        H(Q,P) - H(Q) + H(Q) - H(P) \ge 0 \\
        \quad&\Leftrightarrow\quad
        \operatorname{KL}(Q \,\|\, P) + H(Q) - H(P) \ge 0 \label{eq:KL+H-difference}
        .
    \end{align}
    Now, recall that the entropy is subadditive in general and additive for product measures, i.e.\ that
    $$
        H(P) \le n H(p) = H(p^{\otimes n}) = H(Q) .
    $$
    This implies $H(Q) - H(P) \ge 0$, and since also $\operatorname{KL}(Q\,\|\,P)\ge 0$, equation~\eqref{eq:KL+H-difference} holds true.
\end{proof}

\section{Subsampling}
\label{sec:subsampling}

The preceding section proposes drawing samples according to the (unnormalised)  distribution $\det(G^{\boldsymbol{x}})d\rho^{\otimes n}(\boldsymbol{x})$, with the intuition that the samples will concentrate where the determinant is large.
This section explores the option to reduce the size of an already given sample.
For this purpose, we suppose that a finite sample $\mcal{D} := \braces{x_1, \ldots, x_n} \subseteq \mcal{X}$ is given, satisfying $\mu(\mcal{D}) \le \mu_\star$,  and we seek a subset $S^\star\subseteq\mcal{D}$ of minimal size satisfying the condition $\mu(\mcal{D}) \le \mu(S^\star) \le \mu^\star$. 
\begin{remark}
    Subsampling algorithms for the weighted least squares projection~\eqref{eq:wls_projection} have already been proposed in~\cite{haberstich2022boosted,bartel2023}.
    These algorithms exploit the fact that the Gram matrix for weighted least squares is a sum of $n$ rank-one matrices $\sum_{i=1}^n a(x_i) a(x_i)^\intercal$, each matrix depending on a single sample point $x_i$.
    These algorithms can not be applied to the Gram matrix $G^{\boldsymbol{x}}$~\eqref{eq:def:Gramian_x} associated with our kernel-based projection.
\end{remark}

As a direct consequence of the definition, the quasi-optimality factor $\mu(S)$ decreases monotonically with increasing sample size $\abs{S}$.
This means that in contrast to classical least squares approximation, adding new sample points is guaranteed to improve the quasi-optimality constant.
This suggests using an incremental, greedy algorithm for sample selection.
Such greedy algorithms enjoy strong convergence guarantees, provided that the objective function is non-negative, monotone and submodular.
Since a lot of the theory on greedy optimisation is formulated in the setting of maximisation, we will adopt this convention and replace the minimisation of $\mu$ with the equivalent maximisation of $\lambda(S) := \mu(S)^{-2}$.
Monotonicity and submodularity of this function can then be defined as follows.

\begin{definition}[Monotonicity]
    Let $\mcal{D}$ be a set, and let $f : 2^{\mcal{D}} \to [0, \infty)$ be a non-negative set function.
    $f$ is called monotone if $f(S) \le f(T)$ whenever $S \subseteq T$.
\end{definition}

\begin{definition}[Submodularity]
    Let $\mcal{D}$ be a set, and let $f : 2^{\mcal{D}} \to [0, \infty)$ be a non-negative, monotone set function.
    $f$ is called submodular if $f(S \cup \braces{v}) - f(S) \ge f(T \cup \braces{v}) - f(T)$ for all $S \subseteq T$ and for all $v\in \mcal{D}$.
\end{definition}

Now suppose that $\lambda$ were non-negative, monotone, and submodular, and consider the algorithm defining the set sequence
$$
    S_0 := \emptyset,
    \qquad
    S_{k+1} = S_k \cup \braces*{v_{k+1}}, \quad v_{k+1} \in   \argmax_{v\in\mcal{D}} \lambda(S_k\cup\braces{v}) .
$$
By a classical result of~\cite{Nemhauser1978}, this algorithm achieves a quasi-optimal function value
$$
    \lambda(S_k) \ge (1 - \tfrac1e) \lambda(S^\star_k),
$$
when $S^\star_k$ is an optimal set of cardinality $k$.
A related theorem of~\cite{Wolsey1982} ensures a similar bound for the ``complementary'' problem of
finding the smallest set $S$ with $\lambda(S) \ge C$.
Defining $k(C) := \min \braces{k : \lambda(S_k) \ge C}$ and $k^\star(C) := \min \braces{\abs{S} : \lambda(S) \ge C}$ for $C \le
\lambda(\mcal{D})$, it holds that
$$
    k(C) \le \pars*{1 + \log\big(\tfrac{C}{C - \lambda(S_{k(C)-1})}\big)} k^\star(C) .
$$
These bounds would justify approximating the sought set $S^\star$ through efficient, greedy algorithms.
%
%
Unfortunately, however, the subsequent proposition shows that $\lambda$ is not submodular.
The proof of this fact can be found in Appendix~\ref{app:proof:prop:not_submodular}.

\begin{proposition}
\label{prop:not_submodular}
    $\lambda$ is monotone but not submodular.
\end{proposition}

This result might seem quite surprising initially because one might intuitively assume that the marginal gains of orthogonal projections become smaller with increasing sample size.
One would intuitively assume that $\lambda$ is submodular, and the construction in Appendix~\ref{app:proof:prop:not_submodular} indeed suggests that $\lambda$ is not ``too far'' from being submodular.
Moreover, it seems natural that the optimality guarantees of submodular functions extend to functions that are ``close'' to being submodular.
This intuition can be formalised by defining the submodularity ratio $\gamma$ as a measure of ``approximate submodularity''~\cite{Das2018}.

\begin{definition}[Submodularity Ratio]
    Let $\mcal{D}$ be a set, and let $f : 2^{\mcal{D}} \to [0, \infty)$ be a non-negative, monotone set function.
    The submodularity ratio of $f$ with respect to a set $U\subseteq\mcal{D}$ and a parameter $k\ge1$ is
    $$
        \gamma_{U,k}(f)
        := \min_{L\subseteq U, S : \abs{S}\le k, L\cap S=\emptyset} \frac{\sum_{s\in S} [ f(L\cup\braces{s}) - f(L)]}{f(L\cup S) - f(L)} ,
    $$
    with the convention that $0/0 = 1$.
\end{definition}

The submodularity ratio captures how much more $f$ can increase by adding any subset $S$ of size $k$ to $L$, compared to the combined benefits of adding its individual elements to $L$.
In particular, a function $f$ is submodular precisely when $\gamma_{U,k} \ge 1$ for all $U$ and $k$~\cite[Proposition~3]{Das2018}.
In this way, the submodularity ratio quantifies the submodularity of $f$.
Moreover, the classical results of~\cite{Wolsey1982} on submodular functions extend to approximately submodular functions by means of the subsequent result of~\cite{Das2018}.

\begin{theorem}[Theorem~11~in~\cite{Das2018}]
\label{thm:submodular_bound}
    For any $\varepsilon \in (0,1)$ and $C>0$, let $k := k((1-\varepsilon)C)$ and $k^\star := k^\star(C)$.
    Then
    $$
        k \le \tfrac1{\gamma_{S_{k}, k^\star}(\lambda)} \log(\varepsilon^{-1}) k^\star .
    $$
\end{theorem}

Recall from Lemma~\ref{lem:mu_is_gramian_error} that
$$
    \lambda(\boldsymbol{x})
    := \mu(\boldsymbol{x})^{-2}
    \ge 1 - \operatorname{tr}(I - G^{\boldsymbol{x}})
    = \operatorname{tr}(G^{\boldsymbol{x}}) - d + 1 .
$$
Instead of maximising $\lambda$, we propose to maximise the surrogate
$$
    \eta(\boldsymbol{x})
    := \operatorname{tr}(G^{\boldsymbol{x}})
    = \sum_{j=1}^d \norm{P_{\mcal{V}_{\boldsymbol{x}}} b_j}_{\mcal{V}}^2 .
$$
The subsequent proposition~\ref{prop:approx_submodular} proves that this function is approximately submodular.
However, replacing the spectral norm with the Hilbert--Schmidt norm, as done in Lemma~\ref{lem:mu_is_gramian_error}, comes at the cost of a slightly suboptimal bound since the Hilbert--Schmidt norm can exceed the spectral norm by a factor of $d^{1/2}$ (cf.~\cite[Remark~3.1]{Binev2018}).

Since our sampling strategy in section~\ref{sec:sampling} is motivated as a surrogate for maximising the function $\boldsymbol{x}\mapsto\log\det(G^{\boldsymbol{x}})$, it may seem more natural to use the function $S \mapsto \log\det(G^{\boldsymbol{x}_S})$ for subsampling.
Indeed, when $H^{\boldsymbol{x}\oplus y}$ can be written as a rank-one update
$H^{\boldsymbol{x} \oplus y} = H^{\boldsymbol{x}} + h_{y}h_{y}^\intercal$, it can be shown~\cite{Cortesi2014} that the functions
$$
    S \mapsto - \operatorname{tr}((H^{\boldsymbol{x}_S})^{-1})
    \qquad\text{and}\qquad
    S \mapsto \log\det(H^{\boldsymbol{x}_S})
$$
are submodular.
However, since the $g_y$-terms in the rank-one update~\eqref{eq:G_rk1_update} of $G^{\boldsymbol{x}\oplus y}$ depend on $\boldsymbol{x}$, the arguments from~\cite{Cortesi2014} do not transfer directly.
Although it may be possible to show (approximate) submodularity for these functions, we focus on the more straightforward case of $\eta$.
Moreover, for the purpose of finding a subsample $\boldsymbol{x}_S$ of close-to-optimal size $\abs{S}\approx d$, it is better to optimise the function $\eta$ (at least for the first $d$ iterations).
The functions $S \mapsto - \operatorname{tr}((H^{\boldsymbol{x}_S})^{-1})$ and $S \mapsto \log\det(H^{\boldsymbol{x}_S})$ are both constant for $\abs{S}<d$.

The greedy maximisation of $\eta$ is a specialised version of the $\mathrm{SDS}_{\mathrm{OMP}}$ algorithm for dictionary selection presented in~\cite[Section~4.2]{Das2018} and comes with the same guarantees on monotonicity and approximate submodularity, as is shown in the subsequent proposition.

\begin{proposition}
\label{prop:approx_submodular}
    Define $C_{i,j} := \frac{k(x_i, x_j)}{\sqrt{K(x_i)K(x_j)}}$ and denote by $C_{S,S}$ the restriction of $C$ onto the indices contained in $S$.
    Moreover, let $\lambda_{\mathrm{min}}(C, k) := \min_{\abs{S} \le k} \lambda_{\mathrm{min}}(C_{S,S})$ denote the \emph{minimal $k$-sparse eigenvalue of $C$}.
    Then, $\eta$ is monotone and approximately submodular with
    $$
        \gamma_{U,k}(\eta)
        \ge \lambda_{\mathrm{min}}(C, \abs{U}+k) \ge \lambda_{\mathrm{min}}(C) .
    $$
\end{proposition}
\begin{proof}
    Let $\boldsymbol{x} \in \mcal{X}^n$ be given.
    To facilitate this proof, we introduce the notation $\boldsymbol{x}_L$ for the sub-vector of $\boldsymbol{x}$ with indices in $L \subseteq \braces{1,\ldots,n}$ and extend this notion to matrices as well.
    Since $\mcal{V}_{\boldsymbol{x}_L} \subseteq \mcal{V}_{\boldsymbol{x}_S}$ for all $L \subseteq S  \subseteq\braces{1,\ldots,n}$, it follows that $\norm{P_{\mcal{V}_{\boldsymbol{x}_L}} b_j}_{\mcal{V}} \le \norm{P_{\mcal{V}_{\boldsymbol{x}_S}} b_j}_{\mcal{V}}$ for every $j=1,\ldots, d$.
    This proves $\eta(\boldsymbol{x}_L) \le \eta(\boldsymbol{x}_S)$ and thereby monotonicity.

    To compute the submodularity ratio, we define the dictionary of 
    normalised functions $\omega : \mcal{X}\to\mbb{R}^n$ by $\omega_i := K(x_i)^{-1/2}k(x_i, \bullet)$
    and introduce for any $L,S\subseteq\braces{1,\ldots,n}$ the notation $P_{\inner{\omega_L}}$ for the $\mcal{V}$-orthogonal projection onto $ \inner{\omega_L}:= \operatorname{span}\{\omega_l : l\in L\}$ and write $P_{\inner{\omega_L}} \omega_S$ for the component-wise application of the projection operator to $\omega_S$.
    With this, we can succinctly write
    $$
        P_{\inner{\omega_{L\cup S}}} = P_{\inner{\omega_L}} + P_{\inner{(I-P_{\inner{\omega_L}})\omega_S}} .
    $$
    It thus holds that
    \begin{align}
        \gamma_{U,k}(\eta)
        &:= \min_{L\subseteq U, \abs{S}\le k, L\cap S=\emptyset} \frac{\sum_{s\in S} (\eta(\boldsymbol{x}_{L\cup\braces{s}}) - \eta(\boldsymbol{x}_L)}{\eta(\boldsymbol{x}_{L\cup S}) - \eta(\boldsymbol{x}_L)}) \\
        &= \min_{L\subseteq U, \abs{S}\le k, L\cap S=\emptyset} \frac{
            \sum_{s\in S} ( \sum_{j=1}^d
                \norm{P_{\inner{\omega_{L\cup\braces{s}}}} b_j}_{\mcal{V}}^2
                - \norm{P_{\inner{\omega_{L}}} b_j}_{\mcal{V}}^2)
        }{
            \sum_{j=1}^d 
                \norm{P_{\inner{\omega_{L\cup S}}} b_j}_{\mcal{V}}^2
                - \norm{P_{\inner{\omega_L}} b_j}_{\mcal{V}}^2
        } \\
        &= \min_{L\subseteq U, \abs{S}\le k, L\cap S=\emptyset} \frac{
            \sum_{s\in S} \sum_{j=1}^d
                \norm{P_{\inner{(I - P_{\inner{\omega_L}}) \omega_{\braces{s}}}} b_j}_{\mcal{V}}^2
        }{
            \sum_{j=1}^d 
                \norm{P_{\inner{(I - P_{\inner{\omega_L}}) \omega_S}} b_j}_{\mcal{V}}^2
        } .
    \end{align}
    We now define a basis $\omega_s^L$ for the spaces $\inner{(I - P_{\inner{\omega_L}})\omega_{\braces{s}}}$ occurring in the numerator of the preceding expression.
    Using the Gramian matrix $C_{i,j} := (\omega_i, \omega_j)_{\mcal{V}}$, these basis elements and their corresponding Gramian can be expressed as
    $$
        \omega^L_i
        := (I - P_{\inner{\omega_L}}) \omega_i
        = \omega_i - C_{i,L} C_{L,L}^{-1} \omega_L
        \qquad\text{and}\qquad
        C^L_{i,j}
        := (\omega^L_i, \omega^L_j)_{\mcal{V}}
        = C_{i,j} - C_{i,L} C_{L,L}^{-1} C_{L,j}
        .
    $$
    Similarly, we can express the inner products between $\omega^L$ and $b$ through the cross-Gramian matrix $B^L_{i,j} := (\omega^L_i, b_j)_{\mcal{V}}$ and define the diagonal scaling matrix $D := \operatorname{diag}(\norm{\omega^L_1}_{\mcal{V}}, \ldots, \norm{\omega^L_n}_{\mcal{V}})$.
    This allows us to write the submodularity ratio as
    \begin{align}
        \gamma_{U,k}(\eta)
        &= \min_{L\subseteq U, \abs{S}\le k, L\cap S=\emptyset} \frac{
            \sum_{j=1}^d \sum_{s\in S}
                \norm{P_{\inner{(I - P_{\inner{\omega_L}}) \omega_{\braces{s}}}} b_j}_{\mcal{V}}^2
        }{
            \sum_{j=1}^d 
                \norm{P_{\inner{(I - P_{\inner{\omega_L}}) \omega_S}} b_j}_{\mcal{V}}^2
        } 
        = \min_{L\subseteq U, \abs{S}\le k, L\cap S=\emptyset} \frac{
            \sum_{j=1}^d
                (B^L_{S,j})^\intercal D^{-2} (B^L_{S,j})
        }{
            \sum_{j=1}^d 
                (B^L_{S,j})^{\intercal} (C^L_{S,S})^{-1} (B^L_{S,j})
        } .
    \end{align}
    Stacking $d$ copies of $C^L_{S,S}$ and $D$ in the block-diagonal matrices $\bar{C} := \operatorname{diag}(C^L_{S,S}, \ldots, C^L_{S,S})$ and $\bar{D} := \operatorname{diag}(D, \ldots, D)$ and concatenating all $B^L_{S,j}$ into the vector $b_{i + \abs{S}(j-1)} := (B^L_{S,j})_i$
    the fraction can be estimated as
    $$
        \frac{
            \sum_{j=1}^d (B^L_{S,j})^\intercal D^{-2} (B^L_{S,j})
        }{
            \sum_{j=1}^d  (B^L_{S,j})^{\intercal} (C^L_{S,S})^{-1} (B^L_{S,j})
        }
        = \frac{b^\intercal \bar{D}^{-2} b}{b^\intercal \bar{C}^{-1} b}
        \ge \min_{b\in\mbb{R}^{d\abs{S}}} \frac{b^\intercal b}{b^\intercal \bar{D}\bar{C}^{-1}\bar{D} b}
        = \lambda_{\mathrm{min}}(\bar{D}^{-1}\bar{C}\bar{D}^{-1})
        = \lambda_{\mathrm{min}}(D^{-1}C^L_{S,S}D^{-1}) .
    $$
    To estimate this eigenvalue, we proceed in two steps.
    First, we note that
    $$
        D_{i,i}
        = \norm{\omega^L_i}_{\mcal{V}}
        = \norm{(I - P_{\inner{\omega^L}})\omega_i}_{\mcal{V}}
        \le \norm{\omega_i}_{\mcal{V}}
        = 1 ,
    $$
    which implies
    $$
        \lambda_{\mathrm{min}}(D^{-1}C^L_{S,S}D^{-1})
        \ge \lambda_{\mathrm{max}}(D)^{-2} \lambda_{\mathrm{min}}(C^L_{S,S})
        \ge \lambda_{\mathrm{min}}(C^L_{S,S})
        .
    $$
    Next, we note that $C^L_{S,S}$ is precisely the Schur complement $C_{S\cup L, S\cup L}/C_{L,L}$, which implies
    $$
        \lambda_{\mathrm{min}}(C^L_{S,S}) \ge \lambda_{\mathrm{min}}(C_{S\cup L, S\cup L}) .
    $$
    A proof of this fact can be found in Appendix~\ref{app:schur}.
    The submodularity ratio is, therefore, bounded by
    \begin{align}
        \gamma_{U,k}(\eta)
        &\ge \min_{L\subseteq U, \abs{S}\le k, L\cap S=\emptyset} \lambda_{\mathrm{min}}(C_{S\cup L, S\cup L})
        \ge \lambda_{\mathrm{min}}(C, \abs{U}+k) \ge \lambda_{\mathrm{min}}(C) . \qedhere
    \end{align}
\end{proof}

\begin{remark}
    The idea of using submodular optimisation to select subsamples is not new and was studied for the maximisation of the determinant of a kernel matrix in~\cite{NEURIPS2018_dbbf603f}.
\end{remark}

\begin{remark}
    For the set $\mcal{D} = \mcal{X}$, this algorithm was already proposed in~\cite{Binev2018} under the name collective OMP.
    A problem with the approach in~\cite{Binev2018} is that the proofs require dense dictionaries $\mcal{D}$ for which $\mcal{V}_d \subseteq \overline{\mcal{D}}$.
    Hence, we cannot use their algorithms in the setting of section~\ref{sec:subsampling}.
    Their algorithm was originally designed to tackle the original sampling problem in section~\ref{sec:sampling}.
    However, its complexity is unknown, and it could very well be NP-hard to find the minimum.
\end{remark}

\begin{remark}
    It is illustrated in~\cite{Binev2018} that the incremental selection of sample points (when the dimension $d$ of $\mcal{V}_d$ increases) is of roughly the same quality as when the sample points are drawn anew for each $d$.
     This observation is not surprising since, for a given set of basis functions $\mcal{B} := \braces{b_1,\ldots, b_d}$, the corresponding function
    $$
       \eta_{\mcal{B}}(\boldsymbol{x}) := \sum_{b\in\mcal{B}} \norm{P_{\mcal{V}_{\boldsymbol{x}}} b}_{\mcal{V}}^2
    $$
    is modular in $\mcal{B}$, i.e.\ $\eta_{\mcal{B}\cup\braces{v}} = \eta_{\mcal{B}} + \eta_{\braces{v}}$ for all $\mcal{B}$ and $v\not\in\mcal{B}$. 
    When the function $\eta_{\braces{v}}$ is not significantly larger than $\eta_{\mcal{B}}$ (which seems reasonable), then a subset $S\subseteq\mcal{D}$ that is selected to maximise $\eta_{\mcal{B}}$ will also produce a large value for the function $\eta_{\mcal{B}\cup\braces{v}} = \eta_{\mcal{B}} + \eta_{\braces{v}}$.
    This encourages the idea that old sample points that have already been used for the approximation in $\mcal{V}_d$ can be recycled for the approximation in $\mcal{V}_{d+1}$ without drawbacks.
\end{remark}

We conclude this section by listing the proposed sampling strategy with greedy subsampling in pseudo-code.
\begin{algorithm}
\caption{Sampling algorithm with greedy subsampling}\label{alg:greedy}
    \KwData{$\mu_\star > 1$, $\beta\in[1, \infty)$ and sequence of probability measures $\braces{\rho_k}_{k\in\mbb{N}}$}
    \KwResult{$\boldsymbol{x}\in\mcal{X}^n$ satisfying $\mu(\boldsymbol{x})\le\mu_\star$}
    Use Algorithm~\ref{alg:det} to draw a sample set $\boldsymbol{x}\in\mcal{X}^n$ satisfying $\mu(\boldsymbol{x}) \le 1 + \frac{\mu_\star - 1}{\beta}$\\
    Set $\boldsymbol{y} := \emptyset$\\
    \While{$\mu(\boldsymbol{y}) > \mu_\star$}{
        Set $y := \argmax\braces{\eta(\boldsymbol{y}\oplus \{y\}) \,:\, y\in\boldsymbol{x}}$\\
        Update $\boldsymbol{y} := \boldsymbol{y} \oplus  \{y\}$\\
    }
    Set $\boldsymbol{x} :=  \boldsymbol{y}$
\end{algorithm}

\section{Perturbed evaluations}
\label{sec:noisy}

Now assume that the observations are perturbed by deterministic noise,
i.e.\ we only have access to $y_i := u(x_i) + \eta(x_i)$ where $\eta$ is a function in some normed vector space $\mcal{R}\supseteq \mcal{V}$.
We let $k_\mcal{R}$ be a positive semi-definite kernel, which defines for any $\boldsymbol{x}\in \mcal{X}^n$ a semi-norm 
$$
    \Vert v \Vert_{\boldsymbol{x},\mcal{R}}^2
    :=  \Vert v(\boldsymbol{x}) \Vert_{K_{\mcal{R}}(\boldsymbol{x})^+}^2
    := v(\boldsymbol{x})^\intercal K_{\mcal{R}}^+(\boldsymbol{x}) v(\boldsymbol{x}),
$$
with $K_{\mcal{R}}(\boldsymbol{x})$ being the positive semi-definite kernel matrix associated with $k_{\mcal{R}}$.
Furthermore, we assume that 
\begin{equation}
    \Vert v \Vert_{\boldsymbol{x},\mcal{R}}  \le c_n \Vert v \Vert_{\mcal{R}}, \label{eq:semi-norm-R}
\end{equation}
with a constant $c_n$ that may depend on $n$.

\begin{example}
    Consider the case where $\mcal{R}$ is a RKHS with kernel $k_{\mcal{R}}$.
    Then the semi-norm $\Vert v \Vert_{\boldsymbol{x},\mcal{R}} = \Vert P_{\mcal{R}_{\boldsymbol{x}}} v \Vert_{\mcal{R}}$, where $P_{\mcal{R}_{\boldsymbol{x}}} $ is the $\mcal{R}$-orthogonal projection onto $\mcal{R}_{\boldsymbol{x}} = \operatorname{span}(k_{\mcal{R}}(\boldsymbol{x},\bullet))$. This yields property \eqref{eq:semi-norm-R} with constant $c_n = 1$.  
\end{example}

\begin{minipage}{\textwidth}
\begin{example}
    Consider the case where $\mcal{R}$ is the weighted Lebesgue space $L^\infty_{\gamma^{-1/2}}$ for some weight function $\gamma : \mcal{X} \to (0,\infty)$, equipped with the norm $\Vert v \Vert_{L^\infty_{\gamma^{-1/2}}} := \esssup_{x\in \mcal{X}} \gamma(x)^{-1/2} \vert v(x) \vert$, and where $k_{\mcal{R}}(x,y) = \gamma(x) \mathbf{1}_{x=y}$ is the weighted white noise kernel.
    Then $K_{\mcal{R}}(\boldsymbol{x}) = \operatorname{diag}(\gamma(\boldsymbol{x}))$ and 
    $\Vert v \Vert_{\boldsymbol{x},\mcal{R}}^2 =  \sum_{i=1}^n v(x_i)^2 \gamma(x_i)^{-1} \le n \Vert v\Vert_{\mcal{R}}^2.$ Thus \eqref{eq:semi-norm-R} holds with constant $c_n = \sqrt{n}.$
\end{example}
\end{minipage}

For $\boldsymbol{x}\in \mcal{X}^n$, we define the regularised matrix 
$$
    K_{\mcal{S}}(\boldsymbol{x}) := K(\boldsymbol{x}) + c_n K_{\mcal{R}}(\boldsymbol{x})
$$
and define the corresponding norm $\norm{v}_{\boldsymbol{x},\mcal{S}} := \norm{v(\boldsymbol{x})}_{K_{\mcal{S}}(\boldsymbol{x})^+}$.
 
The regularised projection $u^{\boldsymbol{x},\mcal{S}} \in \mcal{V}_d$ is defined by
$$
    u^{\boldsymbol{x},\mcal{S}}
    := \argmin_{v\in\mcal{V}_d} \; \norm{\boldsymbol{y} - v(\boldsymbol{x})}_{K_{\mcal{S}}(\boldsymbol{x})^+}
    = \argmin_{v\in\mcal{V}_d} \; \norm{(u + \eta) - v}_{\boldsymbol{x},\mcal{S}}
    =: P_{\mcal{V}_d}^{\boldsymbol{x},\mcal{S}} (u + \eta) .
$$

\begin{proposition}
\label{prop:error_bound_noisy}
    Let $b$ be a $\mcal{V}$-orthonormal basis of $\mcal{V}_d$ and define $\mu_{\mcal{S}}(\boldsymbol{x}) := \lambda_{\mathrm{min}}(b(\boldsymbol{x}) K_{\mcal{S}}(\boldsymbol{x})^+ b(\boldsymbol{x})^\intercal)^{-1/2}$.
    Then
    $$
        \norm{u - u^{\boldsymbol{x},\mcal{S}}}_{\mcal{V}}
        \le (1 + \mu_{\mcal{S}}(\boldsymbol{x})) \norm{u - P_{\mcal{V}_d}u}_{\mcal{V}}
        + \mu_{\mcal{S}}(\boldsymbol{x}) \norm{\eta}_{\mcal{R}}.
    $$
\end{proposition}
\begin{proof}
    As in the proof of Theorem~\ref{thm:error_bound_noiseless}, we can estimate
    $$
        \norm{u - u^{\boldsymbol{x},\mcal{S}}}_{\mcal{V}}
        \le \norm{u - P_{\mcal{V}_d}u}_{\mcal{V}} + \norm{P_{\mcal{V}_d}u - u^{\boldsymbol{x},\mcal{S}}}_{\mcal{V}}
        \le \norm{u - P_{\mcal{V}_d}u}_{\mcal{V}} + \mu_{\mcal{S}}(\boldsymbol{x}) \norm{P_{\mcal{V}_d} u - u^{\boldsymbol{x},\mcal{S}}}_{\boldsymbol{x},\mcal{S}} .
    $$
    Moreover, it holds that
    \begin{align}
        \norm{P_{\mcal{V}_d} u - u^{\boldsymbol{x},\mcal{S}}}_{\boldsymbol{x},\mcal{S}}
        &= \norm{P_{\mcal{V}_d} u - P_{\mcal{V}_d}^{\boldsymbol{x},\mcal{S}} (u + \eta)}_{\boldsymbol{x},\mcal{S}} \\
        &\le \norm{P_{\mcal{V}_d}^{\boldsymbol{x},\mcal{S}} (u - P_{\mcal{V}_d} u)}_{\boldsymbol{x},\mcal{S}} + \norm{P_{\mcal{V}_d}^{\boldsymbol{x},\mcal{S}} \eta}_{\boldsymbol{x},\mcal{S}} \\
        &\le \norm{u - P_{\mcal{V}_d} u}_{\boldsymbol{x},\mcal{S}} + \norm{\eta}_{\boldsymbol{x},\mcal{S}}
        .
    \end{align}
    Since $K_{\mcal{S}}(\boldsymbol{x}) \succeq K(\boldsymbol{x})$ and $K_{\mcal{S}}(\boldsymbol{x}) \succeq c_n K_{\mcal{R}}(\boldsymbol{x})$, we can estimate
    \begin{align}
        \norm{u - P_{\mcal{V}_d} u}_{\boldsymbol{x},\mcal{S}}
        = \norm{(u - P_{\mcal{V}_d} u)(\boldsymbol{x})}_{K_{\mcal{S}}(\boldsymbol{x})^+}
        \le \norm{(u - P_{\mcal{V}_d} u)(\boldsymbol{x})}_{K(\boldsymbol{x})^+}
        = \norm{u - P_{\mcal{V}_d} u}_{\boldsymbol{x}} \\
    \intertext{and}
        \norm{\eta}_{\boldsymbol{x},\mcal{S}}
        = \norm{\eta(\boldsymbol{x})}_{K_{\mcal{S}}(\boldsymbol{x})^+}
        \le c_n^{-1} \norm{\eta(\boldsymbol{x})}_{K_{\mcal{R}}(\boldsymbol{x})^+}
        = c_n^{-1} \norm{\eta}_{\boldsymbol{x}, \mcal{R}} .
    \end{align}
    Using the bounds $\norm{u - P_{\mcal{V}_d} u}_{\boldsymbol{x}} \le \norm{u - P_{\mcal{V}_d}u}_{\mcal{V}}$ and $\norm{\eta}_{\boldsymbol{x},\mcal{R}} \le c_n\norm{\eta}_{\mcal{R}}$ concludes the proof.
\end{proof}

\section{Experiments}
\label{sec:experiments}

{\color{red}
}

Let $\nu$ be a probability measure on $\mathcal{X}$, specified differently for each experiment and define
$$
    \mathfrak{K}(x)
    := \sup_{v\in\mathcal{V}_d} \frac{v(x)^2}{\|v\|_{L^2(\nu)}^2}
    \,.
$$
We propose to generate $\boldsymbol{x}\in\mcal{X}^n$ from the distribution~\eqref{eq:our-distribution} using Algorithm~\ref{alg:det} with 
$\rho_1=\ldots=\rho_d \propto \mathfrak{K} \frac{K}{K_d}\nu$
and $\rho_n \propto \mathfrak{K}\nu$ for all $n>d$.
These measures are indeed finite if $\mcal{V}$ is compactly embedded in $L^2(\nu)$ and
\begin{equation}
\label{eq:finite_trace}
    \int K(x) \dx[\nu(x)] < \infty .\footnote{
To see this, let $b_1,\ldots, b_d$ be an $L^2(\nu)$-orthogonal and $\mcal{V}$-orthonormal basis of $\mcal{V}_d$.
Then
$$
    \mathfrak{K}
    \le \pars{\max_j \norm{b_j}_{L^2(\nu)}^{-2}} K_d
    \qquad\text{and}\qquad
    K_d \le K,
$$
which implies $\mathfrak{K}\nu \lesssim K_d\nu \le K\nu$ and $\frac{\mathfrak{K}}{K_d}K\nu \lesssim K\nu$.}
\end{equation}
Propositions~\ref{prop:marginal} and~\ref{prop:conditional_marginal} guarantee that the generated samples are (approximately) distributed according to $\det(G^{\boldsymbol{x}}) \dx[(\rho_1 \otimes \cdots \otimes\rho_n)](\boldsymbol{x})$, if the parametric integrals $Z_i$ from Proposition~\ref{prop:marginal} remain bounded.
Figure~\ref{fig:Z_function_distribution} shows that this is the case for the choice $\rho_1=\ldots=\rho_d \propto \mathfrak{K}\frac{K}{K_d}\nu$.
As a consequence of Proposition~\ref{prop:conditional_marginal} and Remark~\ref{rmk:marginal}, the chosen reference measures $\rho_i$ ensure that the marginals of $\det(G^{\boldsymbol{x}}) \dx[(\rho_1 \otimes \cdots \otimes\rho_n)](\boldsymbol{x})$ approximate $\mathfrak{K}\nu$.

For the purpose of this discussion, we denote the proposed method as \emph{subspace-informed volume sampling} (SIVS), since the density is proportional to the volume of the Gramian matrix $G^{\boldsymbol{x}}$~\eqref{eq:def:Gramian_x}, which depends on the subspace $\mcal{V}_d$ as well as the ambient RKHS $\mcal{V}$.
We compare this method against two other sampling methods.
\begin{enumerate}

    
    \item Christoffel sampling ~\cite{cohen_2017_optimal} draws $\boldsymbol{x}\in \mcal{X}^n$ as $n$ i.i.d.\ samples from $\mathfrak{K}\nu$.
    This is motivated by Theorem~\ref{thm:iid_bounds}, and this method is indeed optimal for $L^2$-approximation with point evaluations (cf.~\cite{trunschke2024optimalsamplingsquaresapproximation}).
    
    \item Continuous volume sampling~\cite{belhadji2020kernel} draws $\boldsymbol{x} \in \mcal{X}^n$ from~\eqref{eq:cvs}.
    It's major distinction from the proposed SIVS method is that continuous volume sampling only takes into account the ambient space $\mcal{V}$, while SIVS depends on the subspace $\mcal{V}_d$ as well.
    Sample size bounds are provided in Corollary~\ref{cor:ayoub}.
\end{enumerate}

After comparing these four sampling methods, we illustrate the convergence of the proposed greedy subsampling method (Algorithm~\ref{alg:greedy}).

\begin{figure}
    \centering
    \includegraphics[width=\textwidth]{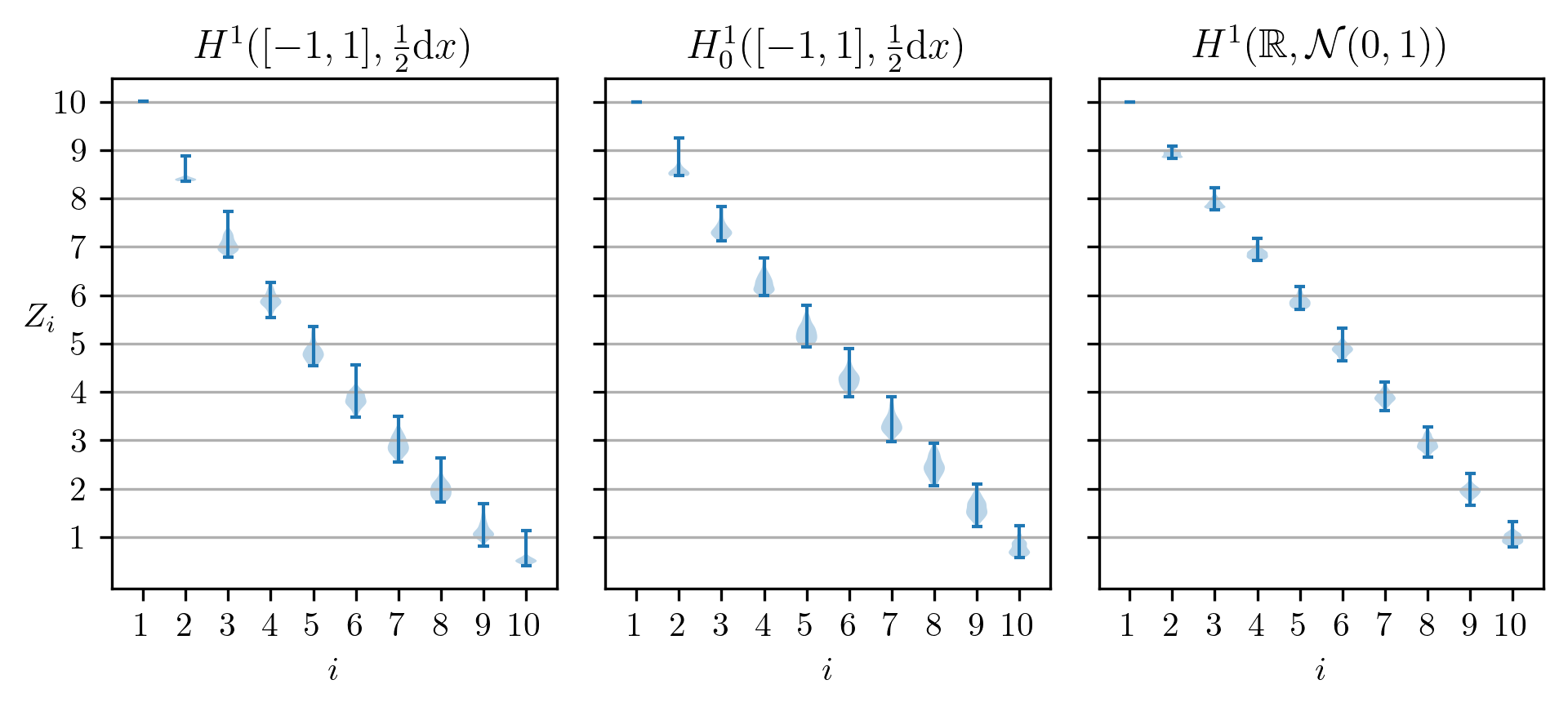}
    \caption{Distribution and bounding intervals of $Z_i(\boldsymbol{x}_{<i})$ for $\boldsymbol{x}\in\mcal{X}^d$ drawn according to $x_i\sim q(\bullet \mid \boldsymbol{x}_{<i}) \rho_i$, with $d=10$ and $\mathcal{V}$ and $\mathcal{V}_d$ as defined in sections~\ref{sec:experiments:h1},~\ref{sec:experiments:h10} and~\ref{sec:experiments:gaussian}.}
    \label{fig:Z_function_distribution}
\end{figure}

\begin{remark}
    An alternative \emph{volume-rescaled sampling distribution} has been proposed in~\cite{Derezinski2022Jan} for least squares approximation in a subspace $\mcal{V}_d$ of $L^2(\nu)$.
    It is equivalent to sampling $d$ points from a projection determinantal point process associated with $\mcal{V}_d$ and the reference measure $\nu$, and $n-d$ additional i.i.d.\ points from the $L^2$-Christoffel sampling distribution $\mfrak{K} \nu$.
    We do not consider this method in our experiments since it has similar properties as i.i.d.\ $L^2$-Christoffel sampling.
\end{remark}

We compare these methods for three prototypical cases.
\begin{itemize}
   \item \textbf{Section~\ref{sec:experiments:h1}:} For $\mcal{V} = H^1([-1,1], \tfrac12\dx)$ the kernel $1 \lesssim K(x) \lesssim 1$ is uniformly bounded from above and below.
   \item \textbf{Section~\ref{sec:experiments:h10}:} For $\mcal{V} = H^1_0([-1,1], \tfrac12\dx)$ the kernel $K(x) \lesssim 1$ is uniformly bounded only from above.
   \item \textbf{Section~\ref{sec:experiments:gaussian}:} For $\mcal{V} = H^1(\mbb{R}, \mcal{N}(0,1))$ the kernel $1 \lesssim K(x)$ is uniformly bounded only from below.\footnote{Note that an unbounded domain is necessary in this case since $K(x)$ has to be finite for every $x\in\mcal{X}$ by the definition of an RKHS.
   This implies that $k(x,x)$ must be bounded from above as soon as it is continuous and $\mcal{X}$ is compact.}
\end{itemize}

The kernels for $H^1([-1,1], \tfrac12\dx)$ and $H^1_0([-1,1], \tfrac12\dx)$ are standard and can be found for example in~\cite{Binev2018,Paulsen2016,Tutaj2019}.
The kernel for $H^1(\mbb{R}, \mcal{N}(0,1))$ can be found in the same manner as those of $H^1([-1,1], \tfrac12\dx)$ and $H^1_0([-1,1], \tfrac12\dx)$, and a derivation is provided, for the sake of completeness, in Appendix~\ref{app:gaussian_rkhs}.
 For the convenience of the reader, we restate the reproducing kernels for the corresponding spaces in the beginning of each section.


\subsection{A polynomial subspace of \texorpdfstring{$\boldsymbol{H^1([-1,1], \tfrac12\dx)}$}{H1 with uniform measure}}
\label{sec:experiments:h1}

Consider the Hilbert space $\mathcal{V} = H^1([-1, 1], \nu)$ with $\nu := \tfrac12\dx$ and the $d$-dimensional polynomial subspace $\mathcal{V}_d = \operatorname{span}\{1, x, \ldots, x^{d-1}\}$.
The reproducing kernel of $\mcal{V}$ is given by
$$
    k(x, y) := \frac{2 \cosh(1 - \max\{x, y\}) \cosh(1 + \min\{x, y\})}{\sinh(2)} .
$$

Phase diagrams for the probability of $\mu(\boldsymbol{x}) \le 2$ are presented in Figure~\ref{fig:mu_h1}.
Several interesting conclusions can be drawn from these observations.
\begin{itemize}
    \item Christoffel sampling follows a sample size bound of $n \gtrsim d\log(d)$ which is consistent with the prediction of Theorem~\ref{thm:iid_bounds}.
    Quite surprisingly, we cannot observe any influence of the embedding constant $\tilde C_d \asymp d$.
    \item Continuous volume sampling follows a sample size bound of $n \gtrsim d^2$.
    This can probably be attributed to the fact that $\mcal{V}_d$ is not spanned by the spectral basis of $\mcal{V}$.
    \item Subspace-informed volume sampling follows the optimal sample size bound $n\gtrsim d$.
    The factor for the rate in the plot is $1.5$, which is close to the optimal factor $1$.
\end{itemize}

The efficiency of the proposed subsampling scheme is illustrated in Figure~\ref{fig:h1_greedy}.
It can be seen that the proposed algorithm produces an acceptable quasi-optimality constant $\mu(\boldsymbol{x})\le 3$ already for the minimal possible sample size of $n=d$ in all $100$ repetitions of the experiment.

\begin{figure}[htb]
    \centering
    \includegraphics[width=\textwidth]{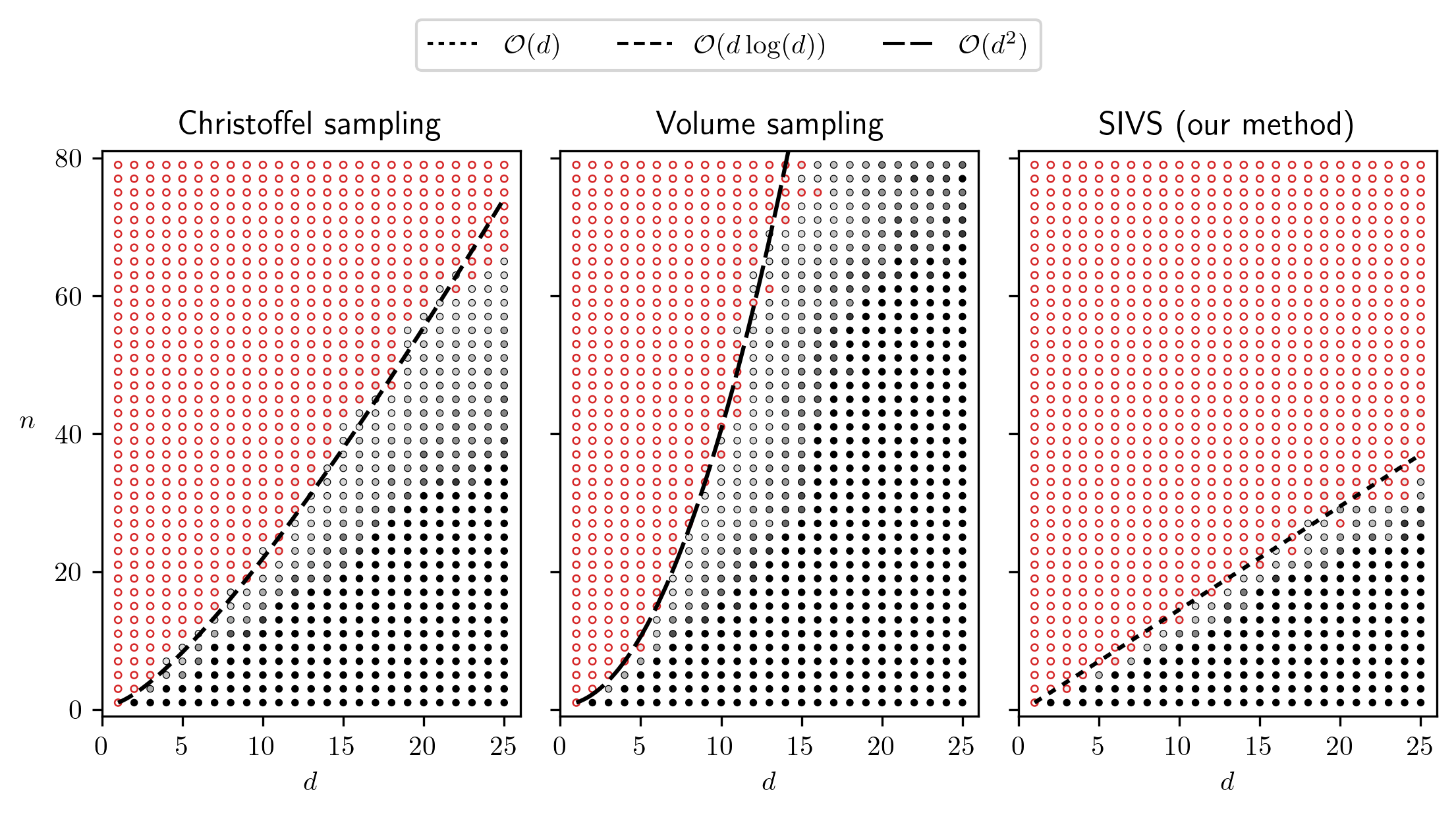}
    \caption{Phase diagrams for the probability $\mbb{P}\bracs*{\mu(\boldsymbol{x}) \le 2}$ with $\mcal{V} = H^1([-1,1], \tfrac12\dx)$ and where $\mcal{V}_d$ is spanned by polynomials.
    The probability is estimated using $200$ independent samples $\boldsymbol{x}\in\mcal{X}^n$ for different dimensions $d$ and sample sizes $n$.
    White marks a probability of $1$, black a probability of $0$.
    Points having $\mbb{P}\bracs*{\mu(\boldsymbol{x}) \le 2} \ge \tfrac12$ are marked with bold, red borders.
    The factor for the linear rate in the phase diagram for subspace-informed volume sampling is $1.5$.
    }
    \label{fig:mu_h1}
\end{figure}

\begin{figure}[htb]
    \centering
    \includegraphics[width=\textwidth]{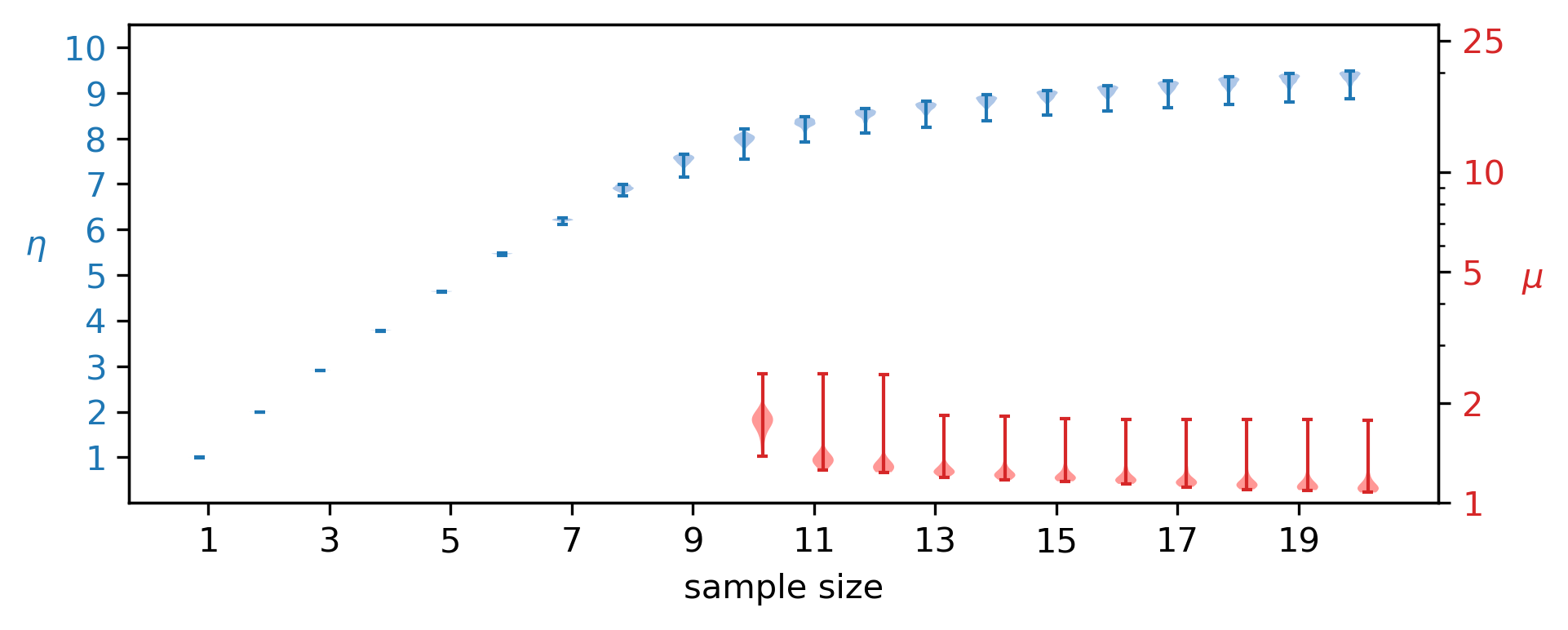}
    \caption{Violin plot of the submodular surrogate $\eta$ and the suboptimality constant $\mu$ for the first $20$ steps of the greedy optimisation procedure.
    The initial sample $\mcal{D}$ is of size $100$ and drawn using the $L^2$-Christoffel sampling method.
    The experiment was repeated $100$ times to compute the violins.
    $\mcal{V} = H^1([-1,1], \tfrac12\dx)$ and $d=10$ with $\mcal{V}_d$ being spanned by polynomials.}
    \label{fig:h1_greedy}
\end{figure}

\subsection{A polynomial subspace of \texorpdfstring{$\boldsymbol{H^1_0([-1,1], \tfrac12\dx)}$}{H10 with uniform measure}}
\label{sec:experiments:h10}

Consider the Hilbert space $\mathcal{V} = H^1_0([-1, 1], \nu)$ with $\nu = \tfrac12\dx$ and the $d$-dimensional polynomial subspace $\mathcal{V}_d$ that is spanned by the monomials $\operatorname{span}\{1, x, \ldots, x^{d+1}\}$ subject to the boundary conditions $v(-1) = v(1) = 0$ for all $v\in\mcal{V}_d$.
The reproducing kernel of $\mcal{V}$ is given by
$$
    k(x, y) := \frac{(\min\braces{x, y} + 1) (1 - \max\braces{x, y})}{4} .
$$

Phase diagrams for the probability of $\mu(\boldsymbol{x}) \le 2$ are presented in Figure~\ref{fig:mu_h10}, and the convergence of the proposed greedy subsampling method is illustrated in Figure~\ref{fig:h10_greedy}.
The qualitative observations remain similar to the $H^1([-1,1], \tfrac12\dx)$-case, and even the factor for linear rate in the subspace-informed volume sampling plot remains the same ($1.5$).
The greedy subsampling algorithm produces an acceptable quasi-optimality constant $\mu(\boldsymbol{x}) \le 3$ for the almost optimal sample size $n = d+2$ in all $100$ repetitions of the experiment.

\begin{figure}
    \centering
    \includegraphics[width=\textwidth]{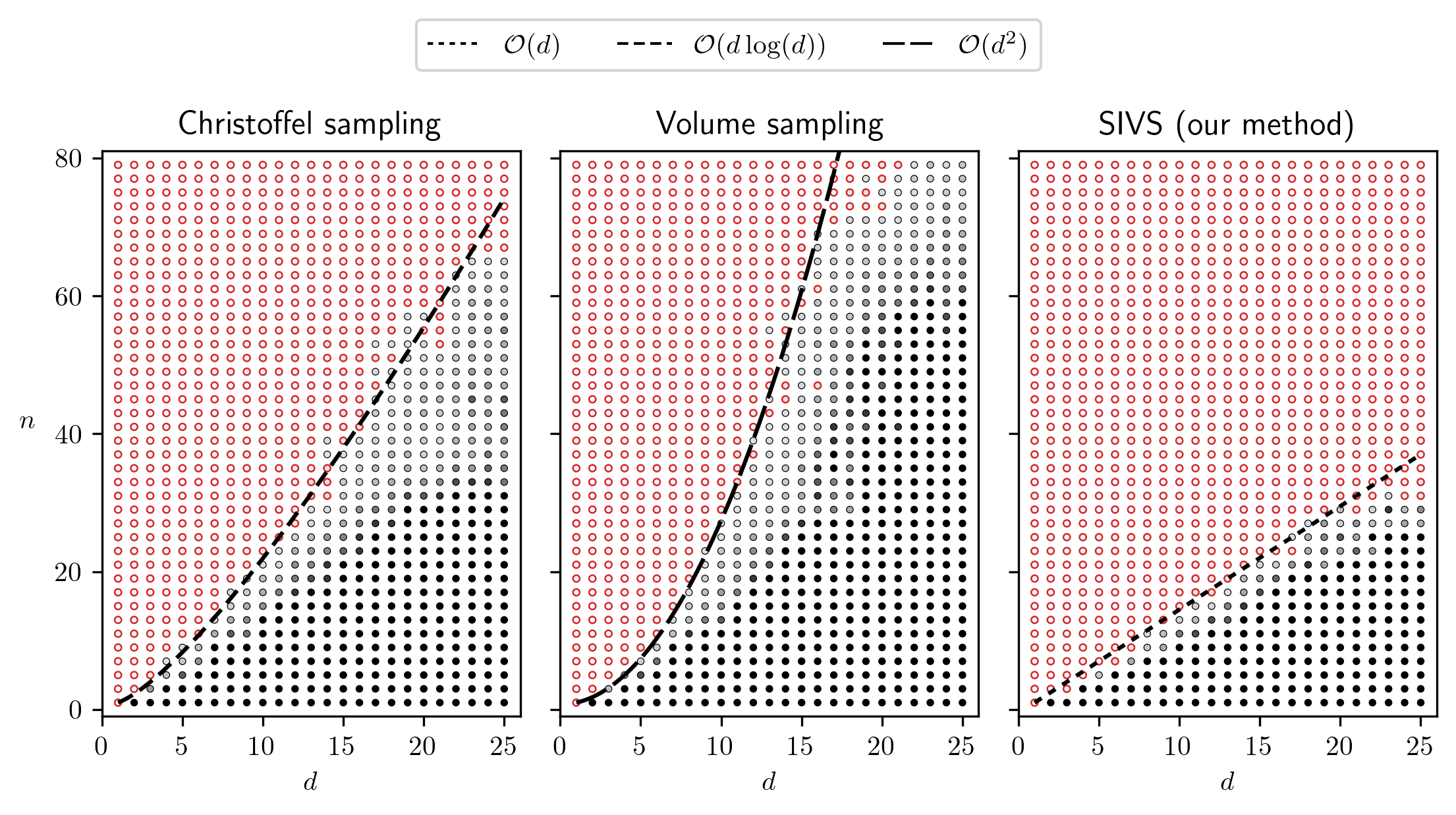}
    \caption{Phase diagrams for the probability $\mbb{P}\bracs*{\mu(\boldsymbol{x}) \le 2}$ with $\mcal{V} = H^1_0([-1,1], \tfrac12\dx)$ and where $\mcal{V}_d$ is spanned by polynomials.
    The probability is estimated using $200$ independent samples $\boldsymbol{x}\in\mcal{X}^n$ for different dimensions $d$ and sample sizes $n$.
    White marks a probability of $1$, black a probability of $0$.
    Points having $\mbb{P}\bracs*{\mu(\boldsymbol{x}) \le 2} \ge \tfrac12$ are marked with bold, red borders.
    The factor for the linear rate in the phase diagram for subspace-informed volume sampling is $1.5$.
    }
    \label{fig:mu_h10}
\end{figure}

\begin{figure}
    \centering
    \includegraphics[width=\textwidth]{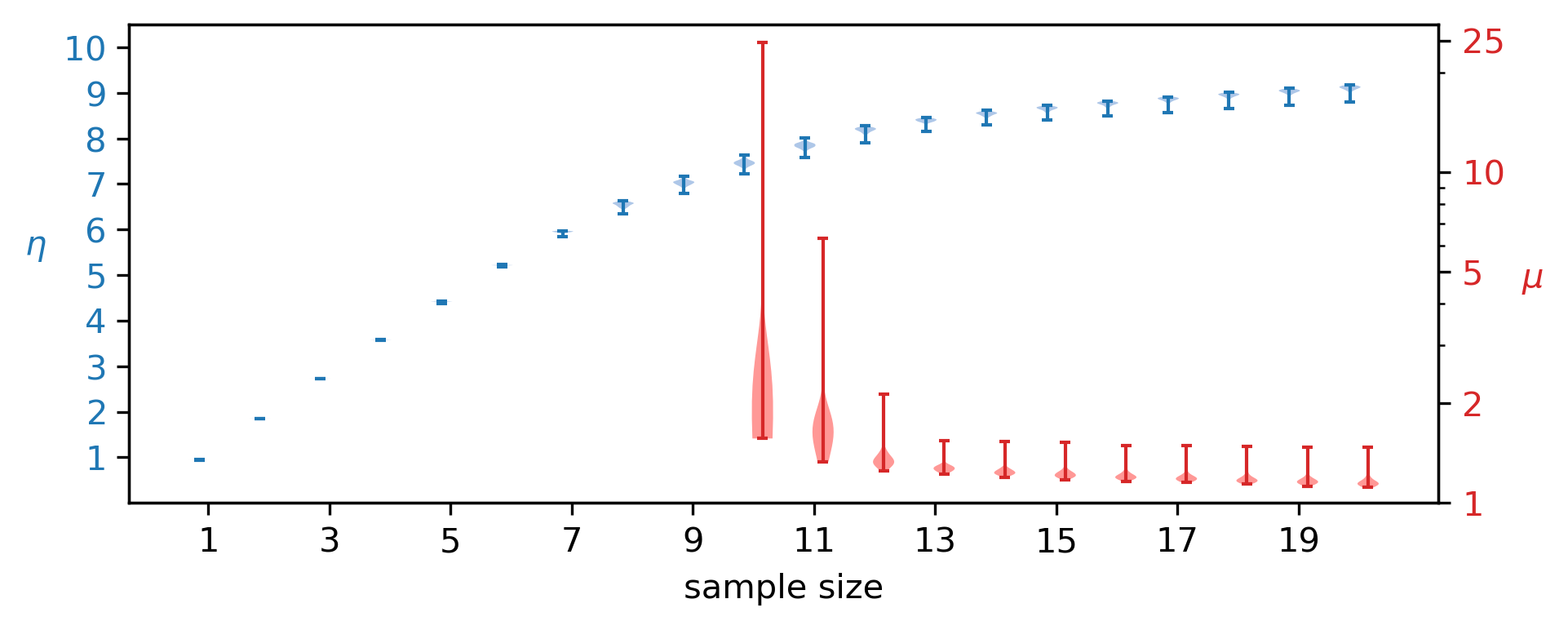}
    \caption{Violin plot of the submodular surrogate $\eta$ and the suboptimality constant $\mu$ for the first $20$ steps of the greedy optimisation procedure.
    The initial sample $\mcal{D}$ is of size $100$ and drawn using the $L^2$-Christoffel sampling method.
    The experiment was repeated $100$ times to compute the violins.
    $\mcal{V} = H^1_0([-1,1], \tfrac12\dx)$ and $d=10$ with $\mcal{V}_d$ being spanned by polynomials.}
    \label{fig:h10_greedy}
\end{figure}

\subsection{A polynomial subspace of \texorpdfstring{$\boldsymbol{H^1(\mbb{R}, \mcal{N}(0,1))}$}{H1 with Gaussian measure}}
\label{sec:experiments:gaussian}

Consider the Hilbert space $\mathcal{V} = H^1(\mbb{R}, \nu)$ with $\nu := \mcal{N}(0, 1)$.
Moreover, consider the $d$-dimensional polynomial subspaces $\mathcal{V}_d = \operatorname{span}\{1, x, \ldots, x^{d-1}\}$.
The reproducing kernel of $\mcal{V}$ is given by
$$
	k(x, y) := \sqrt{\tfrac{\pi}{2}}\exp\pars*{\tfrac{x^2 + y^2}{2}} \pars*{\operatorname{erf}\pars*{\tfrac{\min\braces{x, y}}{\sqrt{2}}} + 1} \pars*{1 - \operatorname{erf}\pars*{\tfrac{\max\braces{x, y}}{\sqrt{2}}}} .
$$

\begin{minipage}{\textwidth}
\begin{remark}
    As discussed in equation~\eqref{eq:finite_trace}, a sufficient condition for the finiteness of the measures $\mathfrak{K}\tfrac{K}{K_d}\nu$ is $\int K(x) \dx[\nu] < \infty$.
    This \textbf{sufficient} condition is satisfied for the choice $\nu = \mcal{N}(0, 1+\varepsilon)$ for any $\varepsilon > 0$, but it is not satisfied for the choice $\varepsilon = 0$.
    However, since this condition may not be \textbf{necessary}, we perform the experiments for this test case with the choice $\nu = \mcal{N}(0, 1)$.
    We observe that the resulting method indeed produces adequate results.
\end{remark}
\end{minipage}

Phase diagrams for the probability of $\mu(\boldsymbol{x}) \le 2$ are presented in Figure~\ref{fig:mu_h1gauss}, and the convergence of the proposed greedy subsampling method is illustrated in Figure~\ref{fig:wh1_greedy}.
As in both preceding cases, the greedy subsampling algorithm performs quite well, producing an acceptable quasi-optimality constant $\mu(\boldsymbol{x}) \le 3$ for the almost optimal sample size $n = d+1$ in all $100$ repetitions of the experiment.
However, in contrast to the preceding test cases, the phase transition boundary for continuous volume sampling is linear in this case.
This is explained theoretically by the fact that the Hermite polynomials form the spectral basis for $\mcal{V}$.
(See the discussion following Corollary~\ref{cor:ayoub}.)

\begin{figure}
    \centering
    \includegraphics[width=\textwidth]{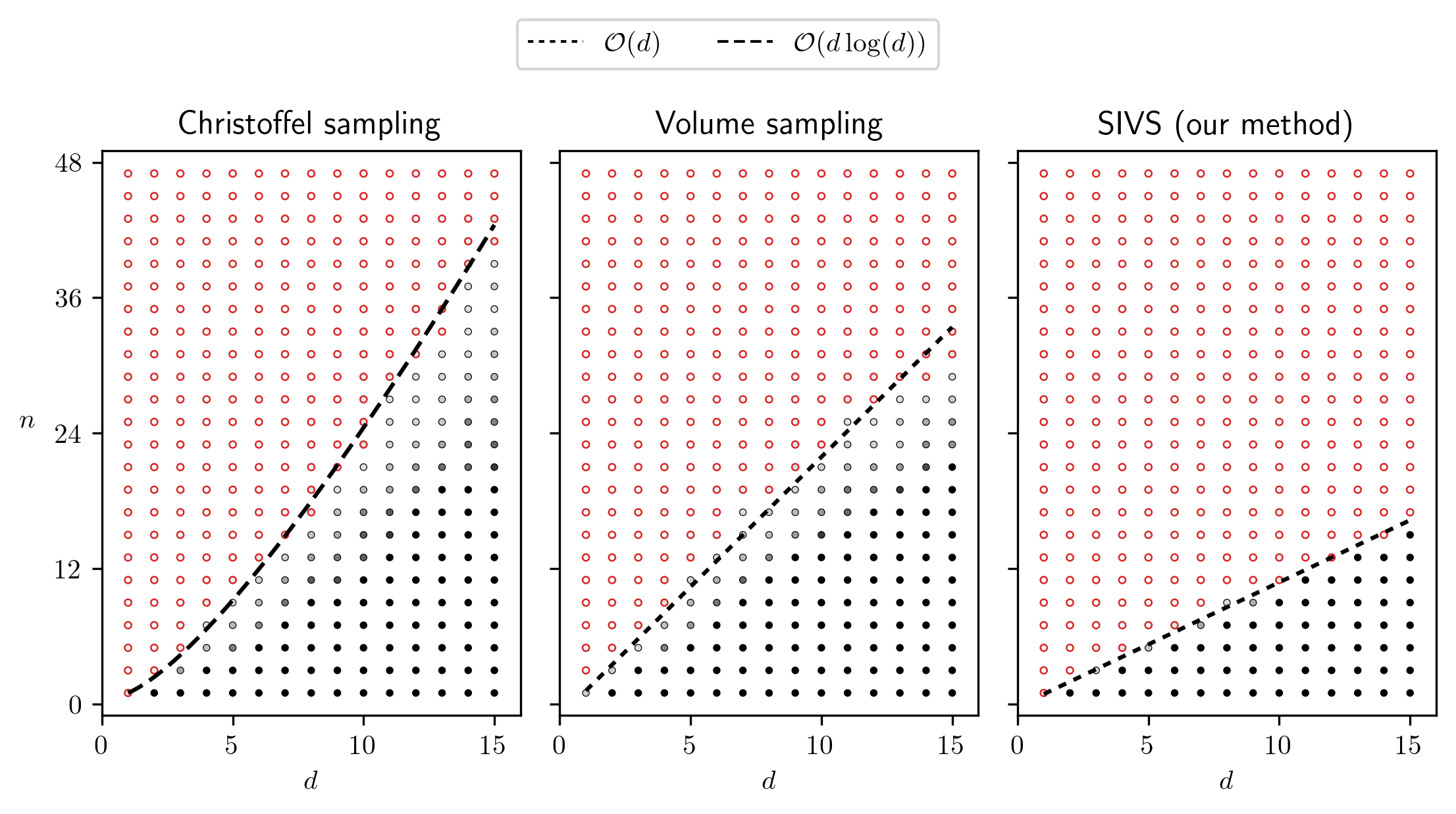}
    \caption{Phase diagrams for the probability $\mbb{P}\bracs*{\mu(\boldsymbol{x}) \le 2}$ with $\mcal{V} = H^1(\mbb{R}, \mcal{N}(0, 1))$ and where $\mcal{V}_d$ is spanned by polynomials.
    The probability is estimated using $200$ independent samples $\boldsymbol{x}\in\mcal{X}^n$ for different dimensions $d$ and sample sizes $n$.
    White marks a probability of $1$, black a probability of $0$.
    Points having $\mbb{P}\bracs*{\mu(\boldsymbol{x}) \le 2} \ge \tfrac12$ are marked with bold, red borders.
    The factor for the linear rate in the phase diagram for volume sampling is $2.1$.
    The factor for the linear rate in the phase diagram for subspace-informed volume sampling is $1.2$.
    }
    \label{fig:mu_h1gauss}
\end{figure}

\begin{figure}
    \centering
    \includegraphics[width=\textwidth]{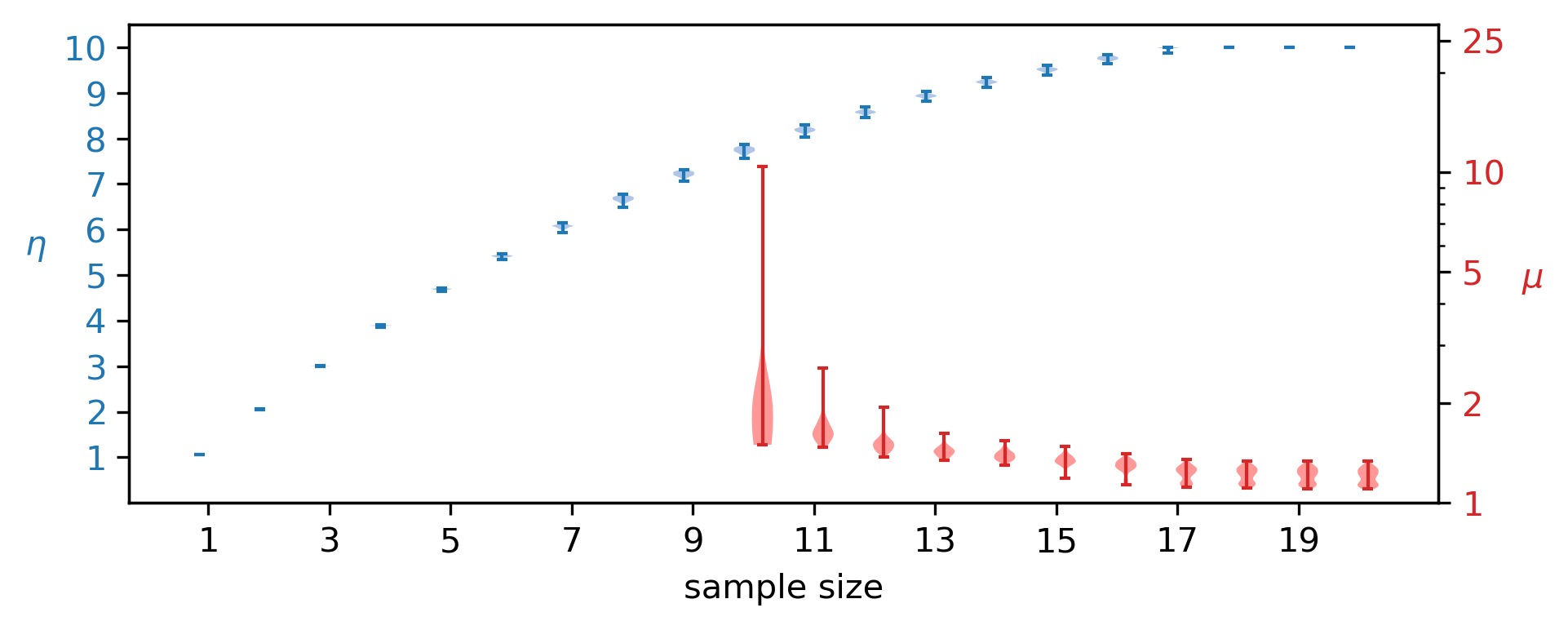}
    \caption{Violin plot of the submodular surrogate $\eta$ and the suboptimality constant $\mu$ for the first $20$ steps of the greedy optimisation procedure.
    The initial sample $\mcal{D}$ is of size $100$ and drawn using the $L^2$-Christoffel sampling method.
    The experiment was repeated $100$ times to compute the violins.
    $\mcal{V} = H^1(\mbb{R}, \mcal{N}(0, 1))$ and $d=10$ with $\mcal{V}_d$ being spanned by polynomials.}
    \label{fig:wh1_greedy}
\end{figure}

\section{Discussion}
\label{sec:discussion}

\paragraph{General}
We view this paper only as a proof of concept.
More experiments in higher dimensions and with higher-order RKHS, like unanchored Sobolev spaces~\cite{sarazin_2023} and RKHS with Gaussian kernels, are necessary to assess the applicability of the proposed methodology in practice.
This is non-trivial, since point evaluations are unstable without the proper (orthogonal) polynomial basis and kernel matrices become unstable when the points in the point set get too close.
A deeper investigation is required to see if these numerical instabilities can be overcome.

\paragraph{Theory}
While the proposition~\ref{cor:ayoub},~\ref{thm:iid_bounds} and~\ref{thm:expectation} provide an initial analysis, none of these statements explains the linear sample size bounds observed for SIVS in the experiments.
It would also be interesting to investigate why $L^2$-Christoffel sampling seems to provide a close-to-optimal $d\log(d)$ sample size bound (without influence of the embedding constant $\tilde{C}_d$) or if this observation is a mere artifact of the chosen experimental parameters.

\paragraph{Application}
This manuscript focuses on finding sample points adapted to the approximation problem at hand.
However, this presupposes the ability to generate new data, which is not possible in many classical approximation tasks, where we have to make the best of the data that is already given.
An interesting direction of research would, therefore, be to restrict the space $\mcal{V}_d$ to be adapted to the given data.
Since $\mu(\boldsymbol{x})$ depends on the smallest eigenvalue of $G^{\boldsymbol{x}}$, it seems natural to use a span of the eigenspaces of $G^{\boldsymbol{x}}\in\mbb{R}^{d\times d}$ as a suitable subspace of $\mcal{V}_d$.
This projection introduces an additional approximation error, but it is conceivable that the eigenvectors corresponding to the smallest eigenvalues have a high $\mcal{V}$-norm.
Projecting onto the complement of these vectors would, therefore, only result in a minor additive approximation error if the sought function is sufficiently regular, and the overall error may indeed decrease.

Many practical applications rely on data afflicted by general random noise (in contrast to the deterministic perturbations considered in Section~\ref{sec:noisy}).
Proposition~\ref{prop:error_bound_noisy} can be extended to this setting, which is shown, in a more general setting, in Theorem~2.3 in~\cite{cohen2022nonlinear}.
Note, however, that the approximation from Section~\ref{sec:noisy} depends on an appropriate choice of norm $\norm{\bullet}_{\mcal{R}}$ for measuring the noise.
Since this norm also influences the approximation $u^{\boldsymbol{x},\mcal{S}}$, acting as a regulariser, the optimal choice of this norm would be an interesting research topic.

The results in this paper depend on the choice of an appropriate RKHS and require explicit knowledge of the corresponding reproducing kernel.
Finding such an explicit expression may be difficult when the domain $\mcal{X}$ is nontrivial.
For certain RKHS of band-limited functions arguments relating to ``escaping the native space'' (cf.~section~2.3.3 in~\cite{kempf_thesis}) may yield error bounds even if the regularity of the sought function $u$ is overestimated.
But as of now, choosing the appropriate kernel remains a practical limitation.

Finally, note that   we assumed that a $\mathcal{V}$-orthonormal basis is known.
In practice, however, such a basis may be hard to compute, especially when the domain is non-trivial.
A discussion of this problem and related works in the $L^2$-setting can be found in~\cite{trunschke2024optimalsamplingsquaresapproximation}.

\section*{Acknowledgements}
This project is funded by the ANR-DFG project COFNET (ANR-21-CE46-0015).
This work was partially conducted within the France 2030 framework programme, Centre Henri Lebesgue ANR-11-LABX-0020-01.

Our code makes extensive use of the Python packages: \texttt{numpy}~\cite{numpy}, \texttt{scipy}~\cite{scipy}, and \texttt{matplotlib}~\cite{matplotlib}.

{
    \emergencystretch=3em
    \printbibliography
}

\clearpage

\appendix
\section{Proof of the error bound~\eqref{eq:PBDW_bound}}
\label{app:PBDW_bound}

To prove this error bound, we start by noting that $\norm{(I - P_{\mcal{V}_{\boldsymbol{x}}})v}_{\mcal{V}}^2 = \norm{v}_{\mcal{V}}^2 - \norm{P_{\mcal{V}_{\boldsymbol{x}}}v}_{\mcal{V}}^2$, and therefore
\begin{align}
	\mu(\boldsymbol{x})^{-2}
	&= \inf_{v \in \mcal{V}_d} \tfrac{\norm{P_{\mcal{V}_{\boldsymbol{x}}}v}_{\mcal{V}}^2}{\norm{v}_{\mcal{V}}^2} 
	= \inf_{v \in \mcal{V}_d} 1 - \tfrac{\norm{(I - P_{\mcal{V}_{\boldsymbol{x}}})v}_{\mcal{V}}^2}{\norm{v}_{\mcal{V}}^2} 
	= 1 - \sup_{v \in \mcal{V}_d} \tfrac{\norm{(I-P_{\mcal{V}_{\boldsymbol{x}}})v}_{\mcal{V}}^2}{\norm{v}_{\mcal{V}}^2} \\
	&= 1 - \Big(\sup_{v \in \mcal{V}_d} \sup_{w \in \mcal{V}_{\boldsymbol{x}}^\perp} \tfrac{\pars{v, w}_{\mcal{V}}}{\norm{v}_{\mcal{V}}\norm{w}_{\mcal{V}}}\Big)^2 
	= 1 - \sup_{w \in \mcal{V}_{\boldsymbol{x}}^\perp} \tfrac{\norm{P_{\mcal{V}_d}w}_{\mcal{V}}^2}{\norm{w}_{\mcal{V}}^2} \\
	&= \inf_{w \in \mcal{V}_{\boldsymbol{x}}^\perp} 1 - \tfrac{\norm{P_{\mcal{V}_d}w}_{\mcal{V}}^2}{\norm{w}_{\mcal{V}}^2} 
	= \inf_{w \in \mcal{V}_{\boldsymbol{x}}^\perp} \tfrac{\norm{(I-P_{\mcal{V}_d})w}_{\mcal{V}}^2}{\norm{w}_{\mcal{V}}^2}
	.
\end{align}
This implies
\begin{equation}
\label{eq:mu_identity}
	\mu(\boldsymbol{x})
	= \sup_{w \in \mcal{V}_{\boldsymbol{x}}^\perp} \frac{\norm{w}_{\mcal{V}}}{\norm{(I-P_{\mcal{V}_d})w}_{\mcal{V}}}
	.
\end{equation}
We can, therefore, conclude that
\begin{align}
	\norm{u - u^{d,\boldsymbol{x}}}_{\mcal{V}}
	&= \norm{u - (P_{\mcal{V}_d}^{\boldsymbol{x}} + P_{\mcal{V}_{\boldsymbol{x}}}(I - P_{\mcal{V}_d}^{\boldsymbol{x}}))u}_{\mcal{V}} \\
	&= \norm{(I -  P_{\mcal{V}_{\boldsymbol{x}}})(I - P_{\mcal{V}_d}^{\boldsymbol{x}})u}_{\mcal{V}} \\
	&\le \mu(\boldsymbol{x}) \norm{(I - P_{\mcal{V}_d}) (I -  P_{\mcal{V}_{\boldsymbol{x}}})(I - P_{\mcal{V}_d}^{\boldsymbol{x}})u}_{\mcal{V}} .
\end{align}
Next, let $\mathcal{W} := (P_{\mathcal{V}_{\boldsymbol{x}}}\!\mathcal{V}_d)$ and
note that $\mu(\boldsymbol{x}) < \infty$ ensures
    that $P_{\mcal{V}_{\boldsymbol{x}}}|_{\mathcal{V}_d} : \mcal{V}_d \to \mathcal{W}$ is invertible and
    \begin{align}
    	P_{\mcal{V}_d}^{\boldsymbol{x}} u
    	&:= \argmin_{v\in\mcal{V}_d} \norm{P_{\mcal{V}_{\boldsymbol{x}}}(u - v)}_{\mcal{V}} \\
    	&= (P_{\mcal{V}_{\boldsymbol{x}}}|_{\mathcal{V}_d})^{-1} \argmin_{v\in (P_{\mcal{V}_{\boldsymbol{x}}}\!\mcal{V}_d)} \norm{P_{\mcal{V}_{\boldsymbol{x}}}u - v}_{\mcal{V}} \\
    	&= (P_{\mcal{V}_{\boldsymbol{x}}}|_{\mathcal{V}_d})^{-1} P_{(P_{\mcal{V}_{\boldsymbol{x}}}\!\mcal{V}_d)} P_{\mcal{V}_{\boldsymbol{x}}}u \\
    	&= (P_{\mcal{V}_{\boldsymbol{x}}}|_{\mathcal{V}_d})^{-1} P_{(P_{\mcal{V}_{\boldsymbol{x}}}\!\mcal{V}_d)} u \\
    	&= (P_{\mcal{V}_{\boldsymbol{x}}}|_{\mathcal{V}_d})^{-1} P_{\mathcal{W}} u
    	,
    \end{align}
    where the second-to-last equality follows from $(P_{\mcal{V}_{\boldsymbol{x}}}\!\mcal{V}_d) \subseteq \mcal{V}_{\boldsymbol{x}}$. 
Consequently, $P_{\mathcal{V}_{\boldsymbol{x}}} P_{\mathcal{V}_{d}}^{\boldsymbol{x}} = P_{\mathcal{V}_{\boldsymbol{x}}}|_{\mathcal{V}_d} P_{\mathcal{V}_{d}}^{\boldsymbol{x}} = P_{\mathcal{W}}$ and
\begin{align}
    (I - P_{\mathcal{V}_d}) (I - P_{\mathcal{V}_{\boldsymbol{x}}}) (I - P_{\mathcal{V}_d}^{\boldsymbol{x}})
    &= (I - P_{\mathcal{V}_d}) (I - P_{\mathcal{V}_d}^{\boldsymbol{x}})
    - (I - P_{\mathcal{V}_d}) P_{\mathcal{V}_{\boldsymbol{x}}} (I - P_{\mathcal{V}_d}^{\boldsymbol{x}}) \\
    &= (I - P_{\mathcal{V}_d})
    - (I - P_{\mathcal{V}_d}) (P_{\mathcal{V}_{\boldsymbol{x}}} - P_{\mathcal{V}_{\boldsymbol{x}}}P_{\mathcal{V}_d}^{\boldsymbol{x}}) \\
    &= (I - P_{\mathcal{V}_d})
    - (I - P_{\mathcal{V}_d}) (P_{\mathcal{V}_{\boldsymbol{x}}} - P_{\mathcal{W}}) \\
    &= (I - P_{\mathcal{V}_d}) (I - P_{\mathcal{V}_{\boldsymbol{x}}} + P_{\mathcal{W}}) \\
    &= (I - P_{\mathcal{V}_d}) (P_{\mathcal{V}_{\boldsymbol{x}}^\perp} + P_{\mathcal{W}}) \\
    &= (I - P_{\mathcal{V}_d}) P_{\mathcal{V}_{\boldsymbol{x}}^\perp\oplus\mathcal{W}} ,
\end{align}
where the final equality follows because $\mathcal{W}\subseteq\mathcal{V}_{\boldsymbol{x}}$ implies $\mathcal{W} \perp \mathcal{V}_{\boldsymbol{x}}^\perp$.
Now, observe that
\begin{align}
    (\mathcal{V}_{\boldsymbol{x}}^\perp\oplus\mathcal{W})^\perp
    &= \mathcal{V}_{\boldsymbol{x}}\cap \mathcal{W}^\perp \\
    &= \{v\in\mathcal{V}_{\boldsymbol{x}} \;:\; (v, w)_{\mathcal{V}} = 0\text{ for all }  w\in\mathcal{W}\} \\
    &= \{v\in\mathcal{V}_{\boldsymbol{x}} \;:\; (v, P_{\mathcal{V}_{\boldsymbol{x}}}  w)_{\mathcal{V}} = 0\text{ for all } w\in\mathcal{V}_d\} \\
    &= \{v\in\mathcal{V}_{\boldsymbol{x}} \;:\; (v, w)_{\mathcal{V}} = 0\text{ for all }  w\in\mathcal{V}_d\} \\
    &= \mathcal{V}_{\boldsymbol{x}}\cap \mathcal{V}_{d}^\perp.
\end{align}
From this, we deduce that 
\begin{align}
    (I - P_{\mathcal{V}_d}) (I - P_{\mathcal{V}_{\boldsymbol{x}}}) (I - P_{\mathcal{V}_d}^{\boldsymbol{x}})
    &=(I - P_{\mathcal{V}_d}) (I - P_{\mathcal{V}_{\boldsymbol{x}}\cap \mathcal{V}_{d}^\perp}) \\
    &= I - P_{\mathcal{V}_d} - P_{\mathcal{V}_{\boldsymbol{x}}\cap \mathcal{V}_{d}^\perp} \\
    &= I - P_{\mathcal{V}_d \oplus (\mathcal{V}_{\boldsymbol{x}}\cap \mathcal{V}_{d}^\perp) },
\end{align}
which concludes the proof.

\section{Lemma~\ref*{lem:tau_bound}}
\label{app:tau_bound}

\begin{lemma}
\label{lem:tau_bound}
    Let $\boldsymbol{x}\in\mcal{X}^n$ and define the operator
    $$
        A = P_{\mcal{V}_{\boldsymbol{x}}} P_{\mcal{V}_d}^{\boldsymbol{x}} (I - P_{\mcal{V}_d}) .
    $$
    Then $\norm{A}_{\mcal{V}\to\mcal{V}} \le \norm{I - G^{\boldsymbol{x}}}_{\mathrm{Fro}} + \norm{I - G^{\boldsymbol{x}}}_{2}$.
\end{lemma}
\begin{proof}
    Observe that for every $u\in\mcal{V}$
    $$
        P_{\mcal{V}_d}^{\boldsymbol{x}} u
        := \argmin_{v\in\mcal{V}_d} \; \norm{u - v}_{\boldsymbol{x}}
        = \argmin_{v\in\mcal{V}_d} \; \norm{P_{\mcal{V}_{\boldsymbol{x}}} (P_{\mcal{V}_{\boldsymbol{x}}} u - v)}_{\mcal{V}}
        = P_{\mcal{V}_d}^{\boldsymbol{x}} P_{\mcal{V}_{\boldsymbol{x}}} u ,
    $$
    which allows us to write
    $$
        A
        = P_{\mcal{V}_{\boldsymbol{x}}} P_{\mcal{V}_d}^{\boldsymbol{x}} (I - P_{\mcal{V}_d})
        = P_{\mcal{V}_{\boldsymbol{x}}} (P_{\mcal{V}_d}^{\boldsymbol{x}} - P_{\mcal{V}_d})
        = \underbrace{P_{\mcal{V}_{\boldsymbol{x}}} (P_{\mcal{V}_d}^{\boldsymbol{x}} - P_{\mcal{V}_d}) P_{\mcal{V}_{\boldsymbol{x}}}}_{=: A_1}
        + \underbrace{P_{\mcal{V}_{\boldsymbol{x}}} P_{\mcal{V}_d} (P_{\mcal{V}_{\boldsymbol{x}}} - I)}_{=: A_2}
        .
    $$
    We can therefore use the triangle inequality $\norm{A}_{\mcal{V}\to\mcal{V}} \le \norm{A_1}_{\mcal{V}\to\mcal{V}} + \norm{A_2}_{\mcal{V}\to\mcal{V}}$ to split the computation into two steps.
    Both steps will rely heavily on the identities
    \begin{equation}
    \label{eq:projectors}
    \begin{aligned}
        P_{\mcal{V}_{\boldsymbol{x}}} k(\boldsymbol{x}, \bullet)
          &= k(\boldsymbol{x}, \bullet), 
        &
        P^{\boldsymbol{x}}_{\mcal{V}_d} b
          &= b, 
        &
        P_{\mcal{V}_d} b
          &= b, \\ 
        P_{\mcal{V}_{\boldsymbol{x}}} b
          &= b(\boldsymbol{x}) K(\boldsymbol{x})^{+} k(\boldsymbol{x}, \bullet),
        &
        P_{\mcal{V}_d}^{\boldsymbol{x}} k(\boldsymbol{x}, \bullet)
          &= b(\boldsymbol{x})^\intercal (G^{\boldsymbol{x}})^{+} b,
        &
        P_{\mcal{V}_d} k(\boldsymbol{x}, \bullet)
          &= b(\boldsymbol{x})^\intercal b
        ,
    \end{aligned}
    \end{equation}
    where the projectors are applied component-wise to the vectors of functions.
    The first line of the preceding identities follows by the projection properties and the second by simple computation.
    
    In the first step, we bound $\norm{A_1}_{\mcal{V}\to\mcal{V}}$.
    Using the identities~\eqref{eq:projectors} and the linearity of the projection operator, we compute
    $$
        P_{\mcal{V}_{\boldsymbol{x}}} P_{\mcal{V}_d}^{\boldsymbol{x}} P_{\mcal{V}_{\boldsymbol{x}}} k(\boldsymbol{x}, \bullet)
        = P_{\mcal{V}_{\boldsymbol{x}}} P_{\mcal{V}_d}^{\boldsymbol{x}} k(\boldsymbol{x}, \bullet)
        = P_{\mcal{V}_{\boldsymbol{x}}} b(\boldsymbol{x})^\intercal (G^{\boldsymbol{x}})^{+} b
        = \underbrace{b(\boldsymbol{x})^\intercal (G^{\boldsymbol{x}})^{+} b(\boldsymbol{x}) K(\boldsymbol{x})^{+}}_{=:D_1} k(\boldsymbol{x}, \bullet)
    $$
    and
    $$
        P_{\mcal{V}_{\boldsymbol{x}}} P_{\mcal{V}_d} P_{\mcal{V}_{\boldsymbol{x}}} k(\boldsymbol{x}, \bullet)
        = P_{\mcal{V}_{\boldsymbol{x}}} P_{\mcal{V}_d} k(\boldsymbol{x}, \bullet)
        = P_{\mcal{V}_{\boldsymbol{x}}} b(\boldsymbol{x})^\intercal b
        = \underbrace{b(\boldsymbol{x})^\intercal b(\boldsymbol{x}) K(\boldsymbol{x})^{+}}_{=:D_2} k(\boldsymbol{x}, \bullet) .
    $$
    Combining both equations yields $A_1 k(\boldsymbol{x}, \bullet) = (D_1 - D_2) k(\boldsymbol{x}, \bullet)$, which allows us to write
    \begin{align}
        \norm{A_1}_{\mcal{V}\to\mcal{V}}^2
        = \norm{A_1 P_{\mcal{V}_{\boldsymbol{x}}}}_{\mcal{V}\to\mcal{V}}^2
        &= \sup_{v\in\mcal{V}_{\boldsymbol{x}}} \frac{\norm{A_1v}_{\mcal{V}}^2}{\norm{v}_{\mcal{V}}^2}
        = \sup_{c\in\mbb{R}^n} \frac{c^\intercal (D_1 - D_2) K(\boldsymbol{x}) (D_1 - D_2)^\intercal c}{c^\intercal K(\boldsymbol{x}) c} \\
        &= \lambda_{\mathrm{max}}(K(\boldsymbol{x})^{+1/2} (D_1 - D_2) K(\boldsymbol{x}) (D_1 - D_2)^\intercal K(\boldsymbol{x})^{+1/2}) \\
        &\le \operatorname{tr}(K(\boldsymbol{x})^{+1/2} (D_1 - D_2) K(\boldsymbol{x}) (D_1 - D_2)^\intercal K(\boldsymbol{x})^{+1/2})
        ,
    \end{align}
    where we use the notation $K(\boldsymbol{x})^{+1/2} := (K(\boldsymbol{x})^+)^{1/2}$.
    The final matrix-algebraic expression can be computed explicitly.
    To do this, we will simplify the notation and write $B := b(\boldsymbol{x})$, $K := K(\boldsymbol{x})$ and $G := G^{\boldsymbol{x}}$.
    We will also make extensive use of the symmetry of $K$ and $G$ and the identity $G = BK^{+}B^\intercal$.
    Noting that $D_1 - D_2 = B^\intercal(G^{+} - I)BK^{+}$, we can now use the invariance of the trace under cyclic permutations to obtain
    \begin{align}
        \operatorname{tr}(K^{+1/2} (D_1 - D_2) K (D_1 - D_2)^\intercal K^{+1/2})
        &= \operatorname{tr}(K^{+1/2} B^\intercal(G^{+} - I)BK^{+} K K^{+}B^\intercal(G^{+} - I)B K^{+1/2}) \\
        &= \operatorname{tr}((G^{+} - I)BK^{+} B^\intercal(G^{+} - I)B K^{+} B^\intercal) \\
        &= \operatorname{tr}((G^{+} - I) G (G^{+} - I)G) \\
        &= \operatorname{tr}((I - G)^2) \\
        &= \norm{I - G}_{\mathrm{Fro}}^2
        .
    \end{align}
    This bounds $\norm{A_1}_{\mcal{V}\to\mcal{V}} \le \norm{I - G^{\boldsymbol{x}}}_{\mathrm{Fro}}$.
    
    In the second step, we bound $\norm{A_2}_{\mcal{V}\to\mcal{V}}$.
    For this, we define the space $\mcal{W} := \mcal{V}_{\boldsymbol{x}} + \mcal{V}_d$ and note that $\operatorname{ker}(A_2) \supseteq \operatorname{ker}(P_{\mcal{V}_{\boldsymbol{x}}} - P_{\mcal{V}_d})\supseteq \mcal{W}^\perp$.
    Denoting the $\mcal{V}$-orthogonal complement of $\mcal{V}_d$ in $\mcal{W}$ by $\mcal{W}/\mcal{V}_d \subseteq \mcal{V}_{\boldsymbol{x}}$, we can thus write
    \begin{align}
        \norm{A_2}_{\mcal{V}\to\mcal{V}}
        &= \norm{A_2 P_{\mcal{W}}}_{\mcal{V}\to\mcal{V}} \\
        &= \norm{A_2 P_{\mcal{V}_d} + A_2 P_{\mcal{W}/\mcal{V}_d} }_{\mcal{V}\to\mcal{V}} \\
        &= \norm{A_2 P_{\mcal{V}_d} + A_2 P_{\mcal{V}_{\boldsymbol{x}}} P_{\mcal{W}/\mcal{V}_d} }_{\mcal{V}\to\mcal{V}} \\ 
        &= \norm{A_2 P_{\mcal{V}_d}}_{\mcal{V}\to\mcal{V}} ,
    \end{align}
    where we have used $A_2 P_{\mcal{V}_{\boldsymbol{x}}} = P_{\mcal{V}_{\boldsymbol{x}}} P_{\mcal{V}_d} (P_{\mcal{V}_{\boldsymbol{x}}} - I) P_{\mcal{V}_{\boldsymbol{x}}} = 0$.
    To compute the norm $\norm{A_2 P_{\mcal{V}_d}}_{\mcal{V}\to\mcal{V}}$, we utilise again the identities~\eqref{eq:projectors} to write
    $$
         P_{\mcal{V}_{\boldsymbol{x}}} P_{\mcal{V}_d} P_{\mcal{V}_{\boldsymbol{x}}} b
         = P_{\mcal{V}_{\boldsymbol{x}}} P_{\mcal{V}_d} b(\boldsymbol{x}) K(\boldsymbol{x})^{+} k(\boldsymbol{x}, \bullet)
         = P_{\mcal{V}_{\boldsymbol{x}}} b(\boldsymbol{x}) K(\boldsymbol{x})^{+} b(\boldsymbol{x})^\intercal b
         = \underbrace{b(\boldsymbol{x}) K(\boldsymbol{x})^{+} b(\boldsymbol{x})^\intercal b(\boldsymbol{x}) K(\boldsymbol{x})^{+}}_{=: E_1} k(\boldsymbol{x}, \bullet)
    $$
    and
    $$
         P_{\mcal{V}_{\boldsymbol{x}}} P_{\mcal{V}_d} b
         = P_{\mcal{V}_{\boldsymbol{x}}} b
         = \underbrace{b(\boldsymbol{x}) K(\boldsymbol{x})^{+}}_{=: E_2} k(\boldsymbol{x}, \bullet) .
    $$
    Combining both equations yields $A_2 b = (E_1 - E_2) k(\boldsymbol{x}, \bullet)$, which allows us to write
    \begin{align}
        \norm{A_2}_{\mcal{V}\to\mcal{V}}^2
        = \norm{A_2 P_{\mcal{V}_d}}_{\mcal{V}\to\mcal{V}}^2
        &= \sup_{v\in\mcal{V}_{d}} \frac{\norm{A_2v}_{\mcal{V}}^2}{\norm{v}_{\mcal{V}}^2}
        = \sup_{c\in\mbb{R}^d} \frac{c^\intercal (E_1 - E_2) K(\boldsymbol{x}) (E_1 - E_2)^\intercal c}{c^\intercal c} \\
        &= \lambda_{\mathrm{max}}((E_1 - E_2) K(\boldsymbol{x}) (E_1 - E_2)^\intercal)
        .
    \end{align}
    Noting that $E_1 - E_2 = BK^{+} (B^\intercal BK^{+} - I)$, we can write
    \begin{align}
        \lambda_{\mathrm{max}}((E_1 - E_2) K (E_1 - E_2)^\intercal)
        &= \lambda_{\mathrm{max}}(BK^{+} (B^\intercal BK^{+} - I) K (K^{+} B^\intercal B - I) K^{+}B^\intercal) \\
        &= \lambda_{\mathrm{max}}((G BK^{+} - BK^{+}) K (K^{+} B^\intercal G - K^{+}B^\intercal)) \\
        &= \lambda_{\mathrm{max}}((G - I) BK^{+} K K^{+} B^\intercal (G - I)) \\
        &= \lambda_{\mathrm{max}}((G - I) G (G - I)) \\
        &= \norm{(G - I) G (G - I)}_{2} \\
        &\le \norm{G}_{2} \norm{G - I}_{2}^2 \\
        &\le \norm{G - I}_{2}^2
        ,
    \end{align}
    where the last inequality follows from Lemma~\ref{lem:gramian_ordering}.
    This bounds $\norm{A_2}_{\mcal{V}\to\mcal{V}} \le \norm{I - G^{\boldsymbol{x}}}_{2}$.

    Combining both estimates yields the claimed bound $\norm{A}_{\mcal{V}\to\mcal{V}} \le \norm{A_1}_{\mcal{V}\to\mcal{V}} + \norm{A_2}_{\mcal{V}\to\mcal{V}} \le \norm{I - G^{\boldsymbol{x}}}_{\mathrm{Fro}} + \norm{I - G^{\boldsymbol{x}}}_{2}$.
\end{proof}

\section{Proof of Proposition~\ref{prop:not_submodular}}
\label{app:proof:prop:not_submodular}

Let $\boldsymbol{x}' \subseteq \boldsymbol{x}$ and $y$ be fixed and observe that $\lambda(\boldsymbol{x}) = \mu(\boldsymbol{x})^{-2} = \min_{v\in S(\mcal{V}_d)} \norm{P_{\mcal{V}_{\boldsymbol{x}}} v}_{\mcal{V}}^2$.
This ensures that $\lambda$ is monotone because $\boldsymbol{x}' \subseteq \boldsymbol{x}$ implies
$\|P_{\mcal{V}_{\boldsymbol{x}'}}v\|_{\mathcal{V}} \le \|P_{\mcal{V}_{\boldsymbol{x}}}v\|_{\mathcal{V}}$ for every $v\in\mathcal{V}_d$.
To show that it is not submodular, we have to construct a counter-example to the inequality
\begin{equation}
\label{eq:exact_submodularity}
    \lambda(\boldsymbol{x} \oplus y) - \lambda(\boldsymbol{x})
    \le \lambda(\boldsymbol{x}' \oplus y) - \lambda(\boldsymbol{x}') .
\end{equation}
We do this by defining $\mcal{V}_d := \mcal{V}_{\boldsymbol{x}'} + \inner{v}$ for some $v\in S(\mcal{V})$, which is chosen later.
This simplifies
$$
    \lambda(\boldsymbol{z})
    = \min_{w\in \mcal{V}_{\boldsymbol{x}'}, w\perp v} \frac{\norm{w}_{\mcal{V}}^2 + \norm{P_{\mcal{V}_{\boldsymbol{z}}} v}_{\mcal{V}}^2}{\norm{w}_{\mcal{V}}^2 + \norm{v}_{\mcal{V}}^2}
    = \norm{P_{\mcal{V}_{\boldsymbol{z}}} v}_{\mcal{V}}^2
$$
for any $\boldsymbol{z}$ with $\boldsymbol{x}' \subseteq \boldsymbol{z}$.
Now, we decompose $k(y, \bullet) = \omega_1 + \omega_2 + \omega_3$ with
\begin{align}
    \omega_1 := P_{\mcal{V}_{\boldsymbol{x}'}} k(y,\bullet),
    \qquad
    \omega_2 := (P_{\mcal{V}_{\boldsymbol{x}}} - P_{\mcal{V}_{\boldsymbol{x}'}}) k(y,\bullet)
    \qquad\text{and}\qquad
    \omega_3 := (I - P_{\mcal{V}_{\boldsymbol{x}}}) k(y,\bullet) .
\end{align}
This allows us to write
\begin{align}
    \lambda(\boldsymbol{x}\oplus y)
    &= \norm{P_{\mcal{V}_{\boldsymbol{x}}} v + P_{\inner{\omega_3}}v}_{\mcal{V}}^2
    = \lambda(\boldsymbol{x}) + \frac{(\omega_3, v)_{\mcal{V}}^2}{\norm{\omega_3}_{\mcal{V}}^2} \\
    \lambda(\boldsymbol{x}'\oplus y),
    &= \norm{P_{\mcal{V}_{\boldsymbol{x}'}} v + P_{\inner{\omega_2 + \omega_3}}v}_{\mcal{V}}^2
    = \lambda(\boldsymbol{x}') + \frac{(\omega_2 + \omega_3, v)_{\mcal{V}}^2}{\norm{\omega_2 + \omega_3}_{\mcal{V}}^2},
\end{align}
and implies that equation~\eqref{eq:exact_submodularity} becomes
\begin{equation}
\label{eq:submodular_counter-example}
    \frac{(\omega_3, v)_{\mcal{V}}^2}{\norm{\omega_3}_{\mcal{V}}^2}
    \le \frac{(\omega_2 + \omega_3, v)_{\mcal{V}}^2}{\norm{\omega_2 + \omega_3}_{\mcal{V}}^2} .
\end{equation}
A counter-example is given by the choice $v = \omega_3$.

\section{The Schur complement}
\label{app:schur}

Consider a linearly independent set of vectors $\omega_1,\ldots,\omega_n$ with Gramian matrix $C$.
Moreover, define for $L \dot\cup S = \braces{1,\ldots,n}$ the notations $\omega_S$ and $P_{\inner{\omega_L}}\omega_S$ just as in Proposition~\ref{prop:approx_submodular}.
We can then ask ourselves what the Gramian of $(I - P_{\inner{\omega_L}}) \omega_S$ looks like.
Some algebra reveals that this Gramian is given by the Schur complement $C/C_{L,L} := C_{S,S} - C_{S,L} C_{L,L}^{-1} C_{L,S}$.
To prove $\lambda_{\mathrm{min}}(C/C_{L,L}) \ge \lambda_{\mathrm{min}}(C)$, consider the quadratic function
$$
    v
    \mapsto v^\intercal C v
    = \begin{bmatrix}
        v_S & v_L
    \end{bmatrix}^\intercal
    \begin{bmatrix}
        C_{S,S} & C_{S,L} \\
        C_{L,S} & C_{L,L}
    \end{bmatrix}
    \begin{bmatrix}
        v_S \\
        v_L
    \end{bmatrix}
$$
for fixed $v_S$ and optimise it over $v_L$.
The minimiser of this problem is given by $v_L = - C_{L,L}^{-1}C_{L,S}v_S$ and the minimum is
$$
    v_S^\intercal (C_{S,S} - C_{S,L}C_{L,L}^{-1}C_{L,S}) v_S
    = v_S^\intercal (C / C_{L,L}) v_S .
$$
This implies
$$
    \lambda_{\mathrm{min}}(C/C_{L,L})
    = \min_{v_S\in\mbb{R}^{\abs{S}}} \frac{v_S^\intercal (C / C_{L,L}) v_S}{v_S^\intercal v_S}
    = \min_{v_S\in\mbb{R}^{\abs{S}}} \frac{\min_{v_L} v^\intercal C v}{v_S^\intercal v_S}
    = \min_{v\in\mbb{R}^{n}} \frac{v^\intercal C v}{v_S^\intercal v_S}
    \ge \min_{v\in\mbb{R}^{n}} \frac{v^\intercal C v}{v^\intercal v}
    = \lambda_{\mathrm{min}}(C)
    .
$$
\section{The kernel of \texorpdfstring{$\boldsymbol{H^1(\rho)}$}{H1 with Gaussian measure}}
\label{app:gaussian_rkhs}

In this section, we derive an explicit expression for the kernel $k_x := k(x, \bullet)$ of the RKHS $H^1(\rho)$ with the Gaussian measure $\rho$.
By an abuse of notation, we also denote by $\rho$ the density of the Gaussian measure with respect to the Lebesgue measure.
Performing integration by parts on the intervals $(-\infty, x]$ and $[x,\infty)$ for the Riesz representation equation $(\phi, k_x)_{H^1(\rho)} = \phi(x)$ with smooth test functions $\phi$ yields the condition
\begin{equation}
\label{eq:kernel_weak_form}    
    -(\phi, k_x'')_{L^2(\rho)} + (\phi, pk_x')_{L^2(\rho)} + (\phi, k_x)_{L^2(\rho)} + [\phi k_x'\rho]_{-\infty}^x + [\phi k_x'\rho]_x^{\infty} = \phi(x) ,
\end{equation}
with the function $p:\mbb{R}\to\mbb{R}$ defined by $p(y) := y$.
Now define $k_{\mathrm{L},x} := k_x\vert_{(-\infty, x]}$ and $k_{\mathrm{R},x} := k_x\vert_{[x, \infty)}$ and observe that equation~\eqref{eq:kernel_weak_form} provides the following six conditions on $k_x$.
\begin{enumerate}
    \item By considering the $L^2$-term in the variational equation for test functions $\phi$ that are compactly supported in $(-\infty, x)$, we see that
    \begin{equation}
    \label{eq:kL_ode}
        -k_{\mathrm{L},x}''(y) + y k_{\mathrm{L},x}'(y) + k_{\mathrm{L},x}(y) = 0
        \quad\text{on}\quad (-\infty, x) .
    \end{equation}
    \item By considering the $L^2$-term in the variational equation for test functions $\phi$ that are compactly supported in $(x, \infty)$, we see that
    \begin{equation}
    \label{eq:kR_ode}
        -k_{\mathrm{R},x}''(y) + y k_{\mathrm{R},x}'(y) + k_{\mathrm{R},x}(y) = 0
        \quad\text{on}\quad (x, \infty) .
    \end{equation}
    \item By considering the boundary terms that contain $x$, we obtain $(k_{\mathrm{L}, x}'(x) - k_{\mathrm{R},x}'(x))\rho(x) = 1$.
    \item By considering the boundary term containing the limit $y\to-\infty$, we conclude that
    \begin{equation}
    \label{eq:kL_limit}
        \lim_{y\to-\infty} k_{\mathrm{L},x}'(y)\rho(y)=0 ,
    \end{equation}
    because $\phi$ can approach arbitrary values.
    \item For the same reason, we obtain the condition
    \begin{equation}
    \label{eq:kR_limit}
        \lim_{y\to\infty} k_{\mathrm{R},x}'(y)\rho(y)=0 .
    \end{equation}
    \item Since $k_x\in H^1(\rho)$ must be continuous, it must hold that $k_{\mathrm{L},x}(x) = k_{\mathrm{R},x}(x)$.
\end{enumerate}
The differential equations~\eqref{eq:kL_ode} and~\eqref{eq:kR_ode} are Sturm--Liouville equations and are solved by the functions
$$
    k_{\mathrm{L/R}, x}(y) = c_{\mathrm{L/R}, 1}\exp(\tfrac{y^2}{2})\pars*{\operatorname{erf}(\tfrac{y}{\sqrt{2}}) + c_{\mathrm{L/R}, 2}} .
$$
This means that $k_{\mathrm{L/R}, x}'(y) = yk_x(y) + \sqrt{\tfrac{2}{\pi}}c_{\mathrm{L/R}, 1}$ and the limit condition~\eqref{eq:kL_limit} is thus equivalent to
$$
    \lim_{y\to-\infty} y\pars*{\operatorname{erf}(\tfrac{y}{\sqrt{2}}) + c_{\mathrm{L},2}} = 0
    \qquad\Leftrightarrow\qquad
    \lim_{y\to-\infty} \operatorname{erf}(\tfrac{y}{\sqrt{2}}) + c_{\mathrm{L},2} = 0
    \qquad\Leftrightarrow\qquad
    c_{\mathrm{L},2}=1 .
$$
Analogously, we obtain
$$
    c_{\mathrm{R},2}=-1
$$
and, since the continuity condition $k_{\mathrm{L},x}(x) = k_{\mathrm{R},x}(x)$ has to be satisfied for all $x\in\mbb{R}$, we find that
\begin{align}
	&\ c_{\mathrm{L}, 1}\exp(\tfrac{x^2}{2})\pars*{\operatorname{erf}(\tfrac{x}{\sqrt{2}}) + 1} = c_{\mathrm{R}, 1}\exp(\tfrac{x^2}{2})\pars*{\operatorname{erf}(\tfrac{x}{\sqrt{2}}) - 1}\\
	\Leftrightarrow
	&\ c_{\mathrm{L}, 1}\pars*{\operatorname{erf}(\tfrac{x}{\sqrt{2}}) + 1} = c_{\mathrm{R}, 1}\pars*{\operatorname{erf}(\tfrac{x}{\sqrt{2}}) - 1} .
\end{align}
We can thus write $c_{\mathrm{L},1} := c\pars*{\operatorname{erf}(\tfrac{x}{\sqrt{2}}) - 1}$ and $c_{\mathrm{R},1} := c\pars*{\operatorname{erf}(\tfrac{x}{\sqrt{2}}) + 1}$ for some constant $c\in\mbb{R}$.
This yields the equations
$$
\begin{aligned}
	k_{\mathrm{L},x}(y)
	&= c\exp(\tfrac{y^2}{2})\pars*{\operatorname{erf}(\tfrac{y}{\sqrt{2}}) + 1}\pars*{\operatorname{erf}(\tfrac{x}{\sqrt{2}}) - 1} \\
	k_{\mathrm{R},x}(y)
	&= c\exp(\tfrac{y^2}{2})\pars*{\operatorname{erf}(\tfrac{x}{\sqrt{2}}) + 1}\pars*{\operatorname{erf}(\tfrac{y}{\sqrt{2}}) - 1} .
\end{aligned}
$$
The constant $c$ is determined by the boundary condition $(k_{\mathrm{L}, x}'(x) - k_{\mathrm{R},x}'(x))\rho(x) = 1$
and evaluates to $c = -\tfrac1{2\rho(x)}$.
This yields the final formula
$$
	k_x(y) = \sqrt{\tfrac{\pi}{2}}\exp\pars*{\tfrac{x^2 + y^2}{2}} \pars*{1 + \operatorname{erf}\pars*{\tfrac{\min\braces{x, y}}{\sqrt{2}}}} \pars*{1 - \operatorname{erf}\pars*{\tfrac{\max\braces{x, y}}{\sqrt{2}}}} .
$$

\end{document}